\documentclass[reqno, xcolor=pdftex,letterpaper]{amsart}

\usepackage{csquotes}
\usepackage[multiple]{footmisc}

\usepackage[english]{babel}
\usepackage{amsmath,amsthm,amssymb,amsfonts}
\usepackage{mathrsfs}
\usepackage[utf8x]{inputenc}
\usepackage[normalem]{ulem}
\usepackage{cancel}

\usepackage{exscale}
\usepackage{amsbsy}
\usepackage[colorlinks=true,citecolor=red,pagebackref,hypertexnames=false]{hyperref}

\usepackage{bbm} 
\usepackage{csquotes}

\usepackage{enumitem}
\setlist[enumerate]{leftmargin=*}

\usepackage[utf8x]{inputenc}

\newcommand{\e}{\mathrm{e}}

\newcommand{\rad}{\mathrm{rad}}
\newcommand{\diff}{{d}} 
\usepackage[non-compressed-cites,initials,nobysame]{amsrefs}

\renewcommand{\approx}{ \asymp}

\DeclareMathOperator{\supp}{supp}

\DeclareFontFamily{U}{mathx}{\hyphenchar\font45}
\DeclareFontShape{U}{mathx}{m}{n}{
	<5> <6> <7> <8> <9> <10>
	<10.95> <12> <14.4> <17.28> <20.74> <24.88>
	mathx10
}{}
\DeclareSymbolFont{mathx}{U}{mathx}{m}{n}
\DeclareFontSubstitution{U}{mathx}{m}{n}
\DeclareMathAccent{\widecheck}{0}{mathx}{"71}
\DeclareMathAccent{\wideparen}{0}{mathx}{"75}

\makeatletter
\newcommand{\leqnomode}{\tagsleft@true}
\newcommand{\reqnomode}{\tagsleft@false}
\makeatother

\numberwithin{equation}{section}

\usepackage[usenames,dvipsnames]{color}

\renewcommand{\leq}{\leqslant}
\renewcommand{\geq}{\geqslant}

\newcommand{\dd}{{d}}

\newcommand{\R}{\mathbb{R}}
\newcommand{\N}{\mathbb{N}}

\theoremstyle{theorem}
\newtheorem{theorem}{\sc \textbf{Theorem}}[section]  
\newtheorem{proposition}[theorem]{\sc \textbf{Proposition}}   
\newtheorem{corollary}[theorem]{\sc \textbf{Corollary}}        
\newtheorem{lemma}[theorem]{\sc \textbf{Lemma}}

\renewcommand{\approx}{ \asymp}

\theoremstyle{plain}
\newcounter{thm}

\newtheorem{main_theorem}[thm]{Theorem}

\theoremstyle{remark}
\newtheorem{definition}[theorem]{\sc \textbf{Definition}}

\newtheorem{remark}[theorem]{\sc \textbf{Remark}}

\usepackage[T1]{fontenc}
\DeclareFontFamily{T1}{calligra}{}
\DeclareFontShape{T1}{calligra}{m}{n}{<->s*[1.44]callig15}{}
\DeclareMathAlphabet\mathcalligra   {T1}{calligra} {m} {n}
\DeclareMathAlphabet\mathzapf       {T1}{pzc} {mb} {it}
\DeclareMathAlphabet\mathchorus     {T1}{qzc} {m} {n}
\DeclareMathAlphabet\mathrsfso      {U}{rsfso}{m}{n}

\makeatletter
\newcommand{\myitem}[1]{%
	\item[#1]\protected@edef\@currentlabel{#1}%
}
\makeatother

\begin{document}

	\title[Blow-up exponents and a semilinear fractional equation on $\mathbb{H}^{n}$]{Blow-up exponents and a semilinear elliptic equation for the fractional Laplacian on hyperbolic spaces}

	\author[T.\ Bruno]{Tommaso Bruno}
	\address{Dipartimento di Matematica, Universit\`a degli Studi di Genova\\ Via Dodecaneso 35, 16146 Genova, Italy}
	\email{tommaso.bruno@unige.it}
	
	\author[E.\ Papageorgiou]{Effie Papageorgiou}
	\address{Institut f{\"u}r Mathematik, Universit\"at Paderborn, Warburger Str. 100, D-33098
		Paderborn, Germany}
	\email{papageoeffie@gmail.com}

	\thanks{\emph{Math Subject Classification} 35J60 (primary) 26A33, 43A85, 22E30 (secondary)}
\thanks{T.\ B.\  was partially supported by the 2025 INdAM--GNAMPA grant {\em Function Spaces and Applications} (CUP\_E53C22001930001). E.\ P.\ acknowledges support by the Deutsche Forschungsgemeinschaft (DFG, German Research Foundation)--SFB-Gesch{\"a}ftszeichen--Projektnummer SFB-TRR 358/1 2023 --491392403.}

	\begin{abstract}
Let $\mathbb{H}^n$ be the $n$-dimensional real hyperbolic space, $\Delta$ its nonnegative Laplace--Beltrami operator whose bottom of the spectrum we denote by $\lambda_{0}$, and $\sigma \in (0,1)$. 

The aim of this paper is twofold. On the one hand, we determine the Fujita exponent for the fractional heat equation
\[
			\partial_t u + \Delta^{\sigma}u = \e^{\beta t}|u|^{\gamma-1}u,
\]
by proving that nontrivial positive global solutions exist if and only if $\gamma\geq  1 + \beta/ \lambda_{0}^{\sigma}$. On the other hand, we prove the existence of non-negative, bounded and finite energy solutions of the semilinear fractional elliptic equation
\[
			 \Delta^{\sigma} v - \lambda^{\sigma} v - v^{\gamma}=0
\]
for $0\leq \lambda \leq \lambda_{0}$ and $1<\gamma< \frac{n+2\sigma}{n-2\sigma}$. The two problems are known to be connected and the latter, aside from its independent interest, is actually instrumental to the former.

\smallskip

At the core of our results stands a novel fractional Poincar\'e-type inequality expressed in terms of a new scale of $L^{2}$ fractional Sobolev spaces, which sharpens those known so far, and which holds more generally on Riemannian symmetric spaces of non-compact type. We also establish an associated Rellich--Kondrachov-like compact embedding theorem for radial functions, along with other related properties.
\end{abstract}
	
		\maketitle
	
	\addtocontents{toc}{\setcounter{tocdepth}{1}}
	\tableofcontents

	\section{Introduction}

The global solvability of nonlinear evolution problems, such as 
\begin{equation}\label{eq: heatRn}
\partial_tu+\Delta u=|u|^{\gamma-1}u, \quad u(0,\cdot)=u_{0}\geq 0,   \quad \text{in }  (0,\tau)\times \mathbb{R}^n,
\end{equation} 
where $\Delta$ is the nonnegative Laplacian, occupies a special place in the theory of nonlinear partial differential equations. Assuming that $u_0\in L^{\infty}(\mathbb{R}^n)$, the problem has a long history which goes back to Fujita~\cite{Fuj} (see also~\cite{Hay,KST}), who showed that for~\eqref{eq: heatRn} the following dichotomy holds:
\begin{itemize}
	\item if $1 < \gamma \leq  \gamma^{*} := 1+ \frac2n$, then~\eqref{eq: heatRn} does not possess nontrivial global solutions;
	\item if $\gamma > \gamma^{*}$, then solutions corresponding to small data in an appropriate sense are global in time.
\end{itemize}
This behavior of the solutions was named ``Fujita phenomenon'' thereafter. On the other hand, by a generalization of result by Kaplan \cite{Kaplan}, solutions corresponding to sufficiently large data blow up for any $\gamma>1$.

\smallskip

The situation on a general Riemannian manifold, $\Delta$ being now the nonnegative Laplace--Beltrami operator,  can be significantly different. A remarkable case is that of the hyperbolic space $\mathbb{H}^n$, where an analogue of \eqref{eq: heatRn} was first studied by Bandle, Pozio and Tesei~\cite{BPT}. They showed that for all $\gamma>1$, sufficiently small initial data give always rise to global in time solutions; in other words, the Fujita phenomenon does not occur. Nonetheless, the dichotomy can be recovered provided a suitable time-dependent nonlinearity is present. Namely, if the reaction term $|u|^{\gamma-1}u$ is replaced by $\e^{\beta t}\, |u|^{\gamma-1}u$, where $\beta>0$ is a fixed parameter, then a Fujita-type phenomenon takes indeed place, the threshold value being $\gamma^{*}=1+\beta/\lambda_0$,  where $\lambda_0:= (n -1)^2/4$ is the bottom of the $L^2$ spectrum of $\Delta$ on $\mathbb{H}^n$. It is also shown in \cite{BPT} that if the exponential factor in time is replaced by a power of time the Fujita phenomenon still does not occur. Heuristically, the exponential-type time-dependent nonlinearity may be seen as a counterbalance to the negative curvature and the spectral gap of $\Delta$; so that rather its Euclidean counterpart, the Fujita problem on $\mathbb{H}^{n}$ resembles the  Euclidean problem on bounded domains~\cite{Meier}. 

For our discussion, it is important to stress that though the critical case $\gamma=\gamma^{*}$ was shown to belong to the non-blow-up case in \cite{BPT}, this was done only for a restricted range of $\beta$'s. The puzzle for the critical case $\gamma=\gamma^{*}$ was completed by Wang and Yin~\cite{WY}, who proved the existence of supersolutions of the form $\e^{-\lambda t}\bar{v}$, where $\bar{v}$ is a nonnegative supersolution of the semilinear elliptic equation
\begin{equation}\label{eq:semilinear}
	\Delta v-\lambda v - u^{\gamma}=0
\end{equation}
which had already been studied by Mancini and Sandeep~\cite{MS}, and which has an independent interest beyond the Fujita problem. Let us also emphasize that to the best of our knowledge, the hyperbolic space case is the only case to date where a complete picture of the Fujita phenomenon for $\Delta$ outside the Euclidean setting is understood. On other negatively curved Riemannian manifolds, only partial results are at disposal, see e.g.~\cite{Punzo2012, Punzo2014}.

\smallskip

From the point of view of L\'evy processes, as well as from an abstract functional-analytic perspective, it is of interest to consider an analogue of~\eqref{eq: heatRn} where $\Delta$ is replaced by a nonlocal operator like its fractional powers $\Delta^{\sigma}$, where $\sigma \in (0,1)$. As long as the Euclidean setting is concerned, a fractional semilinear heat equation was considered in~\cite{Sug, IKK}, see also~\cite{BPV25,HIN, DPF25}, where it is shown that the Fujita phenomenon occurs at the threshold $1 +\frac{2\sigma}{n}$. A related fractional version of the semilinear problem~\eqref{eq:semilinear}, 
in the Sobolev space $H^{\sigma}(\mathbb{R}^n)$, was studied for $n=1$ in~\cite{FL13} and later for all $n\geq 1$ in \cite{FLS16}.

\smallskip

In this paper we provide a complete picture of the Fujita phenomenon for the fractional Laplace--Beltrami operator on the hyperbolic spaces $\mathbb{H}^{n}$, and at the same time, we study nonnegative solutions to a fractional version of~\eqref{eq:semilinear} as we now explain.

Our first result concerns the Fujita exponent for the fractional heat equation
\begin{equation}\label{eq:fracheat0}
\begin{cases}
	\partial_t u + \Delta^{\sigma}u = \e^{\beta t}|u|^{\gamma-1}u, \quad t>0, \, x\in \mathbb{H}^n \\
	u(0, \cdot)=f,
\end{cases} 
\end{equation}
where $\Delta$ is the nonnegative Laplace--Beltrami operator on the hyperbolic space $\mathbb{H}^n$, and $f \geq 0$ is bounded and continuous. It reads as follows.
	\begin{main_theorem}\label{thm:A}
	Suppose $\beta>0$ and $\sigma \in (0,1)$. Define $\gamma^{*} = 1 + \frac{\beta}{\lambda_{0}^{\sigma}}$. Then the following holds.
	\begin{itemize}
		\item If $1<\gamma< \gamma^{*}$, then every nontrivial nonnegative mild solution to the problem~\eqref{eq:fracheat0} blows up in finite time.
		\item If $\gamma\geq \gamma^{*}$, then there exists a nontrivial nonnegative classical global solution to the problem~\eqref{eq:fracheat0} for sufficiently small initial data $f$.
	\end{itemize}
\end{main_theorem}
Theorem~\ref{thm:A} may be seen as the fractional analogue on $\mathbb{H}^{n}$ of the above-mentioned celebrated results by Bandle, Pozio and Tesei~\cite{BPT} and Wang and Yin~\cite{WY}. It does not come as a surprise  that the result does not resemble its fractional Euclidean counterpart.

The critical case $\gamma=\gamma^{*}$ for small $\beta$ requires a detour in the spirit of~\cite{WY} where the fractional analogue of~\eqref{eq:semilinear} comes into play, as the discussion above suggests. This leads us to our second main result, which has an independent interest and which is the first of its kind on hyperbolic spaces.
	\begin{main_theorem}\label{thm:B}
	Suppose $0\leq \lambda \leq \lambda_{0}$, $\sigma \in (0,1)$ and $1<\gamma< \frac{n+2\sigma}{n-2\sigma}$. Then the equation
	\begin{equation}\label{criticreduction0}
		\Delta^{\sigma} v - \lambda^{\sigma} v - v^{\gamma}=0
	\end{equation}		
	has at least one nontrivial nonnegative radial bounded solution such that
\begin{equation}\label{fracenergy}
	\int_{\mathbb{H}^{n}} (|\Delta^{\sigma/2} v|^{2}-\lambda^{\sigma}v^{2})\, \dd \mu
	\end{equation}
	is finite, and which belongs to $L^q(\mathbb{H}^{n})$ for all $q\in (2, \infty]$. The solution is understood in the classical sense, i.e.  $v\in \mathcal{C}^{0,2\sigma+\varepsilon}(\mathbb{H}^{n})$ if $0<\sigma<1/2$, while $v\in \mathcal{C}^{1,2\sigma-1+\varepsilon}(\mathbb{H}^{n})$ if $1/2\leq \sigma <1$ for some small $\varepsilon>0$, and~\eqref{criticreduction0} is satisfied pointwise everywhere on $\mathbb{H}^n$.
\end{main_theorem}
The energy~\eqref{fracenergy} of the solution is linked to our third main result, which is the actual backbone of the paper. It is a fractional Poincar\'e-type inequality which extends those known so far, see~\cite{MS, BP2022} and Remark~\ref{rem:Poincarecomparison} below.
	\begin{main_theorem}\label{thm:C}
	Suppose $\sigma \in (0,1)$ and $2<q\leq \frac{2n}{n-2\sigma}$. There exists $C>0$ such that for all $\phi\in \mathcal{C}_{c}^{\infty}(\mathbb{H}^{n})$
	\begin{equation*}
\int_{\mathbb{H}^n} (|\Delta^{\sigma/2} \phi |^{2}-\lambda_{0}^{\sigma}\phi^{2})\, \dd \mu \geq C \|\phi \|_{q}^{2}.
	\end{equation*}
\end{main_theorem}
Let us stress that the left-hand side of such inequality is not equivalent to the fractional Sobolev norm or order $\sigma$ of $\phi$. It is, though, a norm, which gives rise to a new scale of $L^2$ fractional Sobolev spaces. For these we shall prove a compact embedding theorem in $L^{q}$ for radial functions of Rellich--Kondrachov-type, when $q$ is in the open range $2<q< \frac{2n}{n-2\sigma}$, along with other related properties. It is also worth pointing out that the corresponding non-fractional inequality for test functions (where $\sigma=1$ and $\Delta^{1/2}$ is replaced by the Riemannian gradient $\nabla$) holds for $n\geq 3$ and $2<q\leq \frac{2n}{n-2}$, while for $n=2$ it holds for all $q>2$. Since $\sigma\in (0,1)$ this dichotomy on the dimension for the range of $q$ does not occur in our case. 

\smallskip

A few comments on our results and techniques are now in order. Let us first say that as already known in the Euclidean case, $\Delta^{\sigma}f$ can be expressed with a pointwise formula only if $f$ has a mild growth at infinity.  Inspired by Silvestre~\cite{Sil}, we detect a suitable class of functions for which this holds, tailored on the hyperbolic geometry at infinity; see~\eqref{eq: class Lsigma}. 

\smallskip

The presence of the spectral gap of $\Delta$, as well as the non-locality of its fractional powers, considerably change the picture with respect to both the Euclidean setting and the non-fractional hyperbolic case. At the same time, despite some a posteriori similarities, the techniques for bounded domains on $\mathbb{R}^n$ do not apply. As for Theorem~\ref{thm:A}, when $\gamma<\gamma^{*}$ we rely on precise pointwise bounds for the fractional heat kernel with respect to both space and time. Let us stress that these differ from the Euclidean case and are often more challenging to handle in the hyperbolic setting. On the other hand, if $\gamma>\gamma^{*}$, or $\gamma=\gamma^{*}$ and $\beta >\frac{2}{3}\lambda_{0}^{\sigma}$, we prove global existence by adapting some ideas of Weissler~\cite{Weissler1981}; precise information on the fractional heat kernel is required once more. However, much like in the non fractional case, the critical regime $\gamma=\gamma^{*}$ and $\beta \leq\frac{2}{3}\lambda_{0}^{\sigma}$ needs a different approach, which goes through a comparison principle for pointwise (super)solutions to \eqref{eq:fracheat0} with enough regularity; see Theorem \ref{thm: comparison}. As mentioned above, equation~\eqref{criticreduction0} comes into play.
 
 \smallskip
 
 This brings us to Theorems~\ref{thm:B}, and in turn to~\ref{thm:C}. Provided that a good fractional Sobolev embedding is at hand, we use functional-analytic techniques and a compactness argument to show existence of a weak solution to~\eqref{criticreduction0}. Indeed, if one defines 
 \[
 \|\phi \|_{\lambda,\sigma}:= \| (\Delta^{\sigma} - \lambda^{\sigma})^{1/2}\phi\|_{2}, \qquad \phi \in \mathcal{C}_{c}^{\infty},
 \]
 then $\| \cdot \|_{\lambda,\sigma}$ is a norm, and the completion $\mathcal{H}_{\lambda,\sigma}$ of $\mathcal{C}_{c}^{\infty}$ with respect to it is a Hilbert space. If $\lambda <\lambda_{0}$, then $\mathcal{H}_{\lambda,\sigma}$ is nothing but the fractional Sobolev space $\mathrm{H}^{\sigma}$ of order $\sigma$, but this is not the case if $\lambda = \lambda_{0}$; actually $\mathcal{H}_{\lambda_{0},\sigma} \supsetneq \mathrm{H}^{\sigma}$ for all $\sigma \in (0,1)$. Nevertheless, Theorem~\ref{thm:C} shows that $\mathcal{H}_{\lambda_{0},\sigma} \hookrightarrow L^{q}$ for  $2<q  \leq  \frac{2n}{n-2\sigma}$. 
 
 Despite being a larger space, $\mathcal{H}_{\lambda_{0},\sigma}$ shares many convenient properties with $\mathrm{H}^{\sigma}$, e.g., the facts that absolute value and radialization preserve the space $\mathcal{H}_{\lambda_{0},\sigma}$ without increasing the norm; see Corollary~\ref{cor:abs} and Proposition~\ref{prop:rad}. Moreover -- and this is the key towards Theorem~\ref{thm:B} -- a compact embedding in $L^{q}$ for radial functions holds in the open range of $q$'s; this is Theorem \ref{thm: cpt radial emb}. Its proof needs quite some effort, as different techniques are needed locally and globally: the first is treated by a classical Rellich-type compactness argument, while at infinity we prove and use a hyperbolic fractional analog of a result due to Lions~\cite{Lions}. To this end, a further distinction is needed between the case $0<\sigma\leq 1/2$, which we treat via harmonic analysis tools and pointwise bounds, and the case $\sigma>1/2$, for which we rely on interpolation.
  
With this toolkit at hand, inspired by Mancini and Sandeep~\cite{MS} and by Dutta and Sandeep~\cite{DuttaSandeep}, we show in Theorem \ref{existence-MS} that for $0\leq \lambda \leq \lambda_{0}$, $\sigma \in (0,1)$ and $1<\gamma< \frac{n+2\sigma}{n-2\sigma}$, equation~\eqref{criticreduction0} has at least one nontrivial non-negative weak solution in $\mathcal{H}_{\lambda,\sigma}$. This is done via a compactness argument. The problem of its regularity boils down, by means of some regularity results from~\cite{Sil} and~\cite{BanEtAl} for the fractional Laplacian, to showing that such a solution is bounded. The latter is proved by exploiting the pointwise behavior of the convolution kernel of $(\Delta^{\sigma}-\lambda^{\sigma})^{-1}$. 

It is clear by now that the fractional Poincar{\'e} inequality given by Theorem~\ref{thm:C} lies at the heart of all our results. We prove it via a delicate use of pointwise and Fourier estimates for the fractional heat kernel and the Kunze--Stein phenomenon.

\smallskip

All of our results actually generalize, with essentially no additional effort, to rank one symmetric spaces of non-compact type. We keep the discussion on a hyperbolic space level only for simplicity. In higher rank the fractional Poincar{\'e} inequality remains valid, and the same is true for all Fujita-type results except for the critical range of parameters $\gamma=\gamma^{*}$ and $0<\beta<\frac{2}{\nu}\lambda_{0}^{\sigma}$, where $\nu$ is now the pseudo-dimension of the symmetric space. We discuss this at the very end.

\smallskip

Let us finally say that in this paper we do not delve into an exhaustive study of equation~\eqref{criticreduction0} in the spirit of~\cite{MS} or~\cite{DuttaSandeep}, but rather restrict to those results needed by the Fujita problem. In particular, we do not have pointwise results for the solutions: in the search for these, the nonlocal nature of the fractional Laplacian $\Delta^{\sigma}$ manifests in its stronest form.  This is certainly an interesting problem for the future.

\subsection*{Structure of the paper} Let us finally describe the structure of the paper. In Section \ref{Sec:2} we recall certain preliminaries on hyperbolic spaces and Fourier analysis thereon. In Section \ref{Sec:3} we discuss the Fujita problem, various notions on solutions and present our complete results concerning finite time blow-up or global existence of solutions to \eqref{eq:fracheat0}. Their proofs are then given in Section \ref{Sec:4} and Section \ref{Sec:5}, respectively,  under the assumption that Theorems \ref{thm:B} and \ref{thm:C} hold. To this end, we prove Theorem~\ref{thm:C}, i.e.\ the fractional Poincar{\'e} inequality, in Section~\ref{Sec:6}, as an indispensable tool towards Theorem~\ref{thm:B}, to which the following three sections are dedicated. More precisely, in Section \ref{Sec:7} we discuss the space $\mathcal{H}_{\lambda, \sigma}$. In Section \ref{Sec:8} we show that the embedding of the radial subspace $\mathcal{H}_{\lambda, \sigma}^{\text{rad}}\hookrightarrow L^q$ is compact if $q\in (2, \frac{2n}{n-2\sigma})$. Section~\ref{Sec:9} contains the actual proof of Theorem~\ref{thm:B}. Finally, in Section \ref{Sec:10} we discuss the validity of our results on Riemannian symmetric spaces of non-compact type.

	\subsection*{Acknowledgements} T. B. was partially supported by the INdAM--GNAMPA Project “Function Spaces and Applications” (CUP\_E5324001950001). 
	
E. P. is supported by the Deutsche Forschungsgemeinschaft (DFG, German Research Foundation) –SFB-Gesch{\"a}ftszeichen –Projektnummer SFB-TRR
358/1 2023 –491392403. 

The authors are extremely grateful to Kunnath Sandeep for an enlightening discussion about his work. 
	
		\section{Preliminaries: analysis on hyperbolic spaces}	\label{Sec:2}
	In this section we give some preliminaries on the real hyperbolic space $\mathbb{H}^n=\mathbb{H}^n(\mathbb{R})$, $n\geq 2$, a model of which is given by the upper sheet of the hyperboloid
	\[
	\mathbb{H}^n=\{ (x_0,x_1, ... ,x_n)\in\mathbb{R}^{n+1} \colon x_0^2-x_1^2-...-x_n^2=1,\;x_0>0\}.\]
	Using polar coordinates,
	\[
	x_0=\cosh r\quad\text{and}\quad(x_1, ..., x_n)=(\sinh r)\,\omega\quad
	\text{with}\quad r\ge 0,\,\; \omega\in\mathbb{S}^{n-1}.
	\]
	The distance of a point $x=x(r,\omega)=\bigl(\cosh r,(\sinh r)\,\omega\bigr)\in\mathbb{H}^n$
	to the origin $o=(1,0, ... ,0)$ of the hyperboloid  is
	\[
	d(x,o)=r,
	\]
	and more generally, the distance between two arbitrary points
	$x=x(r,\omega)$ and $y=y(s,\omega')$ satisfies
\[
		\cosh d(x,y)=(\cosh r)(\cosh s)-(\sinh r)(\sinh s)\,\omega\cdot\omega'.
\]
We shall often write $|x|$ for $d(x,o)$, and denote the geodesic ball around the origin of radius $\varrho>0$ by $B(o, \varrho)$. Integration on $\mathbb{H}^n$ with respect to its Riemannian measure $\mu$ is given by
	\begin{equation*}
		\int_{\mathbb{H}^n} f(x) \, d\mu(x)
		=\int_0^{\infty}\int_{\mathbb{S}^{n-1}}u_0(r,\omega) \, d\omega \,(\sinh r)^{n-1} \, dr .
	\end{equation*}
Given that it can be written as the quotient $\mathbb{H}^n=SO^{0}(n,1)/SO(n)$, the hyperbolic space $\mathbb{H}^n$ is the simplest example of a noncompact symmetric space (of rank one). Let us elaborate on this, since the group framework will be quite useful. Our main reference is~\cite{Hel}.
	
	\subsection{Rank one noncompact symmetric space structure}\label{subsection SS} 
	Let ${G}$ be a connected,  noncompact semisimple Lie group with finite center. Let $K$ be a maximal compact subgroup of ${G}$ and $\mathbb{X}={G}/K$ be the corresponding symmetric space. We consider a Cartan decomposition $\mathfrak{g}=\mathfrak{k}\oplus\mathfrak{p}$ of the Lie algebra of ${G}$. Fix a maximal abelian subspace $\mathfrak{a}$ of $\mathfrak{p}$ and consider the decomposition $\mathfrak{g}=\mathfrak{n}\oplus\mathfrak{a}\oplus\mathfrak{k}$. If $\mathfrak{a}\cong \mathbb{R}$, then we say that the symmetric space $\mathbb{X}$ has rank one. Taking $G=SO^{0}(n,1)$ and $K=SO(n)$ from now on, one has that the real hyperbolic space $\mathbb{H}^n$ is a symmetric space of rank one.
	
	The group admits the following decompositions,
	\begin{align*}
		\begin{cases}
			\,{G}\,=\,N\,(\exp \mathfrak{a})\,K 
			\qquad&\textnormal{(Iwasawa)}, \\[5pt]
			\,{G}\,=\,K\,(\exp\overline{\mathfrak{a}^{+}})\,K
			\qquad&\textnormal{(Cartan)},
		\end{cases}
	\end{align*}
	where $N\cong \mathbb{R}^{n-1}.$
Denote by $\alpha$ the unique positive root of $\mathfrak{g}$ with respect to $\mathfrak{a}$, with multiplicity $n-1$, and let $H_0$ be the unique element of $\mathfrak{a}$ with the property that $\langle \alpha, H_0 \rangle=1$. Let $\tau(g)$ be the real number such that
	\[
	g=n\exp (\tau(g)  H_0)k, \qquad g\in G.
	\] To simplify the notation,
	we often identify the Lie subgroup $\exp \mathfrak{a}$ with the real line $\mathbb{R}$ using the map $\tau \mapsto \exp(\tau\,H_0)$, and identify  $\exp\overline{\mathfrak{a}^{+}}$ with $[0, +\infty)$. In the Cartan decomposition, integration on $G$ with respect to a Haar measure writes
\[
		\int_{{G}}f(g)\,\diff{g}
		=
		\textrm{const.} \int_{K}
		\int_{0}^{\infty} \int_{K}f(k_{1}(\exp r\, H_0)k_{2}) \delta(r) \, \diff{k_2}\, \diff{r} \, \diff{k_1}.
\]
Here $K$ is equipped with its normalized Haar measure and ``const'' is a positive normalizing constant, so that for right-$K$ invariant functions, we have
	$$\int_{\mathbb{H}^n}f(x)\,\diff\mu(x)=\int_{{G}}f(g)\,\diff g.$$
	The density $\delta$ satisfies
	\begin{equation}\label{dens}
		\delta(r) =(\sinh r)^{n-1} \asymp 
		\begin{cases} r^{n-1}, \quad & 0<r<1\\
			\e^{(n-1)r}, \quad & r\geq 1.
		\end{cases}
	\end{equation}
Above and all throughout, we write $f \approx g$ for two positive functions $f$ and $g$ whenever there exists $C > 0$ (depending on circumstantial parameters) such that $C^{-1}g \leq f \leq C g$. Analogously, we shall write $f\lesssim g$ if there exists such a $C$ such that $f\leq C g$. 	
	
	Finally, we describe Fourier analysis on rank one non-compact symmetric spaces. We refer to \cite{Hel}*{Ch.\ III} for more details. Denote by $\rho$ the number $(n-1)/2$. For continuous compactly supported functions, the Helgason--Fourier transform is defined by
	\begin{align}\label{H-Ftr}
		\widehat{f}(\xi,k\mathbb{M})=\,\int_{{G}}\,f(gK)\,
		\e^{(-i\xi+\rho)\,\tau(k^{-1}g)}\,\diff{g}, \qquad \xi \in \mathbb{C}, \; k\in K.
	\end{align}
	Here, $\mathbb{M}$ denotes the centralizer of $\exp\mathfrak{a}$ in $K$. Let us also define the spherical transform of continuous compactly supported and radial functions by
	\begin{align*}
		\mathcal{H}f(\xi)=\!\int_{G}f(gK)\,\varphi_{-\xi}(g)\, \diff{g}=C\int_{0}^{\infty} f(r)\, \varphi_{-\xi}(r)\, (\sinh r)^{n-1} \, \dd r \qquad \xi \in \mathbb{C},
	\end{align*}
	where $\varphi_{\xi}$ is the
	elementary spherical function of index $\xi\in\mathbb{C}$. The functions $\varphi_{\xi}$ are normalized eigenfunctions of the (nonnegative) Laplace-Beltrami operator $\Delta$, that is, $\Delta \varphi_{\xi}=(\xi^2+\rho^2)\varphi_{\xi}$ with $\varphi_{\xi}(o)=1$, the spectrum of $\Delta$ being
\[
\sigma(\Delta) =[\lambda_{0}, \infty), \qquad  \lambda_0=\rho^2=(n-1)^2/4.
\]
The spherical functions $\varphi_{\xi}$ are smooth and radial in space, and even in the spectral parameter, that is $\varphi_{\xi}=\varphi_{-\xi}$. This in turn implies that $\mathcal{H}f$ is even on $\mathbb{R}$ when $f$ is radial on $\mathbb{H}^n$; recall that a function $f$ on $\mathbb{H}^n$ is said to be \emph{radial}  if there exists $\tilde f \colon [0,+\infty) \to \mathbb{C}$ which we call the \emph{profile} of $f$ such that $f(x) = \tilde f (d(x,o))$. In the group framework,  $f$ radial means that $f(kx)=f(x)$ for all $k\in SO(n)$ and $x\in \mathbb{H}^n$.

 An integral representation formula is given by
	\begin{align}\label{eq: spherical hyp}
		\varphi_\xi(r) & =\int_K \e^{(-i\xi+\rho)\tau(kx)}\, \dd k
		\end{align}
	see for instance \cite[p.40]{Koo84}. Notice that globally, it holds 
	\begin{equation}\label{eq: ground}
		\varphi_{0}(r)\asymp (1+r)\,\e^{-\frac{n-1}{2}r},
	\end{equation}
whence $\varphi_{0}\in L^{2+\varepsilon}$ for every $\varepsilon>0$. Recall that in the case of radial smooth and compactly supported functions, the Helgason--Fourier transform boils down to the spherical transform, that is, 
\[
\widehat{f}(\xi,k\mathbb{M})=\mathcal{H}f(\xi) \qquad  \xi\in\R, \, k\in K.
\]
	Passing to inversion formulas, for a smooth and compactly supported function $f$ on $\mathbb{H}^n$ it holds
	\begin{equation}\label{eq:inversionHelg}
	f(gK)=\text{const.}\int_{K}\int_{-\infty}^{\infty}\e^{(i\xi+\rho)\tau(k^{-1}g)} \widehat{f}(\xi, k\mathbb{M}) \,  \frac{\dd \xi}{|\mathbf{c}(\xi)|^{2}}\, \dd k, \qquad g\in G,
	\end{equation}
	where the Plancherel density is
	\begin{equation}\label{eq: Plancherel}
		|\mathbf{c}(\xi)|^{-2}=\mathbf{c}(\xi)^{-1}\mathbf{c}(-\xi)^{-1}=\text {const.} \frac{\left|\Gamma\left(i \xi+\frac{n-1}{2}\right)\right|^2}{|\Gamma(i \xi)|^2} \asymp \begin{cases}\xi^2 & \text { if }|\xi| \leq 1 \\ |\xi|^{n-1} & \text { if }|\xi| \geq 1\end{cases},
	\end{equation}
	and if the function is in addition radial, then the inversion formula simplifies to
	$$
	f(x)=\text{const.}\int_{0}^{\infty} \mathcal{H}f(\xi) \,  \varphi_{\xi}(x)\frac{\dd \xi}{|\mathbf{c}(\xi)|^{2}}.
	$$
	
	A version of the Plancherel theorem holds in $\mathbb{H}^n$, \cite[Chapter III, Theorem 1.5]{Hel}. Indeed for all $f\in L^{2}$
	\begin{equation}\label{eq: Plancherel thm}
		\int_{\mathbb{H}^n} |f(x)|^2 \, \dd \mu (x)= \text{const.}\int_{K}\int_{\mathbb{R}} |\widehat{f}(\xi, k\mathbb{M})|^2\, |\mathbf{c}(\xi)|^{-2} \, \dd \xi \, \dd k.
	\end{equation}
	Moreover, a version of the Paley--Wiener theorem also holds, \cite[Chapter III, Theorem 5.1]{Hel}: if  $f\in \mathcal{C}_c^{\infty}$, then for all $N\in \mathbb{N}$
	\begin{equation}\label{eq: PW}
		\sup_{\xi \in \mathbb{R}, k\in K} (1+|\xi|)^N |\widehat{f}(\xi, k\mathbb{M})|<\infty.
	\end{equation}

	\subsection{Heat kernel}
Consider the heat equation for $\Delta$, given by
\[
		\begin{cases}
			\partial_{t}u(t,x)+\Delta u(t,x)=0, 
			\qquad t>0,\,\,x\in\mathbb{H}^n,\\
			u(0,x)=f(x).
		\end{cases}
\]
The heat kernel  $h_{t}$, that is, the minimal
	positive fundamental solution of the heat equation or, equivalently, the
	convolution kernel of the heat semigroup $\e^{-t\Delta}$, is a smooth radial function. Its behavior is described by
\[
		h_t(x)\asymp  t^{-\frac n2}(1 + r)\,(1+ t+r)^{\frac{n-3}2}
		\e^{-(\frac{n-1}2)^2t-\frac{n-1}{2}r-\frac{r^2}{4t}}, \qquad r=d(x,o),
\]
	see for instance \cite{AJ99, AO, DaMa1988}.

	\subsection{Positive powers of the Laplacian} Suppose $\sigma \in (0,1)$, and consider the function
	$$
	P_0^\sigma(x):=\int_0^{\infty} h_t(x) \frac{\dd t}{t^{1+\sigma}}.
	$$
which is a nonnegative radial function, as is $h_{t}$. The following two-sided estimates for it were proved in \cite[Theorem 2.4]{BanEtAl} (see also \cite[Theorem 3.3]{BP2022}): 
\begin{alignat}{2}
    P_0^{\sigma}(x) &\asymp |x|^{-1-\sigma} \, \e^{-(n-1)|x|} &&\quad \text{for } |x| \geq 1, \notag \\
                    &\asymp |x|^{-n - 2\sigma}               &&\quad \text{for } |x| < 1, \label{eq: P0 est}
\end{alignat}
	which in turn imply that $P_0^{\sigma}\in L^p(\mathbb{H}^n\backslash B(o,1))$ for all $p\in[1,\infty]$.
	
	These estimates and~\cite[Theorem 2.5]{BanEtAl} yield that if $ {\phi}\in \mathcal{C}_c^{\infty}$, then for all $0<\sigma<1$, the operator $\Delta^{\sigma}$ defined spectrally by
	\[
	\widehat{\Delta^{\sigma}  {\phi}}(\xi, k\mathbb{M}) = (\xi^{2} +\rho^{2})^{\sigma} \widehat{ {\phi}}(\xi, k\mathbb{M}) \qquad \xi\in \R, \, k\in K,
	\]
	can be represented by the integral
	\begin{equation}\label{eq:singint}
		\Delta^{\sigma} {\phi}(x)=\frac{1}{|\Gamma(-\sigma)|} \, \mathrm{p.v.} \! \int_{\mathbb{H}^n}( {\phi}(x)- {\phi}(z)) \, P_0^{\sigma}(z^{-1}x) \, \dd z.
	\end{equation}
This formula actually holds for more general classes of functions, which we shall discuss in due time.  	

For all $0<\sigma<1$, $0\leq \lambda\leq\lambda_{0}$ and for all $\alpha>0$, taking also into account the behavior~\eqref{eq: Plancherel} of the Plancherel density $|\mathbf{c}(\xi)|^{-2}$, the Plancherel theorem~\eqref{eq: Plancherel thm} yields
	\[
	\|(\Delta^{\sigma}-\lambda^{\sigma})^{\frac{\alpha}{2}}f\|_2=\int_K\int_{\mathbb{R}} \left((\xi^2+\lambda_0)^{\sigma}-\lambda^{\sigma}\right)^{\alpha}|\widehat{f}(\xi, k\mathbb{M})|^2 \, \frac{\dd \xi}{|\mathbf{c}(\xi)|^2}\, \dd k<\infty.
	\]

\subsection{Fractional heat kernel}
	Suppose $\sigma\in (0,1)$ and let $\eta_{t}^{\sigma}$ be the inverse Laplace transform of the function $s \mapsto \exp\{-t s^{\sigma}\}$.  Via subordination to the heat semigroup, we may write 
	\[
	\e^{-t\Delta^{\sigma}}=\int_0^{+\infty}\eta_{t}^{\sigma}(s)\, \e^{-s\Delta} \,\diff{s}.
	\]
	Then, the solution to the (homogeneous) fractional heat equation
	\begin{equation*}
		\begin{cases}
			\partial_{t}u(t,x)+\Delta^{\sigma} u(t,x)=0, 
			\qquad t>0,\,\,x\in\mathbb{H}^n,\\
			u(0,x)=f(x).
		\end{cases}
	\end{equation*} 
	is given by the right convolution (for sufficiently regular functions $f$)
	\[
	u(t,x)=\e^{-t\Delta^{\sigma}}f(x)=f\ast P_t^{\sigma}(x)=\int_{G} P_t^{\sigma}(y^{-1}x)f(y)\,\diff{y}, \qquad x\in \mathbb{H}^{n}, \, t>0,
	\]
	where the kernels are given by
	\begin{equation}\label{kernelFH}
		P_t^{\sigma}(x)=\int_0^{+\infty} h_s(x)\,\eta_{t}^{\sigma}(s)\,\diff{s}, \qquad x\in \mathbb{H}^{n} , \, t>0.
	\end{equation}
	Precise estimates for the subordinator $\eta_{t}^{\sigma}$ from above and below are know, see e.g.~\cite[(8)~and~(9)]{GS04}. Observe moreover that formula~\eqref{kernelFH} implies that $P_t^{\sigma}$, which we call \emph{fractional heat kernel}, is radial and positive. We recall the following result concerning its behavior, see~\cite{GS04, BP}: for any $x\in \mathbb{H}^n$, $t>0$, and $\kappa>0$,
		\begin{equation}\label{P kernel estimates}
			P_t^{\sigma}(x)\asymp
			\begin{cases}
				t\,(t^{\frac{1}{2\sigma}}+|x|)^{-(n+2\sigma)} & \quad \text{if } t+|x| \leq \kappa,
				\\[5pt]
				\varphi_{0}(x) \, t^{\frac{1}{2-2\sigma}} \, (t+|x|)^{-\frac{3}{2}-\frac{1}{2-2\sigma}}\, \e^{-\lambda_0^\sigma t} &\quad  \text{if } t+|x| \geq \kappa \text{ and } |x|\leq \sqrt{t},
				\\[5pt]
				\varphi_{0}(x)\,t\,(t+|x|)^{-2-\sigma}\e^{-\rho|x|} &\quad  \text{if } t+|x| \geq \kappa \text{ and } |x|\geq t^{1/\sigma}.
			\end{cases}
		\end{equation}
	The expressions of upper and lower bounds of $P_t^{\sigma}$ in the region excluded above are more complicated, and actually involve a certain implicit function, which we set aside now in order not to overwhelm the discussion at this stage.
	
The semigroup $\e^{-t\Delta^{\sigma}}$ is a strongly continuous contraction semigroup on $L^{p}$ for all $p\in [1,\infty)$. For the $C_0$-semigroup property, see \cite[p.260]{Y}, while for the contractivity, the subordination formula~\eqref{kernelFH} and the fact that the heat kernel is a probability measure on $\mathbb{H}^n$ implies, via a Fubini argument, that $\|P_t^{\sigma}\|_1=1$ for all $t>0$. We also have {$\left\|P_t^{\sigma}\right\|_p \asymp t^{-\frac{n}{2\sigma p'}}, \; t\in(0,1),$} while for $t\geq 1$,
\begin{equation}\label{PtLp}
\left\|P_t^{\sigma}\right\|_p \asymp
\begin{cases}
	t^{-\frac{1}{2 p^{\prime}}} \e^{-t\frac{\rho^{2\sigma}}{(p p^{\prime})^\sigma}}, & 1 \leq p<2; \\
	t^{-\frac{3}{4}} \e^{-\rho^{2 \sigma} t}, & p=2; \\
	t^{-\frac{3}{2}} \e^{-\rho^{2 \sigma} t}, & 2<p\leq \infty.
\end{cases}
\end{equation}
These estimates can be found in \cite[Lemma 3]{CGMII} for $1<p\leq 2$ and $p=\infty$. The upper bound for $2<p<\infty$ may be obtained via the Kunze-Stein phenomenon, see e.g. \cite[p.122]{CGM} and the semigroup property from which $\|P_t^{\sigma}\|_p= \|P_{t/2}^{\sigma}\ast P_{t/2}^{\sigma}\|_p\leq\|P_{t/2}^{\sigma}\|_2^2$. As for the lower bound for $2<p<\infty$, it suffices to integrate over $B(o, \sqrt{t})$, $t\geq 1$, where according to \eqref{P kernel estimates} it holds $P_t^{\sigma}(r)\asymp t^{-\frac{3}{2}}\,\e^{-\rho^{2\sigma}t}\, (1+r)\, \e^{-\rho r}$. 
	
\subsection{Some classes of functions} In this section we discuss some function spaces on $\mathbb{H}^{n}$ which will be used all throughout the paper; some of these have already appeared. In particular, we shall write
\begin{itemize}
\item $\mathcal{C}_{c}^{\infty}$ for the space of smooth and compactly supported functions;
\item $\mathcal{C}^{k}$, for $k\in\N \cup \{0\}$, for the space of $k$-times continuously differentiable functions.
\item $\mathcal{C}_0$ for the space of continuous functions vanishing at infinity.
\item $\mathcal{C}^{0,\alpha}$, given $\alpha \in (0,1]$, for the space of (H\"older continuous) functions $f$ such that there exists $C>0$ for which
	\[
	|f(x)-f(y)|\leq C\, d(x,y)^{\alpha}, \qquad \forall x, y\in \mathbb{H}^n;
	\]
\item $\mathcal{C}^{k,\alpha}$, given $k\in \N$ and $\alpha \in (0,1]$, for the H\"older space of functions $f$ such that $\nabla^{k}f$ belongs to $\mathcal{C}^{0,\alpha}$;
\item $\mathrm{H}^{\sigma}$ for the $L^{2}$ Sobolev space of order $\sigma>0$, defined as the space of functions $f\in L^{2}$ such that $\Delta^{\sigma} f \in L^{2}$, endowed with the Sobolev norm
\[
\|f\|_{\mathrm{H}^{\sigma}} = \|f\|_{2} + \|\Delta^{\sigma} f\|_{2}.
\]
\item $L_{\sigma}$ for the space of measurable functions $f$ such that 
		\begin{equation}\label{eq: class Lsigma}
		\int_{\mathbb{H}^n}|f(x)|\,(1+|x|)^{-1-\sigma}\,\e^{-(n-1)|x|} \, \dd \mu(x)
		\end{equation}
		is finite.
	The class $L_{\sigma}$ is the hyperbolic analog of the Euclidean class of functions introduced by Silvestre (wherefrom the notation $L_{\sigma}$) in \cite[Proposition 2.1.4]{Sil}, as we shall see below.
\end{itemize}

Let us now note or recall some properties of some of these function spaces. First of all, because of the spectral gap of $\Delta$, the Sobolev norm $\| \cdot \|_{\mathrm{H^{\sigma}}}$ is equivalent to the norm
\[
\|\Delta^{\sigma}f\|_{2} = \int_{K}\int_{\mathbb{R}}|\widehat{f}(\xi, k\mathbb{M})|^2 \,(\xi^2+\rho^2)^{\sigma} \, |\textbf{c}(\xi)|^{-2} \,\dd\xi \, \dd k.
\]
Other basic properties of the spaces $\mathrm{H^{\sigma}}$ can be found e.g. in~\cite[Lemma 2.2]{An92}; among these we shall need that $\mathcal{C}_c^{\infty}$ is dense in $\mathrm{H}^{\sigma}$, and that they behave properly with respect to complex interpolation:
	$$
	\mathrm{H}^{\sigma} =\left[\mathrm{H}^{\sigma_0} , \mathrm{H}^{\sigma_1} \right]_\theta \quad \text { when } \sigma=(1-\theta) \sigma_0+\theta {\sigma_1}, \qquad \theta \in (0,1), \, \sigma_{0},\sigma_{1}>0.
	$$
As for the space $L_{\sigma}$, let us notice that $L^p\subseteq L_{\sigma}$ for all $p\in [1,\infty]$, and that we have the following.
	
	\begin{proposition}\label{Prop 4.3} 
	Suppose $0<\sigma<1$ and $\phi\in \mathcal{C}_c^{\infty}$. Then
		\[
		|\Delta^{\sigma}\phi(x)|\lesssim (1+|x|)^{-1-\sigma}\, \e^{-(n-1)|x|},\quad x\in \mathbb{H}^n
		\]
		where the implied constant depends on $\phi$ (but not on $x$).		Hence $\Delta^{\sigma}\phi\in L^p$, $1\leq p \leq \infty$. 
	\end{proposition} 
	
	\begin{proof}
Suppose $\phi\in \mathcal{C}_c^{\infty}$ and $\sigma \in (0,1)$. By~\cite[Proposition 2.6 (a)]{BanEtAl}, we have $\Delta^{\sigma}\phi\in \mathcal{C}^{k, \alpha}$ for all $k\in \mathbb{N}\cup\{0\}$ and $\alpha \in (0,1)$. Hence it suffices to show the desired decay only at infinity, which in turn ensures that $\Delta^{\sigma}\phi\in L^p$ for all $p\in[1,\infty]$.

		Let $m>1$ be such that $\operatorname{supp} \phi \subseteq B(o, m)$. We shall prove that the stated decay holds for $|x|>2m$. Notice that if $x\in B(o, 2m)^{c}$ and $z\in B(o, m)$, then  $d(x,z)\geq d(x,o)-d(z,o)>m$ by the triangle inequality, whence $P_0^{\sigma}(z^{-1} x)$ is non-singular according to~\eqref{eq: P0 est}. This means that by~\eqref{eq:singint} we can write, for $|x|>2m$, 
		\begin{align*}
			|\Delta^{\sigma}\phi(x)|&\leq C(\sigma) \int_{B(o,m)}|\phi(z)|\, P_0^{\sigma}(z^{-1}x)\, \dd z\\
			&\leq C(\sigma) \int_{B(o,m)}|\phi(z)|\,(1+d(x,z))^{-1-\sigma}\,\e^{-(n-1)d(x,z)}\, \dd z\\
			&\leq C(\sigma, m) (1+d(x,o))^{-1-\sigma}\,\e^{-(n-1)d(x,o)}\, \|\phi\|_1,
		\end{align*}
which completes the proof.
	\end{proof}
	 
	 As a consequence of Proposition~\ref{Prop 4.3}, if $u\in L_{\sigma}$ and $\phi\in \mathcal{C}_c^{\infty}$, then
	\begin{align*}
			|	\langle  \Delta^{\sigma}\phi, u\rangle |&\leq  \int_{\mathbb{H}^n} |\Delta^{\sigma}\phi(x)|\, |u(x)| \,\dd \mu(x)\\ 
			&\leq C(\phi) \,\int_{\mathbb{H}^n} |u(x)|\,(1+|x|)^{-1-\sigma}\,\e^{-(n-1)|x|} \, \dd \mu(x)<\infty.
		\end{align*}
		Therefore one can define $\Delta^{\sigma}u$ as a distribution via
		\[
		\langle \Delta^{\sigma}u, \phi\rangle:= \langle u, \Delta^{\sigma}\phi\rangle \qquad \forall {\phi} \in \mathcal{C}_c^{\infty}.
		\]	
We conclude this section with the following proposition, which is modeled after \cite[Proposition 2.1.4]{Sil}.
	\begin{proposition}\label{prop: Sil2.1.4}
	Suppose $\sigma \in (0,1)$, and let $f \in L_{\sigma}$ be such that 
	\begin{itemize}
	\item $f \in \mathcal{C}^{0,2\sigma+\varepsilon}$ if $\sigma <1/2$, or
	\item $f\in \mathcal{C}^{1, 2\sigma+\varepsilon-1}$ if $\sigma \geq 1/2$
	\end{itemize}
 for some $\varepsilon>0$. Then $\Delta^{\sigma}f$ is a continuous function and its values are given by the pointwise formula~\eqref{eq:singint}.
	\end{proposition}
	\begin{proof}
		Assume first that $0<\sigma <1/2$  and $f\in \mathcal{C}^{0, 2\sigma+\varepsilon}$, and take an arbitrary open relatively compact set $\Omega$. Then there exists a sequence $(f_k)_{k\in \mathbb{N}} \subseteq \mathcal{C}_c^{\infty}$ uniformly bounded in $\mathcal{C}^{0, 2\sigma+\varepsilon}$, converging uniformly to $f$ in $\Omega$, and converging also to $f$ in the norm of $L_\sigma$: take, e.g., 
		\[
		f_k=f\, \psi(d(\cdot, o)/k)
		\]
		where $\psi\in \mathcal{C}_c^{\infty}(\mathbb{R})$ is such that $\supp \psi\subseteq [-2,2]$ and $\psi\equiv 1$ on $[-1,1]$.
		
		Let $\delta$ be any positive real number. Then there is a radius $r>0$ such that
		$$
		\int_{B_r(o)} \frac{1}{|y|^{n-\varepsilon}} \dd \mu(y) \leq C \int_{0}^{r}s^{-1+\varepsilon}\, \dd s \leq  \frac{\delta}{3M},
		$$
		where $M = \sup[f_k]_{\mathcal{C}^{0,2\sigma+\varepsilon}}$. Next, write
		\begin{align*}
	|\Gamma(-\sigma)|(\Delta^\sigma f_k)(x) 
		&= \int_{B_r(x)} (f_k(x) - f_k(y))P_0^{\sigma}(y^{-1}x) \dd\mu(y) \\
		& \qquad + \int_{ B_r(x)^{c}} (f_k(x) - f_k(y))P_0^{\sigma}(y^{-1}x) \dd \mu(y)		= I_1 + I_2.
	\end{align*}
		For $I_1$, since $P_0^{\sigma}(y^{-1}x)\asymp d(x,y)^{-(n+2\sigma)}$ when $d(x,y)$ is small, the uniform H{\"o}lder boundedness of $(f_k)_{k\in \mathbb{N}}$ and a change of variables yield, since $f \in \mathcal{C}^{0,2\sigma+\varepsilon}$,
	\begin{equation}\label{sigma<121}
		|I_1| \leq \int_{B_r(o)} \frac{M}{d(z,o)^{n-\varepsilon}} \dd \mu(z) \le C\, \frac{\delta}{3}.
\end{equation}
		For $I_2$, recall that $P_0^{\sigma}$ is in $L^1(\mathbb{H}^n \setminus B_r(o))$ by~\eqref{eq: P0 est}. Hence for $x\in \Omega$ and $k$ large enough we have
		\begin{align*}
		\bigg| I_2 - \int_{B_r(x)^{c}} (f(x) - f(y))P_{0}^{\sigma}(y^{-1}x)\, \dd \mu(y) \bigg| &\leq |f_k(x)-f(x)|  \int_{B_r(x)^c}P_0^{\sigma}(y^{-1}x)\, \dd \mu(y)\\ & \;\;+ C \int_{B_r(x)^c}\frac{|f_k(y)-f(y)|}{(1+d(x,y))^{\sigma+1}}\, \e^{-(n-1)d(x,y)}\,\dd \mu(y) \\
		&\le C \frac{\delta}{3},
		\end{align*}
where for the first integral we used that $f_k \to f$ in $\Omega$, while for the second that $f_k \to f$ in $L_\sigma$, the triangle inequality and the fact that $d(x,o)\leq C$ for all $x\in \Omega$.
	Finally, by the H{\"o}lder regularity of $f \in \mathcal{C}^{0,2\sigma+\varepsilon}$
		\begin{equation}\label{sigma<122}
		\left| \int_{B_r(x)} \frac{f(x) - f(y)}{d(x,y)^{n+2\sigma}} \dd \mu(y) \right| \leq C \int_{ B_r(x)}\frac{1}{d(x,y)^{n-\varepsilon}}\, \dd \mu(y)\leq  C\, \frac{\delta}{3}.
		\end{equation}
		It follows that
		$$
		\left| \int_{\mathbb{H}^n} (f_k(x) - f_k(y))P_0^{\sigma}(y^{-1}x) \, \dd \mu(y) - \int_{\mathbb{H}^n} (f(x) - f(y))P_0^{\sigma}(y^{-1}x) \, \dd \mu(y) \right| \leq C\,  \delta,
		$$
		where the constant $C$ may depend on $f$ and $\Omega$, but not on $x\in \Omega$. Since $\delta$ can be chosen arbitrarily small, for all $x\in \Omega$
		\[
		|\Gamma(-\sigma)| \Delta^{\sigma}f_k=	 \int_{\mathbb{H}^n} (f_k - f_k(y))P_0^{\sigma}(y^{-1}\cdot ) \, \dd \mu(y)\rightarrow  \int_{\mathbb{H}^n} (f - f(y))P_0^{\sigma}(y^{-1}\cdot ) \, \dd \mu(y)
		\]
uniformly in $\Omega$. But since for all $\phi \in \mathcal{C}_c^{\infty}$, by Proposition~\ref{Prop 4.3} we have
	\begin{align*}
	|\langle\Delta^{\sigma} f_k - \Delta^{\sigma} f , \phi \rangle| &
	= 		|\langle f_k - f , \Delta^{\sigma} \phi \rangle| \\
	& \leq C    
			\int_{\mathbb{H}^n} |f_k(x)-f(x)|\,(1+|x|)^{-\sigma-1}\, \e^{-(n-1)|x|} \, \dd \mu(x) \rightarrow 0
			\end{align*}
		by the fact that $f_k\rightarrow f$ in $L_{\sigma}$, we deduce that $\langle\Delta^{\sigma}f_k, \phi \rangle\rightarrow \langle\Delta^{\sigma}f, \phi\rangle$  for all $\phi\in \mathcal{C}_c^{\infty}$. Therefore
		\[
	\Delta^{\sigma}f = \int_{\mathbb{H}^n} (f - f(y))P_0^{\sigma}(y^{-1}\cdot ) \, \dd \mu(y)
	\] 
	in  $\Omega$ by uniqueness of the limit, and $\Delta^{\sigma}f$ is continuous in $\Omega$ since it is the uniform limit of continuous functions. By the arbitrariness of $\Omega$, this happens for all $x\in \mathbb{H}^{n}$.
		
		\smallskip
		
	The case  $\sigma \geq 1/2$ and $f\in \mathcal{C}^{1,2\sigma+\varepsilon-1}$ can be treated in a similar manner, as only the estimates~\eqref{sigma<121} and~\eqref{sigma<122} need to be justified differently.
	
	We shall pass to the tangent spaces (see for instance  \cite[Lemmata 5.2 and 10.3]{GIMM24} for a related computation), and in fact argue in an essentially euclidean fashion. Let $T_y\mathbb{H}^n \cong \mathbb{R}^{n}$ be the tangent space at $y$. Let also $F_y: T_y\mathbb{H}^n\rightarrow \mathbb{R}$, $F_y(W):=f(\exp_y(W))$, be the expression of $f$ in normal geodesic coordinates around the point $y$, and  let $V_{y,x}\in T_y\mathbb{H}^n$ be the unique vector transporting $y$ to $x$ through the Riemannian exponential mapping (i.e., $\exp_yV_{y,x}=x$). By means of the change of variables $W=V_{y,x}$ we pass to integration over $T_y\mathbb{H}^n$. Then the radiality of the Jacobian  $J_{\exp_y}(W)=(\sinh |W|/|W|)^{n-1}$ and of $P_0^{\sigma}$ implies that the first order term in the expansion 
	\[
	F_y(W)-F_y(0)=\nabla F_y(0)\cdot W+\mathrm{O}(|W|^{2\sigma+\varepsilon})
	\]
	integrated over $T_y\mathbb{H}^n$ vanishes. Therefore, we can control
	\[
	\left|\int_{ B_r(x)}\frac{f(x)-f(y)}{d(x,y)^{n+2\sigma}}\, \dd \mu(y) \right|\leq C \int_{ B_r(x)}{d(x,y)^{-(n-\varepsilon)}}\, \dd \mu(y),
	\]
which is the analogue of~\eqref{sigma<122}. It is now enough to follow the exact same argument as in the case $\sigma<1/2$ to complete the proof.
	\end{proof}

	\section{The Fujita problem}\label{Sec:3}
Let us consider the nonlinear fractional heat equation, for  $\sigma \in (0,1)$ and $\gamma>1$,
	\begin{equation}\label{fracheat}
		\begin{cases}
			\partial_t u + \Delta^{\sigma}u = h(t)|u|^{\gamma-1}u, \qquad \; t>0,\\
			u(0,\cdot)=f,
		\end{cases} 
	\end{equation}
where $f$ and $h$ are measurable functions on $\mathbb{H}^{n}$ and $[0,+\infty)$ respectively, satisfying the following standing assumptions:
	\begin{itemize}
\myitem{(A1)}\label{A1}$f\geq 0$ and $f\in L^{\infty}\cap \mathcal{C}^{0}$,
\myitem{(A2)}\label{A2} $h \geq 0$ is continuous on $[0,+\infty)$,
	\end{itemize}
	which will be supposed all throughout. In certain statements and proofs, they might actually be slightly weakened; but for the sake of clarity and to avoid overwhelming the discussion, we shall not focus on the minimal assumptions under which each result can be proved. The reader should have no difficulty adapting the proofs to their needs.
	
	We look for \emph{non-negative} solutions in the following sense.
	
\begin{definition}\label{sols}
A \emph{classical solution} to the problem~\eqref{fracheat} is a function $u\colon [0,\tau)\times \mathbb{H}^n\rightarrow \mathbb{R}$ for some $\tau>0$, such that $u(t,\cdot) \in L_{\sigma}$ for all $t\in [0,\tau)$, such that
\begin{itemize}
	\item[(i)] $u\in \mathcal{C}([0, \tau)\times \mathbb{H}^n)$;
	\item[(ii)] $\partial_{t}u, \Delta^{\sigma}u\in \mathcal{C}((0,\tau)\times \mathbb{H}^n)$;
	\item[(iii)] $\Delta^{\sigma}u$ is given by the pointwise formula~\eqref{eq:singint} for $0<t<\tau$,
\end{itemize}
and which satisfies~\eqref{fracheat} on $[0, \tau)\times\mathbb{H}^n$ in the pointwise sense.

A \emph{mild solution} to the problem~\eqref{fracheat} is a function $u\colon [0,\tau)\times \mathbb{H}^n\rightarrow \mathbb{R}$ for some $\tau>0$ such that (i) holds and which satisfies
\begin{equation}\label{eq:mildsol}
u(t,\cdot ) = \e^{-t\Delta^{\sigma}}f + \int_{0}^{{t}} h(s) \, \e^{-(t-s)\Delta^{\sigma}}(|u|^{\gamma-1}u) \, \dd s
\end{equation}
on $[0, \tau)\times\mathbb{H}^n$ in the pointwise sense.

We say that a solution blows up in finite time whenever there exists $\tau> 0$ such that 
\[
\lim_{t\rightarrow \tau^{-}}\|u(t, \cdot)\|_{\infty}=+\infty,
\]
and $\tau$ is called blow-up time for $u$. Otherwise, if $u(t, \cdot) \in L^{\infty}(\mathbb{H}^n)$ for all $t > 0$, then we say that it is global. If $u\in L^{\infty}([0,\tau)\times \mathbb{H}^{n})$, then we say that $u$ is bounded.
\end{definition}	


The following result makes use of standard arguments, see e.g.~\cite[Proposition 51.40]{QS} or~\cite[Proposition 4.1]{Punzo2012}. We include a proof for the sake of completeness.

\begin{proposition}\label{prop:unique}
The following holds:
		\begin{itemize}
	\item[\emph{(i)}] there exists $\tau> 0$ such that there is a unique bounded mild solution to~\eqref{fracheat} in $[0,\tau]\times \mathbb{H}^{n}$;
	\item[\emph{(ii)}] such a solution can be uniquely extended to a maximum interval $[0, \tau_{\max})$. If $\tau_{\max}$ is finite, then the solution blows up at time $\tau_{\max}$;
	\item[\emph{(iii)}] such a mild solution is also a classical solution.
		\end{itemize}
	\end{proposition}
	
\begin{proof}
Let us write $\delta=\|f\|_{L^{\infty} }$, and for $\tau>0$ fixed, consider the  following Banach space 
\[
	\Theta_s^{\sigma}:=\left\{v \in \mathcal{C}([0, \tau] \times\mathbb{H}^n) \cap L^{\infty}([0, \tau] \times\mathbb{H}^n) \colon \|v\|_{\Theta_s^{\sigma}} \leq 2 \delta, \quad  v(0,\cdot)=f\right\},
\]
	where
	$$
	\|v\|_{\Theta_s^{\sigma}}=\sup_{t\in [0,\tau]}\|v(t, \cdot )\|_{\infty}=\|v\|_{L^{\infty}([0, \tau] \times \mathbb{H}^n)} .
	$$
	For any $v \in \Theta_s^{\sigma}$, we define a map
	\begin{equation}\label{eq: (3.3)}
	\Phi(v)(t,\cdot ):=\e^{-t\Delta^{\sigma}} f+\int_0^t h(s)\,\e^{-(t-s)\Delta^{\sigma}}v^{\gamma}(s,\cdot)\,\dd s, \qquad t\in [0,\tau].
	\end{equation}
	We first show that if $v \in \Theta_{{\tau}}^{\sigma}$, then $\Phi(v) \in \Theta_{{\tau}}^{\sigma}$. By~\eqref{eq: (3.3)}
	\begin{align*}
	\|\Phi(v)(t,\cdot )\|_{L^{\infty} }
	&\leq \|f\|_{\infty}+\int_{0}^{t}h(s)\,\|v(s,\cdot )\|^{\gamma}_{\infty} \,\dd s
	\end{align*}
	so
	\begin{align*}
		\|\Phi(v)\|_{\Theta_s ^{\sigma}}&= \sup_{t\in [0,\tau]}\|\Phi(v)(t,\cdot )\|_{L^{\infty} }\\
		&\leq \|f\|_{\infty}+\sup_{t\in [0, \tau]}\int_{0}^{t} h(s)\, \|v(s,\cdot )\|^{\gamma}_{\infty} \,\dd s \\
		&\leq \delta+2^{\gamma}\delta^{\gamma} \sup_{t\in [0,\tau]}\int_{0}^{t}h(s)\dd s \leq \delta+2^{\gamma}\delta^{\gamma}\int_{0}^{\tau}h(s)\dd s.
		\end{align*}
Then for $\tau>0$, sufficiently small, one can get
	$$
	\|\Phi(v)\|_{\Theta_s^\sigma } \leq 2 \delta.
	$$
Moreover, $\Phi(v) \in \mathcal{C}([0,\tau]\times \mathbb{H}^{n})$ (continuity of the first term, in particular at $t=0$, follows since $f$ is bounded and continuous; the second term is easily checked to be continuous being $v\in \Theta_s^{\sigma}$, see also \cite[Eq. (51.109)]{QS}).

This shows that $\Phi(v) \in \Theta_s^{\sigma}$. We now proceed to show that for sufficiently small $\tau>0$, the map $\Phi$ is a contraction map on $\Theta_s^{\sigma}$. For any  $v_1, v_2 \in \Theta_s^{\sigma}$, we have
\[
\Phi(v_1)(t, \cdot )-\Phi(v_2)(t,\cdot)=\int_{0}^{t}h(s)\,\e^{-(t-s)\Delta^{\sigma}}(v_1^{\gamma}(s,\cdot)-v_2^{\gamma}(s,\cdot))\, \dd s,
\]
therefore for any $0<t<\tau$,
\begin{align*}
	\|\Phi(v_1)(t,\cdot)-\Phi(v_2)(t,\cdot)\|_{{\infty} }&\leq \int_0^t h(s)\|\e^{-(t-s)\Delta^{\sigma}}(v_1^{\gamma}(s,\cdot)-v_2^{\gamma}(s,\cdot))\|_{\infty}\,\dd s\\
	&\leq C(\gamma)\delta^{\gamma-1}\int_0^t h(s)\|v_1(s,\cdot)-v_2(s,\cdot)\|_{\infty}\,\dd s\\
	&\leq  C(\gamma)\delta^{\gamma-1} \, \sup_{s\in [0,\tau]}\|v_1(s,\cdot)-v_2(s,\cdot)\|_{\infty}\int_0^{ \tau} h(s)\,\dd s
\end{align*}
which implies that
\[
	\|\Phi(v_1)-\Phi(v_2)\|_{\Theta_s^{\sigma}} \leq \frac{1}{2}\|v_1-v_2\|_{\Theta_T^{\sigma}}
\]
for $\tau>0$ sufficiently small. Hence, the Banach fixed point theorem ensures the existence and uniqueness of a local mild solution to~\eqref{fracheat} in $ \mathcal{C}([0, \tau] \times\mathbb{H}^n) \cap L^{\infty}([0, \tau] \times\mathbb{H}^n)$.

We now show (ii), i.e.\ the blow up if the maximal time of existence is finite. Let
$$
\tau_{\max }:=\sup \{\tau>0 \colon \text{the Cauchy problem}~\eqref{fracheat}  \text { admits a mild solution in } [0, \tau] \times \mathbb{H}^{n}\}.
$$
Assume that $\tau_{\max}$ is finite and that for each $t \in\left[0, \tau_{\max}\right)$, we have
$$
\|u(t, \cdot )\|_{\infty} \leq C,
$$
for some constant $C>0$. Fix $t^* \in\left(\tau_{\max}/ 2, \tau_{\max}\right)$ and for $\tilde{t} \in\left(0, \tau_{\max }\right)$, let us consider a space
$$
\mathcal{M}:=\{v \in  \mathcal{C}([0, \tilde{t}] \times\mathbb{H}^n) \cap L^{\infty}([0, \tilde{t}] \times\mathbb{H}^n)\colon  \|v\|_{L^{\infty}([0,\tilde{t}]\times \mathbb{H}^n)}<2C, \quad  v(0,\cdot )=u(t^*,\cdot)\}.
$$
With arguments similar to those already used, we define a map  $\Phi^{\prime}$ on $\mathcal{M} $ given by
$$
\Phi^{\prime}(v)(t,\cdot ):=\e^{-t\Delta^{\sigma}} u(t^*,\cdot )+\int_0^t \e^{-(t-s)t\Delta^{\sigma}}v^{\gamma}(s,\cdot) \dd s,
$$
for $t \in[0, \tilde{t}]$. One can see that it is a contraction map on $\mathcal{M} $, so the Banach fixed point theorem ensures the existence of a fixed point of $\mathcal{M}$, say $v$. We now take $t^*$ such that
$$
\tilde{t}+t^*>\tau_{\max }
$$
and consider
\begin{equation}\label{Eq. 3.7}
\bar{u}(t,\cdot ):= \begin{cases}u(t,\cdot ), & \text { for } 0 \leq t \leq t^*, \\ v(t-t^*,\cdot ) & \text { for } t^* \leq t \leq t^*+\tilde{t}.\end{cases}
\end{equation}
Observe that $\bar{u} \in \mathcal{C}([0,  t^*+\tilde{t}] \times\mathbb{H}^n) \cap L^{\infty}([0,  t^*+\tilde{t}] \times\mathbb{H}^n)$ is a solution of~\eqref{fracheat}. This, together with~\eqref{Eq. 3.7}, contradicts the definition of $\tau_{\max}$. Therefore, we get that
$$
\|u(t,\cdot )\|_{\infty} \to + \infty  \quad \text { as } t \to \tau_{\max },
$$
and this concludes the proof of (ii).

It remains to show that $u$ satisfies the required regularity to be a classical solution too. Since
\[
	u(t,x)=\int_{\mathbb{H}^n}P_t^{\sigma}(d(x,y))f(y)\, \dd \mu(y)+\int_0^{t}h(s)\int_{\mathbb{H}^n} P_{t-s}^{\sigma}(d(x,y))(|u|^{\gamma-1}u(s,y)) \,\dd \mu(y)\dd s
\]
for $t\in [0,\tau)$ and $x\in \mathbb{H}^{n}$, in fact $u\in C^{\infty}((0,\tau)\times \mathbb{H}^n)$, see~\cite[Theorem 3.2]{DPF25}. Finally, the fact that $\Delta^{\sigma}u$ is given by the pointwise formula~\eqref{eq:singint} for $0<t<\tau$ holds by the boundedness of $u$ which guarantees that $u(t,\cdot) \in L_{\sigma}$ for all $t\in (0,\tau)$, the smoothness of $u$ in space and Proposition~\ref{prop: Sil2.1.4}.
\end{proof}

One of our main results is the aforementioned Theorem~\ref{thm:A}, which we re-state for the reader's convenience. 
	\begin{theorem}\label{mainconj}
		Suppose $h(t)=\e^{\beta t}$ for some $\beta>0$ and $\sigma \in (0,1)$. Define $\gamma^{*} = 1 + \frac{\beta}{\lambda_{0}^{\sigma}}$. Then the following holds.
		\begin{itemize}
			\item If $1<\gamma< \gamma^{*}$, then every nontrivial nonnegative mild solution to the problem~\eqref{fracheat} blows up in finite time.
			\item If $\gamma\geq \gamma^{*}$, then there exists a nontrivial nonnegative classical global solution to the problem~\eqref{fracheat} for sufficiently small initial data $f$.
		\end{itemize}
	\end{theorem}
We will specify in due course what \emph{sufficiently small} means. The proof of Theorem~\ref{mainconj} will be discussed in the following sections, and will be split into Propositions~\ref{prop:blowup},~\ref{prop:oldFujita} and~\ref{prop:criticalFujita}.

	\section{Blow-up in finite time}\label{Sec:4}
	
In this section we prove the blow-up part of Theorem~\ref{mainconj}, namely the following.
	
	\begin{proposition}\label{prop:blowup}
		Suppose $h(t) \geq \e^{\beta t}$ for some $\beta>0$, and $\sigma \in (0,1)$. If $1<\gamma<\gamma^{*}$, then every nontrivial nonnegative mild solution to the problem~\eqref{fracheat} blows up in finite time.
	\end{proposition}
	

	We begin with a lemma.
	
	\begin{lemma}\label{lemma:existence}
		Suppose $h(t) \geq  \e^{\beta t}$ for some $\beta>0$, $\sigma \in (0,1)$ and $\gamma>1$. Suppose a global mild solution to the problem~\eqref{fracheat} exists. Then there exists a constant $C>0$, depending only on $\beta, \sigma$ and $\gamma$, such that
		\[
		\|  \e^{-t \Delta^{\sigma}}f \|_{\infty} \leq C \, \e^{-\frac{\beta}{\gamma-1} t}, \qquad t \geq \beta^{-1}\log(2).
		\]
	\end{lemma}
	\begin{proof}
		Suppose that $u$ is a non-negative global mild solution to the problem~\eqref{fracheat}. Since $u$ satisfies~\eqref{eq:mildsol} and $f \geq 0$,
			\begin{equation}\label{2ineq}
			u(t,\cdot)\geq  \e^{-t\Delta^{\sigma}}f, \qquad u(t,\cdot) \geq \int_{0}^{{t}} h(s) \, \e^{-(t-s)\Delta^{\sigma}}(u(s,\cdot)^{\gamma}) \, \dd s.
		\end{equation}
By Jensen's inequality (recall that $\|P_{t}^{\sigma}\|_{1}=1$ whence, by left-invariance and symmetry, $\int_{\mathbb{H}^{n}} P_{t}^{\sigma}(y^{-1}x) \, dy =1$ for all $x$), the first inequality in~\eqref{2ineq} leads to
		\[
		\e^{-(t-s)\Delta^{\sigma}}(u(\cdot ,s)^{\gamma}) \geq \e^{-(t-s)\Delta^{\sigma}}(\e^{-s\Delta^{\sigma}}f)^{\gamma} \geq  (\e^{-t\Delta^{\sigma}} f)^{\gamma},
		\]
whence by the second inequality in~\eqref{2ineq} 
		\[
		u(\cdot,t) \geq \int_{0}^{t} h(s) (\e^{-t\Delta^{\sigma}} f)^{\gamma} \, \dd s = H_{1}(t)  (\e^{-t\Delta^{\sigma}} f)^{\gamma},
		\]
		where $H_{1}(t) = \int_{0}^{t} h(s)\, \dd s$. With the exact same reasoning
		\begin{align*}
			u(\cdot ,t) 
			&\geq \int_{0}^{t} h(s) \e^{-(t-s)\Delta^{\sigma}} (H_{1}(s)(\e^{-s\Delta^{\sigma}} f)^{\gamma})^{\gamma}\, \dd s\\
			& \geq \int_{0}^{t} h(s) H_{1}^{\gamma}(s)(\e^{-t\Delta^{\sigma}} f)^{\gamma^{2}}\, \dd s\\
			& = H_{2}(t) (\e^{-t\Delta^{\sigma}} f)^{\gamma^{2}},
		\end{align*}
		where $H_{2}(t) = \int_{0}^{t} h(s)H_{1}^{\gamma}(s)\, \dd s$. Inductively, one gets
		\[
		u(\cdot ,t)  \geq H_{k}(t) (\e^{-t\Delta^{\sigma}} f)^{\gamma^{k}}, \qquad H_{k}(t) = \int_{0}^{t} h(s)H_{k-1}^{\gamma}(s)\, \dd s, \qquad k\geq 2.
		\]
		Assume now that $h(t) \geq \e^{\beta t}$. Then, for all $t  \geq \beta^{-1}\log(2)$,
		\[
		H_{1}(t)  \geq  \int_{0}^{t} \e^{\beta s} \, \dd s =   \frac{1}{\beta}(\e^{\beta t} -1) \geq  \frac{1}{2\beta} \e^{\beta t}.
		\]
		Thus, for $t  \geq \beta^{-1}\log(2)$,
		\begin{align*}
			H_{2}(t)
			&\geq \int_{0}^{t} \e^{\beta s}  \bigg( \frac{1}{2\beta} \e^{\beta s} \bigg)^{\gamma} \, \dd s\\
			& =  \frac{1}{(2\beta)^{\gamma}} \int_{0}^{t} \e^{\beta (1+\gamma)s} \, \dd s =    \frac{1}{(2\beta)^{\gamma}} \frac{1}{\beta(1+\gamma)}( \e^{\beta(1+\gamma) t}-1) \geq   \frac{1}{(2\beta)^{\gamma}} \frac{1}{2\beta(1+\gamma)} \e^{\beta(1+\gamma) t}
		\end{align*}
		and inductively
		\[
		H_{k}(t) \geq \frac{1}{(2\beta)^{1+\dots + \gamma^{k-1}}} \frac{1}{(1+\gamma)^{\gamma^{k-2}}(1+\gamma+\gamma^{2})^{\gamma^{k-3}} \cdots (1+ \gamma + \dots + \gamma^{k-1})} \e^{\beta t (1 + \dots + \gamma^{k-1})},
		\]
		whence
		\begin{multline*}
			u( \cdot ,t)  \geq \frac{1}{(2\beta)^{1+\dots + \gamma^{k-1}}} \frac{1}{(1+\gamma)^{\gamma^{k-2}}} \frac{1}{(1+\gamma+\gamma^{2})^{\gamma^{k-3}}} \dots \frac{1}{(1+ \gamma + \dots + \gamma^{k-1})} \\
			\times \e^{\beta t (1+ \gamma + \dots + \gamma^{k-1})} (\e^{-t\Delta^{\sigma}} f)^{\gamma^{k}}.
		\end{multline*}
		This implies
		\begin{align*}
			\e^{ \beta t \frac{\gamma^{k}-1}{\gamma^{k}(\gamma-1)}} \e^{-t\Delta^{\sigma}} f  &=\e^{\beta t \frac{1+ \dots + \gamma^{k-1}}{\gamma^{k}}}  \e^{-t\Delta^{\sigma}} f\\
			& \leq (2\beta)^{\frac{1+\dots + \gamma^{k-1}}{\gamma^{k}}} (1+\gamma)^{\gamma^{-2}} \cdots (1+ \gamma + \dots + \gamma^{k-1})^{\gamma^{-k}} u(\cdot, t)^{\gamma^{-k}}\\
			& = (2\beta)^{\frac{\gamma^{k-1}}{(\gamma-1)\gamma^{k}}}  u(\cdot, t)^{\gamma^{-k}} \prod_{j=2}^{k} \bigg( \frac{\gamma^{j}-1}{\gamma-1} \bigg)^{\gamma^{-j}}.
		\end{align*}
		Taking the limit for $k\to \infty$, one gets, for all $t  \geq \beta^{-1}\log(2)$
		\[
		(0\leq ) \quad \e^{\frac{\beta }{\gamma-1}t} \e^{-t\Delta^{\sigma}} f \leq C
		\]
		which completes the proof of the lemma.
	\end{proof}
	
	\begin{proof}[Proof of Proposition~\ref{prop:blowup}]
We first observe that it is enough to prove that 
		\begin{equation}\label{claim:div}
			\lim_{t \to \infty} \e^{\frac{\beta}{\gamma-1} t } \|\e^{-t\Delta^{\sigma}} f \|_{\infty} \to +\infty
		\end{equation}
		whenever $\frac{\beta}{\gamma-1} > \lambda_{0}^{\sigma}$, since this combined with Lemma~\ref{lemma:existence} proves the statement.	
			
		Suppose $t\geq \beta^{-1}\log 2$. Then by~\eqref{P kernel estimates} we have
		\begin{align*}
			\e^{-t\Delta^{\sigma}} f(x) 
			&= \int_G P_{t}^{\sigma}(y) \,f(xy^{-1})\, \dd y \\
			& \geq  \int_{\{|y|\leq \sqrt{t}\}} P_{t}^{\sigma}(y) \, f(xy^{-1}) \, \dd y\\
			& \gtrsim  \int_{\{|y|\leq \sqrt{t}\}} \varphi_{0}(y)\, t^{\frac{1}{2-2\sigma}} \,(t+|y|)^{-\frac{3}{2}-\frac{1}{2-2\sigma}} \,\e^{-\lambda_0^{\sigma} t} \, f(xy^{-1}) \, \dd y. 
		\end{align*}
		Since $$t+|y|\asymp t \quad \text{ for } |y|\leq\sqrt{t}, \, t\geq \beta^{-1}\ln2,$$
		it follows that  
		\begin{align*}
			t^{\frac{3}{2}} \e^{ \lambda_{0}^{\sigma} t} \|\e^{-t\Delta^{\sigma}} f \|_{\infty}
			&\geq \sup_{x\in G} \int_{\{|y|\leq \sqrt{t}\}} \varphi_{0}(y)  \, f(xy^{-1}) \, \dd y  \\
			& \geq \sup_{x\in G}\int_{B(o,\frac{1}{2}\sqrt{\beta^{-1}\ln2})}f(xy^{-1})\, \varphi_{0}(y)\, \dd y\\
			& = \sup_{x\in G}(f\ast\varphi_0\mathbbm{1}_{B(o,\frac{1}{2}\sqrt{\beta^{-1}\ln2})})(x),
		\end{align*}
from which~\eqref{claim:div} follows.
	\end{proof}

	\section{Global existence}\label{Sec:5}

\begin{definition}
A function $u\colon [0,\tau)\times \mathbb{H}^n\rightarrow \mathbb{R}$  such that $u(t,\cdot) \in L_{\sigma}$ for all $t\in (0,\tau)$, and satisfying (i)--(iii) of Definition~\ref{sols} is said to be a \textit{classical supersolution (subsolution)} of~\eqref{fracheat} if
\[
		\begin{cases}
			\partial_t u + \Delta^{\sigma}u \geq  (\leq) \e^{\beta t}u^{\gamma}, \qquad &t\in(0, \tau),\, x\in \mathbb{H}^n\\
			u(x,0)\geq  (\leq)f(x), \qquad &x\in \mathbb{H}^n.
		\end{cases} 
\]
We say that $u$ is bounded if  $u\in L^{\infty}([0,\tau) \times \mathbb{H}^{n})$.
\end{definition}

	\subsection{A comparison principle}
\begin{theorem}\label{thm: comparison}
		Let $\overline{u}$ and $\underline{u}$ be a bounded nonnegative classical supersolution and a subsolution of~\eqref{fracheat}, respectively. Then $\underline{u}\leq \overline{u}$.
	\end{theorem}

From this it follows
	
\begin{corollary}\label{faith1}
Suppose a nonnegative bounded global supersolution to~\eqref{fracheat} exists. Then there exists a nonnegative bounded global classical solution to~\eqref{fracheat}.
\end{corollary} 

\begin{proof}
By Proposition~\ref{prop:unique} (i), there exists a local in time, bounded, mild (thus classical by Proposition~\ref{prop:unique} (iii)) solution to~\eqref{fracheat}, which by Theorem~\ref{thm: comparison} is bounded above by the given supersolution. By Proposition \ref{prop:unique} (ii), then, such local solution can be extended to the maximum interval $[0, \infty)$. This proves the claim.
\end{proof}

To prove Theorem~\ref{thm: comparison} we shall need an auxiliary result: i.e., that the positive smooth radial function
\[
\phi(x)=\log (2+|x|^2)
\]
is unbounded at infinity, which is clear, and that $\Delta^{\sigma}\phi\in L^{\infty}$. 

With the use of polar coordinates, it can be easily verified that the function belongs to $L_{\sigma}$, as defined in~\eqref{eq: class Lsigma}. Hence $\Delta^{\sigma}\phi$ makes sense as a distribution. We next show that it is in fact a continuous bounded function.

\begin{lemma}\label{lemmalog}
$\Delta^{\sigma}\phi\in \mathcal{C}^{0} \cap L^{\infty}$ for all $\sigma \in (0,1)$. 
\end{lemma}

\begin{proof}
Since $\phi$ smooth and in $L_{\sigma}$, Proposition~\ref{prop: Sil2.1.4} yields that $\Delta^{\sigma}\phi$ is continuous on $\mathbb{H}^n$ and we may use the integral representation~\eqref{eq:singint} to express $\Delta^{\sigma}\phi$.

For all $x_1, x_2 \in \mathbb{H}^n$ and $d=d(x_1,x_2)$, notice that first and second derivatives of $\phi(d(\cdot,x_2))$  are uniformly bounded on $x_2$. Indeed,
$\phi'(d)=(2d)/(2+d^2)$, $\phi''(d)$ is bounded, and the first and second derivatives of the distance satisfy~\cite[Theorem 1.3]{BSa}. 
Hence, using a Taylor expansion of order $2$ (there is no contribution of first order terms in the integral by parity) and \eqref{eq: P0 est} we have for any $x\in \mathbb{H}^n$, 
$$\left| \int_{d(x,y)<1}(\phi(x)-\phi(y))\,P_0^{\sigma}(d(x,y)) \, \dd \mu (y)\right|\leq C(n, \sigma).
$$

Therefore, it remains to show that 
\[
I(x)=\int_{|y^{-1}x|\geq 1}(\log(2+|x|^2)-\log(2+|y|^2))\, P_0^{\sigma}(y^{-1}x) \, \dd y
\]
is bounded on $\mathbb{H}^n$. Owing to the continuity of $\Delta^{\sigma}\phi$, we may assume that $|x|\geq 1$. Let us split $I=\int_{R_1}+\int_{R_2}+\int_{R_3}=I_{1}+I_{2}+I_{3}$, where 
\begin{align*}
	R_1&=B(x,1)^c\cap \{y: \, |y|\geq 2|x|\}, \\
	R_2&=B(x,1)^c\cap \{y: \, |x|\geq 2|y|\}, \\
	R_3&=B(x,1)^c\cap \{y: \, |x|/2\leq |y| \leq 2|x|\}.
\end{align*}
On $R_1$, we have $\log \left(2+|x|^2\right) \leq \log \left(2+|y|^2\right)$ and 
$\frac{1}{2}|y| \leqslant\left|y^{-1} x\right| \leqslant \frac{3}{2}|y|$
by the triangle inequality. Hence using the estimates~\eqref{eq: P0 est} of $P_0^{\sigma}$ at infinity,
\begin{align*}
	|I_{1}(x)|
	\lesssim \int_{|y^{-1} x|>1} \log \left(2+4|y^{-1}x|^2\right)\,\left(1+\left|y^{-1} x\right|\right)^{-\sigma-1} \,\e^{-(n-1)\left|y^{-1} x\right|}\, \dd y  \lesssim 1.
\end{align*}
On $R_2$, $\log \left(2+|y|^2\right) \leq \log \left(2+|x|^2\right)$ and 
$
\frac{1}{2}|x| \leqslant\left|y^{-1} x\right| \leqslant \frac{3}{2}|x|,
$
thus
\begin{align*}
	|I_{2}(x)|
	\lesssim \log(2+4|x|^2)\, (2+|x|)^{-\frac{\sigma}{2}}\int_{|y^{-1} x|>1} \left(1+|y^{-1} x|\right)^{-\frac{\sigma}{2}-1} \,\e^{-(n-1)|y^{-1} x|}\, \dd y  \lesssim 1.
\end{align*}
Finally, on $R_3$, $|\log(2+|x|^2)-\log(2+|y|^2)|=|\log((2+|x|^2)/(2+|y|^2))|\lesssim 1$, hence 
\begin{align*}
	|I_{3}(x)|
	\lesssim \int_{|y^{-1} x|>1} \left(1+|y^{-1} x|\right)^{-\sigma-1} \,\e^{-(n-1)|y^{-1} x|}\, \dd y  \lesssim 1.
\end{align*}
The proof is now complete.
\end{proof}

We are now ready to prove the comparison principle given by Theorem~\ref{thm: comparison}.

\begin{proof}[Proof of Theorem~\ref{thm: comparison}]
Suppose $\underline{u},\overline{u}$ are bounded sub/supersolutions respectively, defined on $\mathbb{H}^{n} \times [0,\tau)$ with $\tau \in (0,+\infty]$. We shall prove that 
\[
\underline{u} (x,t) \leq \overline{u} (x,t) \qquad \forall \, (x,t) \in \mathbb{H}^{n} \times [0,T), \; \; \forall \, T \in (0,\tau).
\]

The proof follows a classical argument, see e.g.~\cite[Theorem 3.3]{DPF25}, but we give the full arguments for the sake of completeness.
First, according to Lemma~\ref{lemmalog}, we note that $\phi(x)=\log(2+|x|^2)$  is a positive regular function which satisfies
	$$
	\lim _{|x| \rightarrow \infty} \phi(x)=+\infty, \quad |\Delta^{\sigma} \phi(x)| \leq K.
	$$
	Now, let us define $w=\underline{u}-\overline{u}-\varepsilon \e^{\mu t} \phi(x)$, for any $\varepsilon, \mu>0$. Since $\overline{u}$ and $\underline{u}$ are bounded and $w \leq \underline{u}-\overline{u}-\varepsilon \phi(x)$, there is $R>0$ independent of $\mu$ such that
	$$
	w(x, t)<0, \quad |x|>R, \;\; t\in [0, T).
	$$
It will be enough to prove, then, that $w(x, t)\leq 0$ for $|x|\leq R$ and $t\in [0, T)$ if $\mu$ is sufficiently large (independently of $\varepsilon$).  If this is the case, indeed, taking the limit $\varepsilon \rightarrow 0$ gives the desired result.

Notice that $w(\cdot, 0) \leq-\frac12\varepsilon$ in $\mathbb{H}^n$. We argue by contradiction, and suppose that there exists
\[
t_{0} := \inf\{t>0\colon \exists \, x\in \mathbb{H}^{n}, \, |x|\leq R, \mbox{ s.t.\ }w(x,t)=0 \}.
\]
Then, by continuity of $w$ and by compactness, at $t=t_{0}$ there is at least one $x_{0}\in \mathbb{H}^{n}$ with $|x_{0}|\leq R$ such that $w(x_{0},t_{0})=0$. By definition of $t_{0}$ we have $w(x, t)<0$ for all $|x|\leq R$ and $t\in [0,t_{0})$,  so $w(x,t_0) \leq 0$ for all $|x|\leq R$ by continuity. Thus
\[
w_t\left(x_0, t_0\right) \geq 0, \quad \Delta^{\sigma} w\left(x_0, t_0\right) \geq 0,
\]
by the pointwise formula~\eqref{eq:singint} for $\Delta^{\sigma} w$ (notice that $w(x_0,t_0)-w(x, t_0)=-w(x, t_0)\geq 0$ for all $x\in \mathbb{H}^{n}$).

On the other hand, since $\underline{u}$ and $\overline{u}$ are bounded sub/supersolutions, for all $(x,t) \in \mathbb{H}^{n} \times [0,T)$ one has
	\begin{align*}
	w_t(x,t)+\Delta^{\sigma} w(x,t) & \leq h(t)(\underline{u}^{\gamma}(x,t)-\overline{u}^{\gamma}(x,t))-\varepsilon \mu \e^{\mu t} \phi(x)-\varepsilon \e^{\mu t} \Delta^{\sigma} \phi (x)\\
	& \leq C h(t)  |\underline{u}(x,t)-\overline{u}(x,t)|-\varepsilon \mu \e^{\mu t} \phi(x)+\varepsilon \e^{\mu t} K \\
	& \leq C \e^{\beta t} \left(|w|(x,t)+\varepsilon \e^{\mu t} \phi(x)\right)-\varepsilon \mu \e^{\mu t} \phi(x)+\varepsilon \e^{\mu t} K.
\end{align*}
Therefore
$$
0 \leq w_t\left(x_0, t_0\right)+\Delta^{\sigma} w\left(x_0, t_0\right) \leq \varepsilon \e^{\mu t}\left(C\e^{\beta t} \phi\left(x_0\right)-\mu \phi\left(x_0\right)+K\right)<0
$$
if $\mu>\frac{Ce^{\beta T}\phi(x_0)+K}{\phi(x_0)}$, and this is a contradiction.	
\end{proof}
	
\subsection{Subcritical regime and mild solutions} Let us start by recalling that $\e^{-t\Delta^{\sigma}}$ is strongly continuous on the space of continuous functions vanishing at infinity $\mathcal{C}_{0}$, see e.g.~\cite[p.\ 260]{Y}. The next result is inspired by~\cite[Theorem 3]{Weissler1981}.
	\begin{lemma}\label{prop:ex} 
		If $f\in \mathcal{C}_{0}$ and
		\begin{equation}\label{lem2_for1}
			\int_{0}^{\infty} h(s)\|\e^{-s\Delta^{\sigma}}  f\|_\infty^{\gamma-1}ds<\frac{1}{\gamma-1},
		\end{equation}
		then there exists a non-negative global mild solution $u$ to~\eqref{fracheat}.
	\end{lemma}

	\begin{proof}
		Define
		\[
		\omega(t) = \bigg(1-(\gamma-1)\int_{0}^{t} h(s) \|  \e^{-s\Delta^{\sigma}}  f\|^{\gamma-1}_\infty\, ds\bigg)^{-\frac{1}{\gamma-1}},
		\]
		and observe that by~\eqref{lem2_for1}, $\omega$ is well defined, bounded, and that $\omega(0)=1$ and $\omega(t)>1$ for all $t>0$. Moreover, a simple computation shows that $\omega'(t)=h(t)\|  \e^{-t\Delta^{\sigma}}f\|^{\gamma-1}_\infty  \omega(t)^{\gamma}$, whence \begin{equation}\label{eq_w(t)_1}
			\omega(t)=1+\int_{0}^{t}h(s)\omega(s)^{\gamma} \| \e^{-s\Delta^{\sigma}}  f\|^{\gamma-1}_\infty \, ds.
		\end{equation}
		Assume now that $v\colon [0,\infty) \times \mathbb{H}^{n}  \to \R$ is a function satisfying
		\begin{equation}\label{choiceu}
			\e^{-t\Delta^{\sigma}}  f\leq v(t,\cdot )\leq \omega(t)\e^{-t\Delta^{\sigma}}   f
		\end{equation}
		for all $t> 0$, and $v(0,\cdot )=f$; whence $v(t,\cdot )\in L^{\infty}$ for $t\geq 0$. Define
		\[
		\Phi(v)(t,\cdot) = \e^{-t\Delta^{\sigma}}  f+\int_{0}^{t} h(s)\, \e^{-(t-s)\Delta^{\sigma}} (v(s,\cdot))^{\gamma}ds .
		\]
		By the upper bound in~\eqref{choiceu}, we get
	\[
			\begin{split}
				\e^{-(t-s)\Delta^{\sigma}}(v(s,\cdot))^{\gamma} &\leq \e^{-(t-s)\Delta^{\sigma}}(\omega(s)  \e^{-s \Delta^{\sigma}}   f)^{\gamma} \\
				&\leq  \omega(s)^{\gamma} \| \e^{-s\Delta^{\sigma}}  f\|_{{\infty}}^{\gamma-1}\e^{-t\Delta^{\sigma}}  f.
			\end{split}
\]		Therefore
		\begin{equation*}\label{eq_w(t)_2}
			\begin{split}
				\Phi(v)(t,\cdot)
				&\leq     \e^{-t\Delta^{\sigma}}   f\bigg(1+\int_{0}^{t} h(s) \omega(s)^{\gamma} \|   \e^{-s\Delta^{\sigma}}  f\|^{\gamma-1}_{{\infty}}\, ds\bigg )= \omega(t)\e^{-t\Delta^{\sigma}}  f,
			\end{split}
		\end{equation*}
		the last equality by~\eqref{eq_w(t)_1}. The above proves
		\begin{equation}\label{calF}
			\mbox{$v$ satisfies~\eqref{choiceu}} \quad \Longrightarrow \quad  \e^{-t\Delta^{\sigma}}   f\leq \Phi(v)(t,\cdot)\leq     \omega(t)\e^{-t\Delta^{\sigma}}  f, \quad t>0.
		\end{equation}
		Consider now the sequence of functions $(u_{k}(t,\cdot))_{k}$ defined by
		\[
		u_{0}(t,\cdot)= \e^{-t\Delta^{\sigma}}  f, \qquad u_{k+1}(t,\cdot)=\Phi(u_{k})(t,\cdot).
		\]
		We show that this sequence converges to the desired solution. Since $u_0$ satisfies~\eqref{choiceu}, by~\eqref{calF}  we get 
		\begin{equation}\label{bounduk}
			\e^{-t\Delta^{\sigma}}  f\leq u_{k}(t, \cdot )\leq \omega(t) \e^{-t\Delta^{\sigma}}   f
		\end{equation}
		for all $k$. We also note that if $0\leq v\leq w$, then $0\leq \Phi(v) \leq \Phi(w)$, whence $0\leq u_{k} \leq u_{k+1}$. For each $t>0$, thus, $u_{k}(t, \cdot )$ is a non-decreasing sequence of non-negative functions, and by the monotone convergence theorem the sequence $u_{k}(t,\cdot )$ converges pointwise to a function which we call $u(t, \cdot)$. By~\eqref{bounduk}, moreover, 
		\begin{equation}\label{upplowu}
			\e^{-t\Delta^{\sigma}}  f\leq u(t,\cdot )\leq \omega(t) \e^{-t\Delta^{\sigma}}   f.
		\end{equation}
		We now prove that $u(t,\cdot )$ is a global mild solution of~\eqref{fracheat}.
Observe that the functions $\e^{-(t-s)\Delta^{\sigma}}(u_{k}(s,\cdot))^{\gamma}$ converge to $\e^{-(t-s)\Delta^{\sigma}}(u(s,\cdot))^{\gamma}$ pointwise. Then, by monotone convergence	again	\[
		\lim_{k\rightarrow\infty}\int_{0}^{t}h(s)\,\e^{-(t-s)\Delta^{\sigma}}(u_{k}(s,\cdot))^{\gamma}ds=\int_{0}^{t}h(s)\,\e^{-(t-s)\Delta^{\sigma}}(u(s,\cdot))^{\gamma}\, ds,
		\]
		which implies
		\[
		\lim_{k\rightarrow\infty}\Phi(u_{k})(t,\cdot) = \Phi(u)(t,\cdot).
		\]
		We finally conclude
		\[
		u(t,\cdot)=\lim_{k\rightarrow \infty} u_{k+1}(t,\cdot)=\lim_{k\rightarrow \infty} \Phi(u_{k})(t,\cdot)=\Phi(u)(t,\cdot),
		\]
	which proves that $u$ is a global mild solution of~\eqref{fracheat} provided the regularity $u \in \mathcal{C}([0,+\infty) \times \mathbb{H}^{n})$ holds. Notice that by~\eqref{upplowu}, $u(t,\cdot) \in \mathcal{C}_{0}$ for all fixed $t>0$ (the vanishing at infinity ensured by \eqref{upplowu}, while continuity in space for each $t>0$ by $u=\Phi(u)$ and properties of $f$, $P_t^{\sigma}$); this implies that $\Phi(u)(t,\cdot)$, whence $u(t,\cdot)$, are continuous for all $t>0$.
	
Moreover
	\[
	\lim_{t \to 0^{+}} \| \e^{-t\Delta^{\sigma}}f - f\|_{\infty} =0
	\]
because $f\in \mathcal{C}_{0}$ and recalling $u(0,\cdot)=f$ by~\eqref{upplowu}.  Let now $t>t'>0$ (one works similarly for $t<t'$). We show that when $t'\rightarrow t$, we have 
		\[
		\|u(t,\cdot )-u(t',\cdot )\|_\infty \rightarrow 0.
		\]
		In other words, $u$ is continuous in $t$ uniformly in $x$. This together with its continuity in $x$ for all fixed $t>0$ implies that $u$ is jointly continuous, i.e.\ $u \in \mathcal{C}([0,+\infty) \times \mathbb{H}^{n})$.
		
Writing $F_s=u(s, \cdot)^{\gamma}$ for notational convenience,
		\begin{align*}
			u(t,\cdot)-u(t',\cdot )&=\Phi(u)(t, \cdot)-\Phi(u)(t', \cdot)= \\
			&=\e^{-t\Delta^{\sigma}}  f+\int_{0}^{t} h(s)\, \e^{-(t-s)\Delta^{\sigma}} F_{s}\, ds - \e^{-{t'\Delta^{\sigma}}}  f - \int_{0}^{t'} h(s)\, \e^{-(t'-s)\Delta^{\sigma}} F_{s}\, ds\\
			&=(\e^{-t\Delta^{\sigma}}- \e^{-t'\Delta^{\sigma}}) f + \int_{0}^{t} h(s)\, \e^{-(t-s)\Delta^{\sigma}} F_{s}\, ds -\left( \int_{0}^{t}-\int_{t'}^{t}\right)\\
			&=\e^{-t'\Delta^{\sigma}}(\e^{-(t-t')\Delta^{\sigma}}f-f)\\
			&\qquad +\int_{0}^{t}h(s)\,\e^{-(t'-s)\Delta^{\sigma}}(\e^{-(t-t')\Delta^{\sigma}}F_{s}-F_{s})\,\diff s + \int_{t'}^{t}h(s)\,\e^{-(t'-s)\Delta^{\sigma}}F_{s} \,\diff s.
		\end{align*}
		Now, since by the upper bound~\eqref{upplowu}
	\[
				\|F_{s}\|_{\infty} = \|u(s, \cdot)^{\gamma}\|_{\infty}	\leq \omega(s)^{\gamma}\,\|{\e^{-s\Delta^{\sigma}}f}\|_{\infty}^{\gamma-1} \|\e^{-s\Delta^{\sigma}}f\|_{\infty} \leq C\|f\|_{\infty} \|{\e^{-s\Delta^{\sigma}}f}\|_{\infty}^{\gamma-1} ,
				\]
we have
				\[
				h(s)\,\|\e^{-(t-t')\Delta^{\sigma}}F_s-F_s\|_{ {\infty}}\ \leq 2 h(s) \|F_{s}\|_{\infty} \leq C h(s) {\|f\|_{\infty}}\|{\e^{-s\Delta^{\sigma}}f}\|_{\infty}^{\gamma-1}
				\]
				and this latter function belongs to $L^{1}(0,t)$ for all $t>0$ by  assumption~\eqref{lem2_for1}. It remains to observe that
						\begin{align*}
			\|	u(t,\cdot )-u(t',\cdot )\|_\infty &\leq 
			\|\e^{-t'\Delta^{\sigma}}(\e^{-(t-t')\Delta^{\sigma}}f-f)\|_\infty \\
			& \hspace{-1.5cm} +\int_{0}^{t}h(s)\,\|\e^{-(t'-s)\Delta^{\sigma}}(\e^{-(t-t')\Delta^{\sigma}}F_s-F_s)\|_\infty\,\diff s +\int_{t'}^{t}h(s)\,\|\e^{-(t'-s)\Delta^{\sigma}}F_s \|_\infty \,\diff s \\
			&\leq \|\e^{-(t-t')\Delta^{\sigma}}f-f\|_\infty  \\
			&\quad +\int_{0}^{t}h(s)\,\|\e^{-(t-t')\Delta^{\sigma}}F_s-F_s\|_\infty \,\diff s +\int_{t'}^{t}h(s)\,\|F_s \|_\infty \,\diff s .
		\end{align*}
	Notice in particular that $\|\e^{-(t-t')\Delta^{\sigma}}F_s-F_s\|_\infty \rightarrow 0$ as $t'\rightarrow t$, because $u(s, \cdot)$ whence $F_s$ belong to $\mathcal{C}_0$ for all $s>0$, thanks to $\eqref{upplowu}$. By dominated convergence, as $t' \rightarrow t$ from the left, the above quantities tend to zero, and the proof is complete.
		 \end{proof}
	
	\begin{proposition}\label{prop:oldFujita}
		Suppose $h(t) = \e^{\beta t}$. Then there exists a non-negative global mild solution $u$  to~\eqref{fracheat} provided
		\begin{itemize}
			\item $\gamma>\gamma^{*}$, or
			\item $\gamma=\gamma^{*}$ and $\beta >\frac{2}{3}\lambda_{0}^{\sigma}$.
		\end{itemize}
	\end{proposition}

	\begin{proof}
		Suppose $\eta>0$ and $0\leq f\leq \varepsilon   P_{\eta}^{\sigma}$ for some $\varepsilon$ to be determined. By Lemma~\ref{prop:ex}, it will be enough to show that with a suitable choice of $\varepsilon$, property~\eqref{lem2_for1} holds. Observe that by the semigroup property
		\begin{align*}
			\int_{0}^{\infty} h(s) \|  \e^{-s\Delta^{\sigma}}   f\|_\infty^{\gamma-1}\, \dd s &=
			\int_{0}^{\infty}h(s)\|   f\ast P_{s}^{\sigma}\|_\infty^{\gamma-1}\, \dd s \\
			& \leq \varepsilon\int_{0}^{\infty} \e^{\beta s}\|  P_{s+\eta}^\sigma\|_\infty^{\gamma-1}\, \dd s\\
			& \leq C_{\eta}\varepsilon \int_{0}^{\infty} \e^{\beta s}(s+\eta)^{-\frac{3}{2}(\gamma-1)}  \, \e^{- \lambda_{0}^{\sigma} (s+\eta)(\gamma-1)} \, \dd s\\
			& \leq C_{\eta}\varepsilon \int_{0}^{\infty} \e^{(\beta- \lambda_{0}^{\sigma} (\gamma-1))s}(s+\eta)^{-\frac{3}{2}(\gamma-1)} \, \dd s,
		\end{align*}
		the last but one inequality by the large-time $L^{\infty}$ estimates of~\eqref{PtLp}. Now we have two cases: if $\gamma>\gamma^{*}$, then  $\beta- \lambda_{0}^{\sigma} (\gamma-1)<0$, whence
		\begin{align*}
			\int_{0}^{\infty} h(s) \|  \e^{-s\Delta^{\sigma}}   f\|_\infty^{\gamma-1}\, ds \leq C_{\eta}\varepsilon \int_{0}^{\infty} \e^{(\beta- \lambda_{0}^{\sigma} (\gamma-1))s} \, ds 
			& \leq C\varepsilon \frac{1}{\lambda_{0}^{\sigma} (\gamma-1)- \beta}\\
			& \leq  C\varepsilon \frac{1}{\gamma-\gamma^{*}}
		\end{align*}
		for a suitable $C>0$; if $\gamma=\gamma^{*}$, then $\beta- \lambda_{0}^{\sigma} (\gamma-1)=0$, and
		\begin{align*}
			\int_{0}^{\infty} h(s) \|  \e^{-s\Delta^{\sigma}}   f\|_\infty^{\gamma-1}\, ds \leq C_{\eta}\varepsilon \int_{0}^{\infty}(s+\eta)^{-\frac{3}{2}(\gamma^{*}-1)} \, ds,
		\end{align*}
		which is finite whenever $\frac{3}{2}(\gamma^{*}-1)>1$, namely $\beta >\frac{2}{3}\lambda_{0}^{\sigma}$.
		In both cases, by taking $\varepsilon$ small enough, the statement follows by Lemma~\ref{prop:ex}.
	\end{proof}
	
\begin{remark}
			In the same subcritical regime of Proposition \ref{prop:oldFujita}, we can also give a proof of the existence of a non-negative global \textit{classical} solution $u$  to~\eqref{fracheat} in the spirit of \cite{BPT}. By Corollary~\ref{faith1}, it is enough to construct a bounded global supersolution $\bar{u}$.

				Let us start with the case $\gamma>\gamma^{*}$. 	
				Take $w=c\,\varphi_0$ for some $c>0$ to be determined later. Set
				$$
				\bar{u}(x, t):=\e^{-\lambda_0^{\sigma} t}\, \zeta(t) \,w(x)
				$$
				for some regular $\zeta(t)>0$. 
				Define the functions
				\[
				\tilde{h}(t):=h(t) \,\e^{-(\gamma-1) \lambda_0^{\sigma} t}, \qquad \tilde{H}(t):=\int_0^t \tilde{h}(s) d s,
				\]
				and let $\zeta$ solve the problem
				$$
				\left\{\begin{array}{l}
					\zeta^{\prime}(t)=\|w\|_{\infty}^{\gamma-1} \tilde{h}(t) \zeta^{\gamma}(t) \\
					\zeta(0)=1,
				\end{array}\right.
				$$
				namely
				$$
				\zeta(t)=\left[1-(\gamma-1)\|w\|_{\infty}^{\gamma-1} \tilde{H}(t)\right]^{-\frac{1}{\gamma-1}}.
				$$
				Since $\Delta^{\sigma}\varphi_0=\lambda_{0}^{\sigma}\varphi_0$, one gets
				\[
				\partial_{t}\bar{u}(t,\cdot) + \Delta^{\sigma}\bar{u}(t,\cdot) = \e^{-\lambda_{0}^{\sigma}t} \|w\|_{\infty}^{\gamma-1} \tilde{h}(t) \zeta^{\gamma}(t) \geq 0,
				\]
				whence $\bar{u}$ is a supersolution of~\eqref{fracheat} if $f \leqslant w$. Observe now that $\gamma>\gamma^{*}$ implies that
				\[
				\tilde H_{\infty} := \lim_{t \to \infty}\tilde H(t) <\infty,
				\]
				so if the constant $c>0$ is chosen so that 
				\[
				\|w\|_{\infty}<\left[\frac{1}{(\gamma-1) \tilde{H}_{\infty}}\right]^{\frac{1}{\gamma-1}},
				\]
				then $\bar{u}$ exists for all $t>0$ and is bounded.
				
				Suppose now $\gamma=\gamma^{*}$ and $\beta>\frac{2}{3}\lambda_0^{\sigma}$, and define for some $t_0>0$
				\[
				\bar{u}=\bar{u}(x, t):=\xi(t) \, P^{\sigma}_{t+t_0}(x).
				\]
				Then we have
				$$
				\partial_t \bar{u}+\Delta^{\sigma} \bar{u}-\e^{\beta t} \bar{u}^{\gamma}=P_{t+t_0}^{\sigma} \left[{\xi}^{\prime}-\e^{\beta t} (P_{t+t_0}^{\sigma})^{\beta/\lambda_{0}^{\sigma}}\, \xi^{1+ \beta/\lambda_{0}^{\sigma}}\right].
				$$
				By~\eqref{PtLp} we deduce that there exists $k_1>0$ such that
				\begin{equation}\label{estPtt0}
					P_{t+t_0}^{\sigma}(x) \leq  k_1 \left(t+t_0\right)^{-3 / 2} \e^{-\lambda_0^{\sigma}\left(t+t_0\right)}
				\end{equation}
				for any $x \in \mathbb{H}^n$ and $t \in[0, \infty)$. Hence $\bar{u}$ is a supersolution of~\eqref{fracheat} if $f \leq \xi(0) P_{t+t_0}^{\sigma}$  and if 
				\[
				{\xi}^{\prime}(t)-\e^{\beta t} (P_{t+t_0}^{\sigma})^{\beta/\lambda_{0}^{\sigma}}\, \xi(t)^{1+ \beta/\lambda_{0}^{\sigma}} \geq 0,
				\]
				which by~\eqref{estPtt0} holds in particular if $\xi$ solves the problem (with $K_{1} = k_{1}^{\beta/\lambda_{0}^{\sigma}} \e^{-\beta t_{0}}$)
				$$
				\xi^{\prime}(t)=K_1(t+t_0)^{-\frac{3}{2}\frac{\beta}{\lambda_0^{\sigma}}}\xi(t)^{\frac{\beta}{\lambda_0^{\sigma}}+1}, 
				$$
				which amounts to
				$$
				\xi(t)^{-\frac{\beta}{\lambda_0^{\sigma}}}=\xi(0)^{-\frac{\beta}{\lambda_0^{\sigma}}}-\frac{\beta K_1}{\lambda_0^{\sigma}}\int_{0}^{t}(s+t_0)^{-\frac{3}{2}\frac{\beta}{\lambda_0^{\sigma}}}\, \dd s.
				$$
				Hence $\xi(t)$ exists for all $t>0$ if $\beta>\frac{2}{3} \lambda_0^{\sigma}$ and if $\xi(0)$ is sufficiently small. Notice that $\xi$ hence the supersolution $\bar{u}$ are bounded, thus by the comparison principle Theorem~\ref{thm: comparison} the proof is complete.
	\end{remark}
	
		\subsection{The critical regime: $\gamma=\gamma^{*}$ and $\beta \leq \frac{2}{3}\lambda_{0}^{\sigma}$}
	From this point on we assume $h(t)= \e^{\beta t}$.
	
	By Proposition~\ref{prop:unique}, there exists a mild solution to the problem~\eqref{fracheat} on $[0,\tau) \times \mathbb{H}^{n}$ for some $\tau \in (0,+\infty]$, which we may assume to be the maximal existence time of the solution.  The key observation, inspired by~\cite{WY}, is the following.
	
	\begin{remark}\label{faith2}
	Suppose $\overline{v}$ is a bounded, nonnegative supersolution of the equation 
\[
			 \Delta^{\sigma} v - \lambda_{0}^{\sigma} v - v^{\gamma}=0,
\]
		namely
		\[
		 \Delta^{\sigma} \bar{v} - \lambda_{0}^{\sigma} \bar{v} - \bar{v}^{\gamma} \geq 0,
		\]
by which we mean that  $\bar{v}\in \mathcal{C}^{0,2\sigma+\varepsilon}$ if $0<\sigma<1/2$, while $\bar{v}\in \mathcal{C}^{1,2\sigma-1+\varepsilon}$ if $1/2\leq \sigma <1$ for some small $\varepsilon>0$ and the equation (or inequality) is satisfied pointwise everywhere on $\mathbb{H}^n$. Then the function $\bar{u}(t,x)=\e^{-t\lambda_{0}^{\sigma}} \bar{v}(x)$ is a global nonnegative bounded classical supersolution of the problem~\eqref{fracheat}. By Corollary \ref{faith1}, this is enough to prove the existence of a nonnegative global classical solution to~\eqref{fracheat} for $f \leq \bar{v}$. Moreover, notice that if $0\leq \bar{v}\in L^q$ for some $q\in [1, \infty]$, then $\|u(t, \cdot)\|_q\leq \e^{-t\lambda_{0}^{\sigma}}\|\bar{v}\|_q$. 

\end{remark}
	
	This leads us to Theorem~\ref{thm:B}, namely the following theorem.
	
	\begin{theorem}\label{propeqlambda}
		Suppose $0\leq \lambda \leq \lambda_{0}$, $\sigma \in (0,1)$ and $1<\gamma< \frac{n+2\sigma}{n-2\sigma}$. Then the equation
		\begin{equation}\label{criticreduction}
			 \Delta^{\sigma} v - \lambda^{\sigma} v - v^{\gamma}=0
		\end{equation}		
			 has at least one nontrivial nonnegative radial bounded solution such that
		\[
		\int_{\mathbb{H}^{n}} (|\Delta^{\sigma/2} v|^{2}-\lambda^{\sigma}v^{2})\, \dd \mu \]
is finite and which belongs to $L^q$ for all $q\in (2, \infty]$. The solution is understood in the classical sense, i.e.  $v\in \mathcal{C}^{0,2\sigma+\varepsilon}$ if $0<\sigma<1/2$, while $v\in \mathcal{C}^{1,2\sigma-1+\varepsilon}$ if $1/2\leq \sigma <1$ for some small $\varepsilon>0$ and~\eqref{criticreduction} is satisfied pointwise everywhere on $\mathbb{H}^n$.
		\end{theorem}
The proof of Theorem~\ref{propeqlambda} will need a rather long detour and it will occupy the second part of the paper. Before getting into it, let us see how to complete the proof of the missing part of Theorem~\ref{mainconj} following~\cite{WY}.
	
	\begin{proposition}\label{prop:criticalFujita}
		Suppose $h(t) = \e^{\beta t}$. If $\gamma=\gamma^{*}$ and $0<\beta \leq \frac{2}{3}\lambda_{0}^{\sigma}$, then there exists a non-negative global classical solution $u$ to~\eqref{fracheat} for some positive continuous $f$ in $L^\infty$. 
	\end{proposition}
	
	\begin{proof}
By Remark~\ref{faith2}, it is enough to construct a nonnegative supersolution for~\eqref{criticreduction} when $\lambda^{\sigma} =\frac{\beta}{\gamma^{*}-1} =\lambda_{0}^{\sigma}$, i.e.\ $\lambda=\lambda_{0}$. We distinguish two cases.

Let us first consider the case when $n<8\sigma$, or $n\geq 8\sigma$ and $\beta <\frac{4\sigma}{n-2\sigma}\lambda_{0}^{\sigma}$. Then $\gamma^{*} < \frac{n+2\sigma}{n-2\sigma}$.  Thus Theorem~\ref{propeqlambda} provides the existence a nonnegative bounded solution to~\eqref{criticreduction}, which in particular is a supersolution.

Suppose then that $n\geq 8\sigma$ and $\frac{4\sigma}{n-2\sigma}\lambda_{0}^{\sigma} \leq \beta \leq \frac{2}{3}\lambda_{0}^{\sigma}$. Consider $\bar{\gamma} =1+\frac{2\sigma}{n-2\sigma} $. Since $\bar{\gamma}<\frac{n+2\sigma}{n-2\sigma}$, then by Theorem~\ref{propeqlambda} there exists a nonnegative bounded $v$ such that
		\[
		\Delta^{\sigma} v - \lambda_{0}^{\sigma} v - v^{\bar{\gamma}}=0.
		\]
		Define now $\eta=\eta(v) = \max(1, \|v\|_{\infty})$ and correspondingly $v_{\eta} = v/\eta$. Then $\|v_{\eta}\|_{\infty} \leq 1$. Now, since $\eta \geq 1$, $\gamma^{*}>\bar{\gamma}$ and $v_{\eta} \leq 1$,
		\begin{align*}
			0&= \Delta^{\sigma} v_{\eta}  -  \lambda_{0}^{\sigma} v_{\eta} - \frac{1}{\eta}v^{\bar{\gamma}} \\
			& \leq  \Delta^{\sigma} v_{\eta} - \lambda_{0}^{\sigma} v_{\eta} - \frac{1}{\eta^{\bar{\gamma}}}v^{\bar{\gamma}}  \leq \Delta^{\sigma} v_{\eta} - \lambda_{0}^{\sigma} v_{\eta} - v_{\eta}^{\bar{\gamma}} \leq  \Delta^{\sigma} v_{\eta} - \lambda_{0}^{\sigma} v_{\eta} - v_{\eta}^{{\gamma^{*}}}  .
		\end{align*}
		Thus $v_{\eta}$ is a nonnegative bounded supersolution of~\eqref{criticreduction}. 
	\end{proof}

	From this point on, our goal is to prove Theorem~\ref{propeqlambda}.
	
	\section{Fractional Poincar\'e inequality}\label{Sec:6}

	In this section we aim to prove Theorem~\ref{thm:C}, namely the following fractional Poincar\'e inequality.
		
	\begin{theorem}\label{thm:fracPoincare}
		Suppose $\sigma \in (0,1)$ and $2<q\leq \frac{2n}{n-2\sigma}$. Then there exists $C>0$ such that for all $\phi\in \mathcal{C}_{c}^{\infty} $
		\begin{equation}\label{fracPoincare}
			\int_{\mathbb{H}^n} (|\Delta^{\sigma/2} \phi |^{2}-\lambda_{0}^{\sigma}\phi^{2})\, \dd \mu \geq C \|\phi \|_{q}^{2}.
		\end{equation}
	\end{theorem}
	
	 To this end, we first compare this result to other known Poincar\'e inequalities, and then discuss negative powers of the Laplacian.
	 
	 	\begin{remark}\label{rem:Poincarecomparison}
In~\cite[(1.2)]{MS}, the Poincar\'e--Sobolev inequality
			\begin{equation}\label{PoincareMS}
				\int_{\mathbb{H}^n} (|\nabla \phi |^{2}-\lambda_{0}\phi^{2})\, \dd \mu \geq C \|\phi\|_{q}^{2}, \qquad \phi \in \mathcal{C}_{c}^{\infty}, \quad q \in \big(2,\tfrac{2n}{n-2}\big]
			\end{equation}
			is proved from a Hardy--Sobolev--Maz'ya inequality. Since $\Delta = - \mathrm{div}\, \mathrm{\nabla}$ but also $\Delta = \Delta^{1/2} \Delta^{1/2}$, for $\phi \in \mathcal{C}_{c}^{\infty}$
			\[
			\int_{\mathbb{H}^n} |\nabla \phi|^{2} \, \dd \mu = \int \Delta \phi \cdot \phi \, \dd \mu = \int |\Delta^{1/2} \phi|^{2}\, \dd \mu,
			\]
			whence~\eqref{PoincareMS} is the same as
			\[
			\int_{\mathbb{H}^n} (|\Delta^{1/2} \phi |^{2}-\lambda_{0}\phi^{2})\, \dd \mu \geq C \|\phi\|_{q}^{2},  \qquad \phi \in \mathcal{C}_{c}^{\infty}, \quad q \in \big(2,\tfrac{2n}{n-2}\big]
			\]
			namely~\eqref{fracPoincare} with $\sigma=1$. In~\cite[Theorem 6.2]{LLY}, see also \cite[Theorem 1.12]{BP2022}, instead, the fractional Poincar\'e inequality			\begin{equation}\label{PoincareBP}
				\|(\Delta -\lambda_{0})^{\sigma/2} \phi \|_{2}^{2}= \int_{\mathbb{H}^n} |(\Delta -\lambda_{0})^{\sigma/2} \phi|^{2} \, \dd\mu \geq C \|\phi\|_{q}^{2}, \qquad \phi \in \mathcal{C}_{c}^{\infty}, \quad q \in \big(2,\tfrac{2n}{n-2}\big]
			\end{equation}
			is proved for $\sigma \in (0,3/2)$ for all $n\geq 3$.			 For $\sigma \in (0,1)$ the one given by Theorem~\ref{thm:fracPoincare}, namely~\eqref{fracPoincare}, is stronger.  Indeed, recalling~\eqref{fracPoincarebis}, it is enough to observe that 
\begin{equation}\label{better}
			\|(\Delta -\lambda_{0})^{\sigma/2} \phi\|_{2}\geq c\| (\Delta^{\sigma} - \lambda_{0}^{\sigma})^{1/2}\phi \|_{2},
	\end{equation}
which follows from the boundedness of the multiplier $s\mapsto (s^{\sigma} - \lambda_{0}^{\sigma})^{1/2} (s-\lambda_{0})^{-\sigma/2}$  on $[\lambda_{0},+\infty)$  for any $\sigma\in (0,1)$. Observe that the reciprocal of such multiplier is instead not bounded, so one cannot obtain~\eqref{fracPoincare} directly from~\eqref{PoincareBP}. Equivalently, by the Plancherel formlula, if $\phi \in \mathcal{C}_c^{\infty}$, we have
				\[
				\|(\Delta -\lambda_{0})^{\sigma/2} \phi \|_{2}^{2}=\mathrm{const.}\int_K\int_{0}^{\infty} \xi^{2\sigma}\, |\widehat{\phi}(\xi, k\mathbb{M})|^2 \, \frac{\dd \xi}{|\textbf{c}(\xi)|^2}\dd k
				\]
					and 
					$$\|(\Delta^{\sigma} -\lambda_{0}^{\sigma})^{1/2} \phi \|_{2}^{2}=\mathrm{const.}\int_K\int_{0}^{\infty} ((\xi^2+\lambda_0)^{\sigma}-\lambda_0^{\sigma})\, |\widehat{\phi}(\xi, k\mathbb{M})|^2 \, \frac{\dd \xi}{|\textbf{c}(\xi)|^2}\dd k.$$
					While for $\xi\geq 1$ the above quantities are comparable, for $0<\xi<1$ we have
					$$(\xi^2+\lambda_0)^{\sigma}-\lambda_0^{\sigma}\asymp \xi^2 \lesssim \xi^{2\sigma}, \quad 0<\sigma<1,$$
					from which~\eqref{better} follows.
		\end{remark}	
	Let us note once and for all that
			\begin{equation}\label{multiplier}
(\xi^2+\lambda_0)^{\sigma}-\lambda^{\sigma} \asymp \begin{cases}
\xi^{2\sigma},  \quad &\text{if } \: 0\leq \lambda\leq \lambda_0, \quad \xi \geq 1\\
\xi^2,  \quad &\text{if }\: \lambda=\lambda_0, \quad \quad  \quad   0\leq \xi \leq 1 \\
1,  \quad &\text{if }\: 0\leq \lambda<\lambda_0, \quad  0\leq \xi \leq 1
	\end{cases}
	\end{equation} 
	as this will be needed several times later on. As a consequence, one also obtains the following immediate corollary of Theorem~\ref{thm:fracPoincare}, which for future reference we state as a remark.
\begin{remark}\label{rem-Poincare-lambda}
The inequality~\eqref{fracPoincare} holds if $\lambda_{0}$ is replaced by any $0\leq \lambda\leq \lambda_{0}$.
\end{remark}
	
\subsection{Negative powers of the shifted fractional Laplacian}
Suppose $0<\sigma<1$ and $\alpha>0$. In this subsection, we discuss how to define $(\Delta^{\sigma}-\lambda_{0}^{\sigma})^{-\frac{\alpha}{2}}$, as this will be of use in the following; namely, a range for the parameter $\alpha$. To this end, we borrow some ideas from \cite[p.145]{CGM} and \cite[p.1073]{AJ99}. Given $\phi \in \mathcal{C}_c^{\infty}$, consider the integral 
	\begin{align*}
		\int_K\int_{\mathbb{R}}  ((\xi^2+\lambda_{0})^{\sigma} -\lambda_{0}^{\sigma})^{-\alpha}\,|\widehat{\phi}(\xi, k\mathbb{M})|^2\, \frac{\dd \xi}{|\mathbf{c}(\xi)|^{2}}\, \dd k&=\int_K\int_{|\xi|<1}+\int_K\int_{|\xi|\geq 1}\\
		&=I_1+I_2.
	\end{align*}
By~\eqref{multiplier} and owing to the Paley--Wiener theorem~\eqref{eq: PW} for smooth compactly supported functions, which ensures that $\widehat{f}(\xi, k\mathbb{M})$ decays at infinity faster than any polynomial in $\xi$, uniformly in $K$, and to the polynomial growth of the Plancherel density~\eqref{eq: Plancherel}, it follows that the integral $I_2$ is finite for any $\sigma\in (0,1)$ and $\alpha>0$.
	
	On the other hand, again by~\eqref{multiplier} and taking into account that $|\mathbf{c}(\xi)|^{-2}\asymp \xi^2$ around the origin by~\eqref{eq: Plancherel}, we conclude that the integral $I_1$ is finite provided
	$$\int_{0}^{1}\xi^{-2\alpha}\, \xi^2\, \dd \xi<\infty,$$
	which is true for $\alpha<3/2$. By the Plancherel theorem~\eqref{eq: Plancherel thm}, this shows that 
	\begin{equation}\label{eq: negPlanch}
	(\Delta^{\sigma}-\lambda_{0}^{\sigma})^{-\frac{\alpha}{2}}\phi \in L^2, \qquad \phi \in \mathcal{C}_c^{\infty}, \quad 0<\sigma<1, \quad 0<\alpha<3/2,
	\end{equation}
	see also \cite[p.146]{CGM}, which explicitly means that
	\begin{align*}
		\|(\Delta^{\sigma}-\lambda_{0}^{\sigma})^{-\frac{\alpha}{2}}\phi\|_2^2
		= \int_K\int_{\mathbb{R}} \left((\xi^2+\lambda_{0})^{\sigma} -\lambda_{0}^{\sigma} \right)^{-\alpha} |\widehat{{\phi}}(\xi, k\mathbb{M})|^2\, \frac{\dd \xi}{|\mathbf{c}(\xi)|^{2}}\,\dd k<\infty.
	\end{align*}
More generally, if $\phi\in \mathcal{C}_c^{\infty}$ then by the inversion formula \eqref{eq:inversionHelg} and arguing as above, we can see that $(\Delta^{\sigma}-\lambda_{0}^{\sigma})^{-\frac{\alpha}{2}}\phi$ is well defined for $0<\alpha<3$. Via a Mellin type expression we can then write
	\begin{equation}\label{eq: Mellin}
		(\Delta^{\sigma}-\lambda_{0}^{\sigma})^{-\frac{\alpha}{2}}=\frac{1}{\Gamma(\alpha/2)}\int_{0}^{\infty} \e^{t\lambda_{0}^{\sigma}}\, \e^{-t\Delta^{\sigma}} \, \frac{\dd t}{t^{1-\frac{\alpha}{2}}}, \quad 0<\sigma<1, \quad 0<\alpha<3,
	\end{equation}
 as an operator acting on test functions, and by the above discussion $(\Delta^{\sigma}-\lambda_{0}^{\sigma})^{-\frac{\alpha}{2}} \colon \mathcal{C}_{c}^{\infty} \to L^{2}$ for $0<\alpha<3/2$.

	Let now $k_\sigma$ be the convolution kernel of the operator $(\Delta^{\sigma} - \lambda_{0}^{\sigma})^{-1} $. Since by~\eqref{eq: Mellin}
	\[
	(\Delta^{\sigma} - \lambda_{0}^{\sigma})^{-1} = \int_{0}^{\infty} \e^{\lambda_{0}^{\sigma}t}\e^{-t\Delta^{\sigma}} \, \dd t,
	\]
	one gets
	\[
	k_{\sigma} = \int_{0}^{\infty} \e^{\lambda_{0}^{\sigma}t} P_{t}^{\sigma} \, \dd t.
	\]
	It follows that  $k_{\sigma}$ is positive and radial since so is $P_{t}^{\sigma}$. Let us split $k_{\sigma}$ as 
\begin{equation}\label{k0infty}
	k_{\sigma}^{0}=\mathbbm{1}_{B(o,1)} \, k_{\sigma}, \qquad k_{\sigma}^{\infty}=\mathbbm{1}_{B(o,1)^{c}} \, k_{\sigma}.
	\end{equation}

	\begin{proposition}\label{prop: ka est}
		Suppose $\sigma\in (0,1)$. Then
		\begin{equation*} 
			k_{\sigma}^{0}(x)\asymp |x|^{-(n-2\sigma)}, \quad |x|<1,
		\end{equation*}
		while there is a constant $K>0$ such that
		\begin{equation*} 
			k_{\sigma}^{\infty}(x)\lesssim (1+|x|)^{K}\, \e^{-\frac{n-1}{2}|x|}, \quad |x|\geq 1.
		\end{equation*}
	\end{proposition}
Let us note that  the upper bound for the part $k_{\sigma}^{\infty}$ is not optimal in the polynomial part. However, since it will be sufficient for our purposes, we will avoid performing more delicate computations that would require effort beyond the scopes of this paper.	
	
	We postpone the proof of Proposition~\ref{prop: ka est} in order to prove Theorem~\ref{thm:fracPoincare} first.

\begin{proof}[Proof of Theorem~\ref{thm:fracPoincare}]
To prove~\eqref{fracPoincare}, we first observe that
	\begin{equation}\label{fracPoincarebis}
		\int_{\mathbb{H}^n} (|\Delta^{\sigma/2} \phi |^{2}-\lambda_{0}^{\sigma}\phi^{2})\, \dd \mu = \| (\Delta^{\sigma} - \lambda_{0}^{\sigma})^{1/2}\phi \|_{2}^{2}, \qquad \phi\in \mathcal{C}_{c}^{\infty}.
	\end{equation}
	It suffices then to show that for $\phi \in \mathcal{C}_{c}^{\infty}$
	\begin{equation}\label{ets1}
		\bigg| \int_{\mathbb{H}^n} (\Delta^{\sigma} - \lambda_{0}^{\sigma})^{-1} \phi \cdot \phi \, \dd \mu\bigg| \leq C \|\phi\|_{q'}^2.
	\end{equation}
	Indeed, if~\eqref{ets1} holds, then by H\"older's inequality for $\psi \in \mathcal{C}_{c}^{\infty}(\mathbb{H}^{n})$	\begin{align*}
		|\langle \phi, \psi \rangle | 
		&= \big|\big\langle(\Delta^{\sigma} - \lambda_{0}^{\sigma})^{1/2} \phi, (\Delta^{\sigma} - \lambda_{0}^{\sigma})^{-1/2}\psi\big\rangle\big|\\ 
		&\leq \big \|(\Delta^{\sigma} - \lambda_{0}^{\sigma})^{1/2} \phi\big\|_{2} \big\|(\Delta^{\sigma} - \lambda_{0}^{\sigma})^{-1/2} \psi\big\|_{2} \\
		&= \big \|(\Delta^{\sigma} - \lambda_{0}^{\sigma})^{1/2} \phi \big\|_{2} \big|\big\langle(\Delta^{\sigma} - \lambda_{0}^{\sigma})^{-1}\psi, \psi\big\rangle\big|^{1/2} \\
		&\leq  C^{\frac{1}{2}}  \big \|(\Delta^{\sigma} - \lambda_{0}^{\sigma})^{1/2} \phi\big\|_{2} \|\psi\|_{q'},
	\end{align*}	
	and hence, by duality,
	\[
	\|\phi\|_{q} \leq C^{\frac{1}{2}} \big \|(\Delta^{\sigma} - \lambda_{0}^{\sigma})^{1/2} \phi\big\|_{2},
	\]
	which is~\eqref{fracPoincare}. In turn, proving~\eqref{ets1} boils down to showing that
	\begin{equation}\label{Lp' to Lp}
		\|\phi \ast k_\sigma\|_{q}\leq C \|\phi\|_{q'}, \qquad q\in \big(2,\tfrac{2n}{n-2\sigma} \big].
	\end{equation}
	In order to prove~\eqref{Lp' to Lp}, it suffices to study the behavior of the local part $k_\sigma^{0}$ and of the part at infinity $k_{\sigma}^{\infty}$ separately.
	 	
	Let us start the local part. By Young's inequality, 
	$$\|\phi\ast k_{\sigma}^{0}\|_q\leq \|\phi\|_{q'}\|k_{\sigma}^{0}\|_{q/2}.$$
	We have, in view of Proposition~\ref{prop: ka est}, the radiality of $k_{\sigma}^{0}$ and~\eqref{dens}
	\begin{align*}
		\int_{\mathbb{H}^{n}} k_{\sigma}^{0}(x)^{q/2}\, \diff \mu(x)&=\int_{B(o,1)}k_{\sigma}^{0}(x)^{q/2}\, \diff \mu(x) \\
		&\lesssim\int_{0}^{1} r^{-(n-2\sigma)q/2}\, r^{n-1}\, \diff r.
	\end{align*}
	The right-hand side is finite if and only if $q <\frac{2n}{n-2\sigma}$. In addition, since $\mu(B(o, s))\asymp s^{n}$ for $0<s<1$, it is easy to check that $k_{\sigma}^{0}\in L^{\frac{n}{n-2\sigma}, \infty}$. By Young's inequality for weak spaces, it follows that 
	$$\|\phi \ast k_{\sigma}^{0}\|_{\frac{2n}{n-2\sigma}}\leq \|k_{\sigma}^{0}\|_{\frac{n}{n-2\sigma}, \infty}\|\phi \|_{\frac{2n}{n+2\sigma}}.$$
	We have shown that for all $1\leq q\leq \frac{2n}{n-2\sigma}$ the convolution operator with kernel $k_{\sigma}^{0}$ maps $L^{q'}$ to $L^{q}$.
	
	We now deal with the part at infinity. To this end, we recall the following criterion based on the Kunze--Stein phenomenon, see e.g. \cite[Lemma 4.1]{AP14}: there exists a constant $C>0$ such that, for every radial measurable function $\kappa$ on $\mathbb{H}^n$ and for every $2 < q<\infty$ and $f \in L^{q^{\prime}}\left(\mathbb{H}^n\right)$,
\[
	\|f * \kappa\|_{L^q} \leq C_q\|f\|_{L^{q^{\prime}}}\left(\int_0^{+\infty}(\sinh r)^{n-1}\,|\kappa(r)|^{q / 2} \varphi_0(r)\, \dd r\right)^{2 / q}.
\]
	Applying this to the kernel $k_{\sigma}^{\infty}$ which satisfies the estimates of Proposition~\ref{prop: ka est} and using~\eqref{eq: ground}, we get that the convolution operator with kernel $k_{\sigma}^{\infty}$ is bounded $L^{q'}\rightarrow L^q$ for $q>2$. This concludes the proof.
	\end{proof}

	Before proving Proposition~\ref{prop: ka est}, we also state the following corollary.
	\begin{corollary}\label{corkaest}
	Suppose $\sigma \in (0,1)$. Then $k^{0}_{\sigma} \in L^{p}$ for all $p\in \big[1,\frac{n}{n-2\sigma}\big)$ while $k^{\infty}_{ \sigma} \in L^{p}$ for all $p\in (2,\infty]$.
	\end{corollary}
	
It remains to prove Proposition~\ref{prop: ka est}.

	\begin{proof}[Proof of Proposition~\ref{prop: ka est}]
		We first deal with the local part of the kernel. Write, for $|x|< 1$,
		\begin{align*}
			k_{\sigma}^{0}(x) &= \int_{0}^{\infty} \e^{\lambda_0^\sigma t}\, P_{t}^{\sigma}(x)\, \dd t \\
			&\asymp \int_{0}^{1}  P_{t}^{\sigma}(x)\, \dd t+\int_{1}^{\infty} \e^{\lambda_0^\sigma t}\, P_{t}^{\sigma}(x)\, \dd t \\
			&\asymp \int_{0}^{1}\frac{t}{(t^{\frac{1}{2\sigma}}+|x|)^{n+2\sigma}}\, \dd t + \int_{1}^{\infty} \varphi_{0}(x) \, t^{\frac{1}{2-2\sigma}} (t+|x|)^{-\frac{3}{2}-\frac{1}{2-2\sigma}} \, \dd t,
		\end{align*}
		having used the estimates~\eqref{P kernel estimates}. Observe that
		\begin{align*}
			\int_{0}^{1}\frac{t}{(t^{\frac{1}{2\sigma}}+|x|)^{n+2\sigma}}\, \dd t& = \int_{0}^{|x|^{2\sigma}}\frac{t}{(t^{\frac{1}{2\sigma}}+|x|)^{n+2\sigma}}\, \dd t + \int_{|x|^{2\sigma}}^{1}\frac{t}{(t^{\frac{1}{2\sigma}}+|x|)^{n+2\sigma}}\, \dd t\\
			&\asymp |x|^{-n-2\sigma}\int_{0}^{|x|^{2\sigma}}t\, \dd t  + \int_{|x|^{2\sigma}}^{1} t^{-\frac{n}{2\sigma}}\, \dd t\\
			& \asymp |x|^{-(n-2\sigma)}, \qquad |x|<1, 
		\end{align*}
		while the remaining integral is much smaller, that is
		$$\int_{1}^{\infty} \varphi_{0}(x) \, t^{\frac{1}{2-2\sigma}} (t+|x|)^{-\frac{3}{2}-\frac{1}{2-2\sigma}} \, \dd t \lesssim \int_{1}^{\infty} t^{-\frac{3}{2}}\, \dd t\lesssim 1.$$
		This establishes the desired upper and lower bounds for $k_{\sigma}^{0}$.
		
		We now pass to the part of the kernel at infinity.
		For $|x|\geq 1$ write
		\begin{align*}
			k_{\sigma}^{\infty}(x) &= \int_{0}^{|x|^{\sigma}} \e^{\lambda_0^\sigma t}\, P_{t}^{\sigma}(x)\, \dd t+\int_{|x|^{\sigma}}^{|x|^{2}} \e^{\lambda_0^\sigma t}\, P_{t}^{\sigma}(x)\, \dd t+\int_{|x|^{2}}^{\infty} \e^{\lambda_0^\sigma t}\, P_{t}^{\sigma}(x) \, \dd t\\			&=J_1+J_2+J_3.
		\end{align*}
		Let us start with $J_3$. In this case, $1\leq |x|\leq \sqrt{t}$, so by~\eqref{P kernel estimates}, 
		\begin{align*}
			J_3&\asymp \varphi_{0}(x) \,\int_{|x|^{2}}^{\infty} t^{\frac{1}{2-2\sigma}} \, (t+|x|)^{-\frac{3}{2}-\frac{1}{2-2\sigma}}\,\dd t \\
			&\asymp
			\varphi_{0}(x) \,\int_{|x|^{2}}^{\infty}  t^{-\frac{3}{2}}\,\dd t \\
			& \asymp |x|^{-1} \,\varphi_{0}(x).
		\end{align*}

		Next, we treat $J_1$: in this case $|x|\geq t^{1/\sigma}$. Again by~\eqref{P kernel estimates}, we have
		\begin{align*}
			J_1&\asymp \varphi_{0}(x) \int_{0}^{|x|^{\sigma}} t\,(t+|x|)^{-2-\sigma}\,\e^{-\rho |x|} \, \e^{\lambda_0^\sigma t}\,\dd t \\
			&\asymp |x|^{-2-\sigma} \,\varphi_{0}(x) \left( \int_{0}^{1} t\,\e^{-\rho|x|} \, \e^{\lambda_0^\sigma t} \,\dd t+\int_{1}^{|x|^{\sigma}} t\,\e^{-\rho |x|} \, \e^{\lambda_0^\sigma t} \,\dd t \right).
		\end{align*}
		The first integral clearly contributes an additional exponential decay in space, while the second integral is dominated by 
		$$\int_{1}^{|x|^{\sigma}} t\,\e^{-\rho |x|} \, \e^{\lambda_0^\sigma|x|^{\sigma}} \,\dd t  =\int_{1}^{|x|^{\sigma}} t\,\e^{-\rho |x|(1-\rho^{2\sigma-1}|x|^{\sigma-1})} \, \dd t =\textrm{O}(\e^{-c_{\rho}|x|}),
		$$ 
		for some constant $c_{\rho}>0$, for all $|x|$ large enough. Therefore, we conclude that 
		$$J_1=\textrm{O}({|x|^{-N}\varphi_{0}(x)})\quad \forall N>0.$$ 
		Finally, we turn to $J_2$ for which estimates~\eqref{P kernel estimates} are of no use; we shall use the inversion formula instead, according to which
	\[
	P_t^{\sigma}(x)=\text{const.}\int_{0}^{\infty} \e^{-t(\xi^2+\lambda_0)^{\sigma}}\varphi_{\xi}(x)\, \frac{\dd \xi}{|\mathbf{c}(\xi)|^{2}}.
	\]
Therefore, since $|\varphi_{\xi}| \leq \varphi_{0}$ for all $\xi\in \mathbb{R}$,
	\[
	P_t^{\sigma}(x)=|P_t^{\sigma}(x)|\lesssim \varphi_0(x) \, \int_{0}^{\infty} \e^{-t(\xi^2+\lambda_0)^{\sigma}}\, \frac{\dd \xi}{|\mathbf{c}(\xi)|^{2}},
	\]
	the last integral being finite by \eqref{eq: Plancherel}, since $t\geq 1$. Thus
	\begin{align*}
	J_2(x)&\lesssim \varphi_{0}(x)\int_{0}^{\infty} \int_{|x|^{\sigma}}^{|x|^2} \e^{-t [( \xi^2+\lambda_0)^{\sigma}-\lambda_0^\sigma]}\, \dd t \, \frac{\dd \xi}{|\mathbf{c}(\xi)|^{2}}\\
		&=\varphi_{0}(x)\int_{0}^{\infty}\frac{1}{(\xi^2+\lambda_0)^{\sigma}-\lambda_0^\sigma}\int_{|x|^{\sigma}[(\xi^2+\lambda_0)^{\sigma}-\lambda_0^\sigma]}^{|x|^2 [(\xi^2+\lambda_0)^{\sigma}-\lambda_0^\sigma]} \e^{-u} \,\dd u \, \frac{\dd \xi}{|\mathbf{c}(\xi)|^{2}}\\
		&=\varphi_{0}(x)\bigg(\int_{0}^{1}\dots \,  \, \frac{\dd \xi}{|\mathbf{c}(\xi)|^{2}} + \int_{1}^{\infty} \dots \,  \, \frac{\dd \xi}{|\mathbf{c}(\xi)|^{2}}\bigg) =:\varphi_{0}(x) (J_2^{1}(x)+J_2^{2}(x)).
	\end{align*}
We now use~\eqref{multiplier} and~\eqref{eq: Plancherel}, which for $J_2^{1}$ give
\[
J_{2}^{1}(x)\lesssim \int_{0}^{1}  \xi^2\, \xi^{-2}\int_{0}^{\infty} \e^{-u}\dd u  \, \dd \xi \lesssim 1,
\]
while for $J_2^{2}$ (recall that $|x|\geq 1$)
\begin{align*}
J_2^{2}(x)&\lesssim \int_{1}^{\infty}\xi^{n-1}\, \xi^{-2\sigma} \, \e^{-c\, |x|^{\sigma}\xi^{2\sigma}}\, |x|^{2}\, \xi^{2\sigma} \, \dd \xi \\
&\lesssim \e^{-\frac{c}{2}|x|^{\sigma}}|x|^{2}\int_{1}^{\infty} \xi^{n-1}\,  \e^{-\frac{c}{2}\, \xi^{2\sigma}}\, \dd \xi \\
&\lesssim 1.
\end{align*}
Altogether, we have proved that 
\[
J_2(x)\lesssim\varphi_0(x).
\]
Combining the upper bounds above with the two-sided estimates of $\varphi_{0}$  in~\eqref{eq: ground} completes the proof of the proposition.
	\end{proof}

	\begin{remark} The computations above show that $$k_{\sigma}^{\infty}(r)\geq J_3\gtrsim r^{-1}\, \varphi_{0}(r)\gtrsim \e^{-\frac{n-1}{2}r}, \quad r>1,$$		
		owing to~\eqref{eq: ground}, which implies that $k_{\sigma}^{\infty}$ cannot belong to any $L^p$ space for $1\leq p\leq 2$.
	\end{remark}

		\section{A new scale of fractional Sobolev spaces}\label{Sec:7}

		Given $0\leq \lambda \leq \lambda_{0}$ and $\sigma \in (0,1)$, define the quantity
		\[
		\|\phi \|_{\lambda,\sigma}= \| (\Delta^{\sigma} - \lambda^{\sigma})^{1/2}\phi\|_{2}, \qquad \phi \in \mathcal{C}_{c}^{\infty}
		\]
If $\lambda <\lambda_{0}$, then $(\Delta^{\sigma} - \lambda^{\sigma})^{1/2}\Delta^{-\sigma/2}$ is a bounded operator on $L^{2}$ with bounded inverse, whence $\|\cdot \|_{\lambda,\sigma}$ is equivalent to the Sobolev $\mathrm{H}^{\sigma}$ norm. If $\lambda = \lambda_{0}$ this is not true anymore;  however, by Theorem~\ref{thm:fracPoincare}, $\|\cdot \|_{\lambda, \sigma}$ is a norm on $\mathcal{C}_{c}^{\infty}$. Then, we are brought to the following definition.

\begin{definition}
Suppose $0\leq \lambda \leq \lambda_{0}$ and $\sigma \in (0,1)$. We denote by  $\mathcal{H}_{\lambda,\sigma}$ the completion of $\mathcal{C}_{c}^{\infty}$ with respect to the norm $\|\cdot \|_{\lambda,\sigma}$. 
\end{definition}	
		Then $\mathcal{H}_{\lambda,\sigma}$ is a Hilbert space with inner product
		\[
		\langle \phi, \psi \rangle_{\lambda,\sigma} := \int (\Delta^{\sigma}- \lambda^{\sigma})^{1/2} \phi \cdot (\Delta^{\sigma}- \lambda^{\sigma})^{1/2} \psi \, d\mu 
		\]
		if $\phi, \psi \in \mathcal{C}_{c}^{\infty}$;  while for $u,v\in \mathcal{H}_{\lambda,\sigma}$, given $(\phi_{j}),(\psi_{j}) \in \mathcal{C}_{c}^{\infty}$ such that $\phi_{j} \to u$ and $\psi_{j} \to v$ in $\mathcal{H}_{\lambda,\sigma}$,
		\[
		\langle u, v \rangle_{\lambda,\sigma}  := \lim_{j} \langle \phi_{j}, \psi_{j} \rangle_{\lambda,\sigma}.
		\]
	The above observations might be summarized as follows:
	\begin{proposition}\label{HsigmaHlambdasigma}
	Suppose $\sigma \in (0,1)$. If $0\leq \lambda<\lambda_{0}$ then $\mathrm{H}^{\sigma} = \mathcal{H}_{\lambda,\sigma}$ with equivalent norms. Moreover, $\mathrm{H}^{\sigma} \subseteq \mathcal{H}_{\lambda_{0},\sigma}$.  In particular, for all $0\leq \lambda \leq \lambda_{0}$ and $f\in \mathrm{H}^{\sigma} $, 
\begin{equation}\label{lambdasigmanormonHsigma}
	\| f\|_{\lambda,\sigma}^{2} = \|\Delta^{\sigma} f\|_{2}^{2} - \lambda\|f\|_{2}^{2}.
\end{equation}
	\end{proposition}	
	\begin{proof}
	Only the last assertion is to be proved. But this follows by noticing that~\eqref{lambdasigmanormonHsigma} is satisfied by functions in $\mathcal{C}_{c}^{\infty}$, and then arguing by density:  if $f\in \mathrm{H}^{\sigma}$ and $\phi_j \rightarrow f$ in $\mathrm{H}^{\sigma}$, where $(\phi_j)_{j}\subseteq \mathcal{C}_c^{\infty}$, we have $\phi_j \rightarrow f$ in $\mathcal{H}_{\lambda,\sigma}$ and
	\begin{align*}
\|f\|^2_{\lambda, \sigma}&=\|(\Delta^{\sigma}-\lambda^{\sigma})^{\frac{1}{2}}f\|_2^2\\
&=\lim_j\|(\Delta^{\sigma}-\lambda^{\sigma})^{\frac{1}{2}}\phi_j\|_2^2\\
&=\lim_j\langle (\Delta^{\sigma}-\lambda^{\sigma})\phi_j, \phi_j \rangle\\
&=\lim_j
\|\Delta^{\sigma/2}\phi_j\|_2^2-\lambda^{\sigma}\|\phi_j\|_2^2=\|\Delta^{\sigma/2}f\|_2^2-\lambda^{\sigma}\|f\|_2^2,
	\end{align*}
which proves~\eqref{lambdasigmanormonHsigma}.
	\end{proof}
	
Let us also observe that Theorem~\ref{thm:fracPoincare} (recall also Remark~\ref{rem-Poincare-lambda}) can be rephrased as a Sobolev embedding theorem as follows.
		
	\begin{theorem}\label{thm:fracPoincare2}
		If $0\leq \lambda \leq \lambda_{0}$, $\sigma \in (0,1)$ and $2<q \leq \frac{2n}{n-2\sigma}$, then $\mathcal{H}_{\lambda,\sigma} \hookrightarrow L^{q}$.
		\end{theorem}
	We will see later, see Theorem~\ref{thm: cpt radial emb} below, that in the open range $2<q<\frac{2n}{n-2\sigma}$ such embedding is also compact. For the time being we describe its dual for later purposes, and show some equivalent norms on $\mathcal{H}_{\lambda,\sigma}$ which will be instrumental to prove that if $f\in \mathcal{H}_{\lambda,\sigma}$, then $|f|\in \mathcal{H}_{\lambda,\sigma}$ too (see Corollary~\ref{cor:abs} below).
	
\subsection{The dual space}
Suppose  $0\leq \lambda\leq \lambda_{0}$ and $\sigma \in (0,1)$. We denote by  $\mathcal{H}_{\lambda, \sigma}^{-1}$ the continuous dual of  $\mathcal{H}_{\lambda, \sigma}$, which we do not identify with  $\mathcal{H}_{\lambda, \sigma}$ itself, and we endow $\mathcal{H}_{\lambda, \sigma}^{-1}$ with the dual norm: for $v^{*} \in \mathcal{H}_{\lambda, \sigma}$,
	\[
	\|v^{*}\|_{\mathcal{H}_{\lambda, \sigma}^{-1}} = \sup \big\{ v^{*}(v)\colon v\in \mathcal{H}_{\lambda, \sigma}, \; \|v\|_{\lambda, \sigma}\leq 1 \big\}.
	\]
	As customary, we identify instead $\mathcal{H}_{\lambda, \sigma}$ with a subset of its dual, so that  $\mathcal{H}_{\lambda, \sigma} \subseteq \mathcal{H}_{\lambda, \sigma}^{-1}$. Moreover, since $\mathcal{H}_{\lambda, \sigma} \subseteq L^{q}$ for all $2<q\leq \frac{2n}{n-2\sigma}$ by Theorem~\ref{thm:fracPoincare2}, then $L^{p} \subseteq \mathcal{H}_{\lambda, \sigma}^{-1}$ for all $  \big(\frac{2n}{n-2\sigma}\big)' = \frac{2n}{n+2\sigma} \leq p<2$, where we identify $w\in L^{p}$ with the linear functional $w^{*}$ given by
	\begin{equation}\label{qduality}
		w^{*}(u) = \int_{\mathbb{H}^n} w u\, \dd \mu, \qquad u\in  \mathcal{H}_{\lambda, \sigma}.
	\end{equation}
	The following lemma is standard Hilbert space theory; see also the Lax--Milgram theorem~\cite[\S 6.2.1]{Evans}. We omit the proof.
	\begin{lemma}\label{lemmastar}
		Suppose $0\leq \lambda \leq \lambda_{0}$ and $\sigma \in (0,1)$. Then for all $v^{*}\in \mathcal{H}_{\lambda, \sigma}^{-1}$ there is a unique $v\in \mathcal{H}_{\lambda, \sigma}$ such that $v^{*}(u) = \langle u,v\rangle_{\lambda, \sigma}$  for all $u\in \mathcal{H}_{\lambda, \sigma}$, and the map $T\colon \mathcal{H}_{\lambda, \sigma}^{-1} \to \mathcal{H}_{\lambda, \sigma}$ given by $Tv^{*} = v$ is an isometry.
	\end{lemma}

\subsection{Some equivalent norms}	
	
We also show the following identity.
	\begin{lemma}\label{lemma: Tom form1}
Suppose $0<\sigma<1$ and $0\leq \lambda\leq\lambda_{0}$. Then for all $f\in \mathrm{H}^{\sigma}$
	\begin{align*}
	\|f\|^2_{\lambda, \sigma}=	\frac{1}{2|\Gamma(-\sigma)|}\int_{\mathbb{H}^n}\int_{\mathbb{H}^n}\int_{0}^{\infty} &\left[ |f(x)-f(y)|^2 \, +\, (\e^{-t\lambda}-1) (|f(x)|^2 \, +\, |f(y)|^2)\right] \\
		&\qquad \qquad \times h_t(y^{-1}x)\,\frac{\dd t}{t^{1+\sigma}}\, \dd\mu(y)\, \dd \mu(x).
	\end{align*}
\end{lemma} 

\begin{proof}
	We know from \cite[Lemma 3.4]{BP2022} that for all $0<\sigma<1$ and $f\in \mathrm{H}^{\sigma}$, it holds
\begin{align*}
\|\Delta^{\sigma/2}f\|_2^2&=\frac{1}{2|\Gamma(-\sigma)|}\int_{\mathbb{H}^n}\int_{\mathbb{H}^n}|f(x)-f(y)|^2\, P_0^{\sigma}(d(x,y))\, \dd \mu(y)\, \dd \mu(x)\\
&=\frac{1}{2|\Gamma(-\sigma)|}\int_{\mathbb{H}^n}\int_{\mathbb{H}^n}\int_{0}^{\infty} |f(x)-f(y)|^2\, \, h_t(y^{-1}x)\,\frac{\dd t}{t^{1+\sigma}}\,  \dd \mu(y)\, \dd \mu(x).
\end{align*}
	Moreover,  for all $0\leq \lambda\leq\lambda_0$,
	\begin{align*}
		\lambda^{\sigma}&=\frac{1}{|\Gamma(-\sigma)|}\int_{\mathbb{H}^n}\int_{0}^{\infty} (1-\e^{-t\lambda})\,h_t(y^{-1}x)\frac{\dd t}{t^{1+\sigma}}\, \dd\mu(y) \\
		&=\frac{1}{|\Gamma(-\sigma)|}\int_{\mathbb{H}^n}\int_{0}^{\infty}  (1-\e^{-t\lambda})\, h_t(x^{-1}y)\,\frac{\dd t}{t^{1+\sigma}}\,  \dd \mu(y).
	\end{align*}
	Then $\|\Delta^{\sigma/2}f\|_2^2-\lambda^{\sigma}\|f\|_2^2$ equals
	\begin{align*}
	 \frac{1}{|\Gamma(-\sigma)|}\int_{\mathbb{H}^n}\int_{\mathbb{H}^n}\int_{0}^{\infty} \left[ \frac{1}{2}|f(y)-f(x)|^2 -(1-\e^{-t\lambda}) |f(x)|^2\right]\, h_t(y^{-1}x)\,\frac{\dd t}{t^{1+\sigma}}\,  \dd \mu(y)\, \dd \mu(x)
		\end{align*} 
		but also, by symmetry,
		\begin{align*}	 
\frac{1}{|\Gamma(-\sigma)|}\int_{\mathbb{H}^n}\int_{\mathbb{H}^n}\int_{0}^{\infty} \left[ \frac{1}{2}|f(y)-f(x)|^2 -(1-\e^{-t\lambda}) |f(y)|^2\right]\, h_t(y^{-1}x)\,\frac{\dd t}{t^{1+\sigma}}\,  \dd \mu(y)\, \dd \mu(x).
	\end{align*} 
Summing term by term and using Proposition~\ref{HsigmaHlambdasigma} brings to the claimed identity.
\end{proof}

An important consequence of Lemma~\ref{lemma: Tom form1} for our purposes it the following.	
	\begin{corollary}\label{cor:abs}
		Suppose $0\leq \lambda\leq \lambda_{0}$ and $\sigma \in (0,1)$. If $f\in \mathcal{H}_{\lambda,\sigma}$, then $|f| \in\mathcal{H}_{\lambda,\sigma} $ and  $\| |f|\|_{\lambda,\sigma} \leq \| f\|_{\lambda,\sigma}$.
	\end{corollary} 
	\begin{proof}
	We first note that the statement is true if $f=\phi \in \mathcal{C}_{c}^{\infty}$, since its absolute value $|\phi|$ is Lipschitz continuous with compact support, whence an element of $\mathrm{H}^{1} \subseteq \mathrm{H}^{\sigma}$ by~\cite[Theorem 2.5]{Hebey}, and
	\[
	||\phi (x)|- |\phi(y)|| \leq |\phi(x)-\phi(y)| \qquad x,y\in \mathbb{H}^n,
	\]
implies $\| |\phi |\|_{\lambda,\sigma} \leq \| \phi\|_{\lambda,\sigma}$ by Lemma~\ref{lemma: Tom form1}. Then, suppose $f\in \mathcal{H}_{\lambda,\sigma}$  and $\phi_j \rightarrow f$ in $\mathcal{H}_{\lambda,\sigma}$, where $(\phi_j)_{j}\subseteq \mathcal{C}_c^{\infty}$. By the above, $(|\phi_j|)_{j}$ is a Cauchy sequence in $\mathcal{H}_{\lambda,\sigma}$, and by the embedding in $L^{q}$ given by Theorem~\ref{thm:fracPoincare2} it is Cauchy in $L^{q}$ too. Thus $\phi_j \rightarrow f$ in $L^{q}$ and $(|\phi_j|)_{j}$ converges to some $g$ in $L^{q}$, from which $g=|f|$. Then
\[
\| |f|\|_{\lambda, \sigma} = \lim_{j} \| |\phi_{j}|\|_{\lambda, \sigma} \leq  \lim_{j} \| \phi_{j} \|_{\lambda, \sigma} = \| f\|_{\lambda, \sigma} 
\]
and this completes the proof.
	\end{proof}


	\section{Radiality and compact embeddings}\label{Sec:8}

		Suppose $0<\sigma<1$ and $0\leq \lambda\leq \lambda_0$.
	Let us denote by $\mathcal{H}_{\lambda,\sigma}^{\text{rad}}$ the subspace
	$$
	\mathcal{H}_{\lambda,\sigma}^{\text{rad}}=\left\{u \in \mathcal{H}_{\lambda,\sigma}\colon  u \text { is radial}\right\},
	$$
	and observe that $\mathcal{H}_{\lambda,\sigma}^{\text{rad}}$ is complete, whence a Hilbert space itself. Indeed, convergence in $\mathcal{H}_{\lambda,\sigma}$ implies convergence in some $L^{q}$ by Theorem~\ref{thm:fracPoincare2}, whence (up to subsequences) almost everywhere. We consider also the Sobolev subspace $\mathrm{H}_{\text{rad}}^{s}=\left\{u \in \mathrm{H}^{s}:\,  u \text { is radial}\right\}$, $s>0$. In this section, cf.~\cite[Theorem 3.1]{BhaktaSandeep}, we prove that 
	\begin{theorem}\label{thm: cpt radial emb}
		The embedding $\mathcal{H}_{\lambda,\sigma}^{\text{rad}} \hookrightarrow L^q$ is compact for $2<q<\frac{2n}{n-2\sigma}$.
	\end{theorem}

		We begin by stating two auxiliary results which we shall use in the proof of Theorem~\ref{thm: cpt radial emb}. The following result is inspired by~\cite[Theorem 1.5]{CT} and~\cite[Theorem 8.6]{LL}. 
	
	\begin{proposition}\label{prop: CT Rellich}
	Suppose $\sigma \in (0,1)$, $2<q<\frac{2n}{n-2\sigma}$ and $0\leq \lambda \leq \lambda_0$. Let $B$ be a ball in $\mathbb{H}^n$. Assume that a sequence $(f_j)$ in  $\mathcal{H}_{\lambda, \sigma}$ converges weakly to some $f\in \mathcal{H}_{\lambda, \sigma}$. Then 
		\[
		\|f_j-f\|_{L^q(B)}\rightarrow 0, \quad j\rightarrow \infty.
		\]
	\end{proposition}
	
	\begin{proof}
		
		
		We begin by observing that for $f\in \mathcal{H}_{\lambda,\sigma}$ it holds
			\begin{equation} \label{eq: CT1}
			\|f-\e^{-t(\Delta^{\sigma}-\lambda^{\sigma})}f\|_2\leq \sqrt{t} \, \|(\Delta^{\sigma}-\lambda^{\sigma})^{\frac{1}{2}}f\|_2.
		\end{equation}
		Indeed, we have
		\begin{align*}
			\|f-\e^{-t(\Delta^{\sigma}-\lambda^{\sigma})}f\|_2^{2}&=\int_{K} \int_{\mathbb{R}} |\widehat{f}(\xi, k\mathbb{M})|^2 \, (1-\e^{-t ( (\xi^2+\lambda_0)^{\sigma}-\lambda^{\sigma})}) \, |\textbf{c}(\xi)|^{-2} \, \dd \xi \, \dd k\\
			&\leq \int_{K} \int_{\mathbb{R}}\left((\xi^2+\lambda_0)^{\sigma}-\lambda^{\sigma}\right) \,|\widehat{f}(\xi, k\mathbb{M})|^2 \,  |\textbf{c}(\xi)|^{-2} \, \dd \xi \, \dd k \\
			&= t \, \|(\Delta^{\sigma}-\lambda^{\sigma})^{\frac{1}{2}}f\|_2^{2}.
		\end{align*}
Moreover, the convolution kernel $P_t^{\lambda, \sigma}$ of $\e^{-t(\Delta^{\sigma}-\lambda^{\sigma})}$ is $\e^{t\lambda^{\sigma}}P_{t}^{\sigma}$ and by~\eqref{PtLp}, we get
		\begin{equation}\label{eq: CT3}
			\|P_t^{\lambda, \sigma}\|_p \leq C(\lambda, \sigma, t, p), \quad t>0.
		\end{equation}
		
	Let now $(f_{j})$ be as in the statement. Since the sequence converges weakly in $\mathcal{H}_{\lambda,\sigma}$, it is bounded; therefore, there is $c>0$ such that
		\begin{equation}\label{eq: BA}
			\|f_j\|_{\lambda,\sigma}\leq c, \quad j\in \mathbb{N}.
		\end{equation} 
		We have
		\begin{align}\label{eq: CT 29}
			\|f_j-f\|_{L^2(B)}\leq \|f_j-\e^{-t(\Delta^{\sigma}-\lambda^{\sigma})}f_j\|_{2} &+\|\e^{-t(\Delta^{\sigma}-\lambda^{\sigma})}(f_j-f)\|_{L^2(B)} \\
			& +\|\e^{-t(\Delta^{\sigma}-\lambda^{\sigma})}f-f\|_{2} \notag. 
		\end{align}
		On the one hand, we get by~\eqref{eq: CT1}
		\begin{equation}\label{eq: CT 30}
			\|f_j-\e^{-t(\Delta^{\sigma}-\lambda^{\sigma})}f_j\|_2 \lesssim \sqrt{t}, \qquad \|f-\e^{-t(\Delta^{\sigma}-\lambda^{\sigma})}f\|_2 \lesssim \sqrt{t},
		\end{equation} 
		and on the other hand, by Theorem~\ref{thm:fracPoincare2} and~\eqref{eq: BA},  we get
		\begin{equation*}
			\|f_j\|_{q} \leq C\|f_j\|_{\lambda,\sigma} \leq C(n , \sigma), \qquad 2<q\leq \frac{2n}{n-2\sigma},
		\end{equation*}
		and likewise for $\|f\|_q$. Hence the H{\"o}lder inequality together with~\eqref{eq: CT3}, imply
		\begin{align}\label{eq: CT 34}
			\|\e^{-t(\Delta^{\sigma}-\lambda^{\sigma})} f_j\|_{\infty}&\leq \|P_t^{\lambda,2\sigma}\|_{q'} \, \|f_j\|_{q}\leq C(\lambda, n, \sigma, t), \notag \\  \|\e^{-t(\Delta^{\sigma}-\lambda^{\sigma})}f\|_{\infty}&\leq \| {P}_t^{\lambda,2\sigma}\|_{q'} \, \|f\|_{q}\leq C(\lambda, n, \sigma, t).
		\end{align}
		Next, observe that due to Theorem~\ref{thm:fracPoincare2}, we also have $f_{j}\rightarrow f$ weakly in $L^{q} $ for $2<q<\frac{2n}{n-2\sigma}$, and due to~\eqref{eq: CT3} (for $p=q'\in (\frac{2n}{n+2\sigma},2)$) we
		get that $ f_j\ast \mathcal{P}_t^{\lambda, \sigma} \rightarrow f\ast \mathcal{P}_t^{\lambda,\sigma}$  pointwise on $\mathbb{H}^{n}$ (notice that $f\ast \mathcal{P}_t^{\lambda,\sigma}(x)=\int_{\mathbb{H}^n}  \mathcal{P}_t^{\lambda,\sigma}(x^{-1}\cdot)f$). From this observation together with~\eqref{eq: CT 34} and the fact that the measure $\mu(B)$ is finite, the dominated convergence theorem yields
		\begin{equation} \label{eq: CT 35}
\| \e^{-t(\Delta^{\sigma}-\lambda^{\sigma})} (f_{j} -f)\|_{L^{2}(B)} =	\|f_j\ast \mathcal{P}_t^{\lambda,\sigma}-f\ast \mathcal{P}_t^{\lambda, \sigma}\|_{L^2(B)}\rightarrow 0, \quad j\rightarrow \infty.
		\end{equation}
		Combining~\eqref{eq: CT 29},~\eqref{eq: CT 30} and~\eqref{eq: CT 35} we get that
		\begin{equation}\label{eq: L2 conv B}
			\|f_j-f\|_{L^2(B)}\rightarrow 0, \quad j\rightarrow \infty.
		\end{equation}
		This proves the theorem if $q \leq 2$, since then, by the H{\"o}lder inequality,
		\[ \|f_j-f\|_{L^q(B)}\leq \mu(B)^{1/r} \|f_j-f\|_{L^2(B)}\]
		where $1/q = 1/r + 1/2$.
		If $2<q< \frac{2n}{n-2\sigma}$, then by choosing $r\in (q, \frac{2n}{n-2\sigma})$, again by the H{\"o}lder inequality,
		\[ \|f_j-f\|_{L^q(B)}\leq \|f_j-f\|^{\theta}_{L^2(B)}\,  \|f_j-f\|^{1-\theta}_{L^r(B)}\]
		for $\theta=(1/q-1/r)(1/2-1/r)$. The convergence claim follows by~\eqref{eq: L2 conv B} and the fact that
		\begin{align*}
			\|f_j-f\|_r&\leq C  \|f_j-f\|_{\lambda,\sigma} \leq C (\|f_j\|_{\lambda,\sigma}+\|f\|_{\lambda,\sigma})\leq C,
		\end{align*}
		owing to Theorem~\ref{thm:fracPoincare2} and~\eqref{eq: BA}. The proof is now complete.
		\end{proof}

	The second result is more technical:
	
		\begin{proposition}\label{prop: Lions}
		Suppose $0<\sigma<1$, $0\leq \lambda \leq \lambda_0$ and $2<q<\frac{2n}{n-2\sigma}$. Then there is $\delta=\delta(n, q)>0$  and $C>0$ such that for all $u\in \mathcal{H}_{\lambda, \sigma}^{\text{rad}}$ and $R>1$ 
		\[
		\|u\|_{L^q(B(o, R)^c)}\leq C \, \e^{-\delta R} \|u\|_{\mathcal{H}_{{\lambda, \sigma} }}.   \]
	\end{proposition}
	We postpone its proof for a moment, to show how Theorem~\ref{thm: cpt radial emb} follows.

\begin{proof}[Proof of Theorem~\ref{thm: cpt radial emb}]
		Take a bounded sequence $(u_j)\in \mathcal{H}_{\lambda, \sigma}^{\rad}$, so that $\|u_{j}\|_{\lambda, \sigma}\leq M$ for some $M>0$ and every $j\in \mathbb{N}$. Passing to a subsequence if necessary, we get that $(u_j)$ converges weakly to some $u\in L^q$. So, according to Proposition~\ref{prop: CT Rellich}, we can take a sequence $R_m\rightarrow\infty$ so that $u_j\rightarrow u$ strongly in $L^q(B(o, R_m))$ for every $m$. We will show that $(u_j)$ is Cauchy in $L^q$. For every $j, k \in \mathbb{N}$, write
		\[
		\|u_j-u_k\|_{L^q}\leq \|u_j-u_k\|_{L^q(B(o,R_m))}+\|u_j-u_k\|_{L^q(B(o,R_m)^c)}.
		\]
By the triangle inequality and Proposition~\ref{prop: Lions}
		\[
		\|u_j-u_k\|_{L^q(B(o,R_m)^c)}\leq 2CM\, \e^{-\delta R_m},
		\]
		therefore for $\varepsilon>0$, there is $m_0\in \mathbb{N}$ so that for all $m\geq m_0$,
		\[
		\|u_j-u_k\|_{L^q(B(o,R_m)^c)}<\frac{\varepsilon}{2}.
		\]
		On the other hand, for the ball $B(o,r_{m_0})$, by Proposition~\ref{prop: CT Rellich}, there is $j_0\in \mathbb{N}$ so that for all $j, k\geq j_0$, it holds
		\[
		\|u_j-u_k\|_{L^q(B(o,r_{m_0}))}<\frac{\varepsilon}{2}.
		\]
		The proof of the compactness of the embedding is now complete.
	\end{proof}

	Hence, in the remaining part of this section, the aim is to prove Proposition~\ref{prop: Lions}.

\subsection{Proof of Proposition~\ref{prop: Lions}}
We shall prove Proposition~\ref{prop: Lions} for $u\in(\mathcal{C}_{c}^{\infty})^{\rad}$, since the conclusion follows by density. Suppose then $u\in(\mathcal{C}_{c}^{\infty})^{\rad}$. Then
	\begin{equation}\label{eq: norms}
		\|u\|_{\lambda, \sigma}=\|(\Delta^{\sigma}-\lambda^{\sigma})^{1/2}u\|_2=\|f\|_2<\infty,
	\end{equation}
	where $f=(\Delta^{\sigma}-\lambda^{\sigma})^{1/2}u\in L^2$. By the spectral theorem and a subordination formula, we have
	\[
	f=(\Delta^{\sigma}-\lambda^{\sigma})^{1/2}u=\frac{1}{|\Gamma(-1/2)|}\int_{0}^{\infty} (u -\e^{t\lambda^{\sigma}} \e^{-t\Delta^{\sigma}}u )\, \frac{\dd t}{t^{3/2}},
	\]
	and since $\e^{-t\Delta^{\sigma}}u=u\ast P_t^{\sigma}$ is radial, as the convolution of two radial functions, so is $f$. Hence, by the Plancherel formula,
	\[
	\|u\|_{\lambda, \sigma}=\|f\|_{L^2}=\|\mathcal{H}f\|_{L^2({\mathbb{R}}, |\textbf{c}(\xi)|^{-2} \dd \xi)}.
\]
	Notice that by the inversion formula we can write
	\begin{align*}
		u(r)&= (\Delta^{\sigma}-\lambda^{\sigma})^{-1/2}f(r)=\text{const.} \int_{0}^{\infty} \frac{1}{\sqrt{(\xi^2+\lambda_0)^{\sigma}-\lambda^{\sigma}}}\, \mathcal{H}f(\xi) \, \varphi_{\xi}(r)\, \frac{\dd \xi}{|\textbf{c}(\xi)|^{2}}\\
		&=\text{const.}\left(\int_0^{1}+\int_{1}^{\infty}\right) \frac{1}{\sqrt{(\xi^2+\lambda_0)^{\sigma}-\lambda^{\sigma}}}\, \mathcal{H}f(\xi) \, \varphi_{\xi}(r)\, \frac{\dd \xi}{|\textbf{c}(\xi)|^{2}}\\
		& =u^{0}(r)+u^{\infty}(r).
	\end{align*}
	Clearly, to prove Proposition~\ref{prop: Lions}, it suffices to show that 
	\begin{equation}\label{eq: u0 and uinfty Lq}
		\|u^0\|_{L^q(B(o, R)^c)}\leq C \, \e^{-\delta R} \|u\|_{{\lambda}, \sigma},\quad \|u^{\infty}\|_{L^q(B(o, R)^c)}\leq C \,\e^{-\delta R} \|u\|_{{\lambda}, \sigma}.
	\end{equation}

	For the low frequency part $u^{0}$, using that $|\varphi_{\xi}|\leq \varphi_{0}$ for $\xi\in \mathbb{R} $ and~\eqref{multiplier} we get
	\begin{align*}
		|u^{0}(r)|&\lesssim \varphi_{0}(r)\int_0^{1} \xi^{-1}|\mathcal{H}f(\xi)|\,\frac{\dd \xi}{|\textbf{c}(\xi)|^{2}} \\
		&\lesssim \varphi_{0}(r) \left(\int_0^{1}  |\mathcal{H}f(\xi)|^2\,\frac{\dd \xi}{|\textbf{c}(\xi)|^{2}} \right)^{1/2} \left( \int_0^{1} \xi^{-2}\,\frac{\dd \xi}{|\textbf{c}(\xi)|^{2}} \right)^{1/2}
	\end{align*}
	by Cauchy--Schwarz. By the local behavior~\eqref{eq: Plancherel} of the Plancherel density and~\eqref{eq: norms}, we conclude that
	\[
	|u^{0}(r)| \lesssim \varphi_{0}(r)\, \|u\|_{\lambda, \sigma}, \quad r>0.
	\]
	Then by the radiality of $u^0$ and the behavior~\eqref{eq: ground} of $\varphi_{0}$, we get for $q>2$ and $R>1$ 
	\begin{align}
		\int_{B(o, R)^c} |u^{0}(x)|^q\, \dd \mu (x) &=c_n\int_{R}^{\infty} |u^0(r)|^q\, \sinh^{n-1}r\, \dd r \notag\\
		&\lesssim \|u\|^q_{\lambda, \sigma}\int_{R}^{\infty} (1+r)^q\, \e^{-q\frac{n-1}{2}}\, \e^{(n-1)r}\, \dd r \notag\\
		&\lesssim  \|u\|^q_{\lambda, \sigma} \int_{R}^{\infty} (1+r)^q\, \e^{-\frac{n-1}{2}(q-2)r}\, \dd r \notag\\
		&\lesssim  \|u\|^q_{\lambda, \sigma} \int_{R}^{\infty} \e^{-\frac{n-1}{2}\frac{(q-2)}{2}r} \, \dd r \notag\\
		&\lesssim  \|u\|^q_{\lambda, \sigma} \e^{-\frac{n-1}{2}\frac{(q-2)}{2}R}, \label{eq: pointwise}
	\end{align}
	which proves~\eqref{eq: u0 and uinfty Lq}  for $\delta=\frac{n-1}{2}\frac{q-2}{2q}$, as far as the low-frequency part $u^0$ is concerned.
	
	It therefore remains to deal with the high frequency part $u^{\infty}$, which we treat with different methods for $\frac{1}{2}<\sigma<1$ and $0<\sigma\leq \frac{1}{2}$.

To begin with, we observe that $u^{\infty}=\mathcal{D}^\sigma f$ where $\mathcal{D}^\sigma$ is the spectral multiplier
	\[
	\mathcal{D}^\sigma=\mathbbm{1}_{(1,\infty)}(\sqrt{\Delta-\lambda_0})(\Delta^{\sigma}-\lambda^{\sigma})^{-1/2}.
	\]
	Since the multiplier
	\begin{equation*}
		m(\xi)=\mathbbm{1}_{(1,\infty)}(\xi)\, \frac{1}{\sqrt{(\xi^2+\lambda_0)^{\sigma}-\lambda^{\sigma}}}
	\end{equation*}
	is bounded on $(0, \infty)$, by spectral theory we get
	\begin{equation}\label{eq: spth1}
		\|u^{\infty}\|_{2}= \|\mathcal{D}^\sigma f\|_2 \lesssim \|f\|_{L^2} \lesssim \|u\|_{\lambda, \sigma},
		\end{equation}
whence $u^{\infty}\in L^2$. Moreover $\Delta^{\sigma/2}u^{\infty}\in L^2$, since
	\begin{equation}\label{eq: spth2}
		\|\Delta^{\sigma/2}u^{\infty}\|_2=\|\mathbbm{1}_{(1,\infty)}(\sqrt{\Delta-\lambda_0})\Delta^{\sigma/2}(\Delta^{\sigma}-\lambda^{\sigma})^{-1/2}f\|_2\lesssim \|f\|_2\lesssim \|u\|_{\lambda, \sigma}<\infty
	\end{equation}
	since the multiplier
	\begin{equation*}
		\widetilde{m}(\xi)=\mathbbm{1}_{(1,\infty)}(\xi)\, \frac{(\xi^2+\lambda_0)^{\sigma/2}}{\sqrt{(\xi^2+\lambda_0)^{\sigma}-\lambda^{\sigma}}}
	\end{equation*}
	is bounded on $(0, \infty)$. Therefore  $u^{\infty}\in \mathrm{H}^{\sigma}$ and by Proposition~\ref{HsigmaHlambdasigma}
	\[
	\|u^{\infty}\|_{\lambda, \sigma}^2=\|\Delta^{\sigma/2}u^{\infty}\|_{2}^2-\lambda \|u^{\infty}\|_2^2.	\]

	\subsubsection{The range $1/2<\sigma<1$.} As in the case of $u^{0}$, we rely on pointwise bounds away from the origin for $u^{\infty}$ to derive~\eqref{eq: u0 and uinfty Lq}.
	
	Recall that for every $\xi\in \mathbb{R}$ it holds $\varphi_{\xi}=\varphi_{-\xi}$. In addition the following large scale converging expansion of the spherical function $\varphi_{\xi}$ is true, 
\[
		\varphi_{\xi}(r)=\e^{-\frac{n-1}{2}r}\left(\e^{i\xi r}\textbf{c}(\xi)m_1(\xi, r)+\e^{-i\xi r}\overline{\textbf{c}(\xi)}m_1(-\xi, r)\right),\quad r>1
\]
	since $\textbf{c}(-\xi)^{-1}=\overline{\textbf{c}(\xi)}$, where 
	\[
	|m_1(\xi, r)|\lesssim 1 \quad \text{for all } \xi \in \mathbb{R}, \, r>1,
	\]
	see \cite[Prop. A.2]{Ion}. Since $m_1$ is uniformly bounded in $\xi$ and $r$, and $\varphi_{\xi}$ and $\mathcal{H}f(\xi)$ are even, it is clear that we can write 
	\begin{align*}
		|u^{\infty}(r)|\lesssim \e^{-\frac{n-1}{2}r} \int_{1}^{\infty}\frac{1}{\sqrt{(\xi^2+\lambda_0)^{\sigma}-\lambda^{\sigma}}}\, |\mathcal{H}f(\xi)|\, \frac{\dd \xi}{|\textbf{c}(\xi)|}.
	\end{align*}
	Therefore, by~\eqref{multiplier}, the Cauchy--Schwarz inequality yields 
	\begin{align*}
		|u^{\infty}(r)|&\lesssim \e^{-\frac{n-1}{2}r} \left(\int_1^{\infty}  |\mathcal{H}f(\xi)|^2\, \frac{\dd \xi}{|\textbf{c}(\xi)|^{2}} \right)^{1/2} \left( \int_1^{\infty} \xi^{-2\sigma}\, \dd \xi \right)^{1/2} \\
		&\lesssim \e^{-\frac{n-1}{2}r} \, \|u\|_{\lambda, \sigma},
	\end{align*}
	again by~\eqref{eq: norms} and the fact that $\sigma>1/2$. 
	
	Hence we get~\eqref{eq: u0 and uinfty Lq} with $\delta=\frac{n-1}{2}(q-2)$ for $1/2<\sigma<1$, by arguing as in~\eqref{eq: pointwise}.
	
	\subsubsection{The range $0<\sigma\leq 1/2$.} In this range, we shall prove~\eqref{eq: u0 and uinfty Lq} for $u^{\infty}$ by a different method, which leads to the following.
	\begin{proposition}\label{lemma: L3}
		Suppose  $0< \sigma \leq 1/2$ and $2<q<\frac{2n}{n-2\sigma}$. Then for $u\in \mathrm{H}^{\sigma}_{\text{rad}}$
		\[
		\|(\sinh r)^{\frac{n-1}{2}} \, \tilde u\|_{L^{q}(1, \infty)}\leq C\, \|u\|_{\mathrm{H}^{\sigma}}.
		\]
	\end{proposition}
	Before proving the proposition, we first complete the proof of Proposition~\ref{prop: Lions}, i.e.\ the proof of~\eqref{eq: u0 and uinfty Lq}.  Observe indeed that a radial function $u\in \mathrm{H}^{\sigma}$, $0<\sigma\leq 1/2$, satisfies
	\begin{align}
		\int_{B(o,R)^c}|u|^q &=C_n \int_{R}^{\infty} |u(r)|^q \,(\sinh r)^{n-1} \dd r \notag \\
		&\leq C\,(\sinh R)^{-\frac{n-1}{2}(q-2)} \int_{R}^{\infty}\, |u(r)|^q\, (\sinh r)^{\frac{n-1}{2}q}\, \dd r \notag  \\
		&\leq C\,(\sinh R)^{-\frac{n-1}{2}(q-2)} \, \|u\|_{\mathrm{H}^{\sigma}}, \label{eq: Lqint}
	\end{align}
	where the last step is justified by Proposition~\ref{lemma: L3}. 
	
	Finally, let us apply~\eqref{eq: Lqint} to the radial function $u^{\infty}\in \mathrm{H}^{\sigma}$. Then
	\begin{align*}
		\int_{B(o,R)^c}|u^{\infty}|^q &\leq C\, \e^{-\frac{n-1}{2}(q-2)R}\, \|u^{\infty}\|_{\mathrm{H}^{\sigma}}\\ 
		&\leq C\, \e^{-\frac{n-1}{2}(q-2)R}\, \|u\|_{\lambda, \sigma},
	\end{align*}
	where the last step is justified by~\eqref{eq: spth1} and~\eqref{eq: spth2}. This proves~\eqref{eq: u0 and uinfty Lq} for $\delta=\frac{n-1}{2}\frac{q-2}{q}$ for the high frequency part $u^{\infty}$, in the case $0<\sigma\leq 1/2$, and concludes the proof of Proposition~\ref{prop: Lions}.
	
Therefore, it remains to prove Proposition~\ref{lemma: L3}.

	Let $I=(a, b)$ be an open interval, possibly unbounded. Recall that the Sobolev space $W^{1, 2}(I)$ is defined to be
	$$
	W^{1, 2}(I)=\left\{u \in L^2(I) \colon  \exists u' \in L^2(I) \text { such that } \int_I u \varphi^{\prime}=-\int_I u' \varphi \quad \forall \varphi \in C_c^1(I)\right\},
	$$
endowed with the norm
	\[
	\| u\|_{W^{1,2}(I)} = \|u\|_{L^{2}(I)} + \|u'\|_{L^{2}(I)}.
	\]
	On the other hand, for $\sigma$ in $(0,1)$, the fractional Sobolev space $W^{\sigma, 2}(I)$ is classically defined as 
	$$
	W^{\sigma, 2}(I)=\left\{u \in L^2(I)\colon \frac{|u(x)-u(y)|}{|x-y|^{\frac{1}{2}+\sigma}} \in L^2(I \times I)\right\}
	$$
	i.e., an intermediate Banach space between $L^2(I)$ and $W^{1, 2}(I)$, endowed with the natural norm
	$$
	\|u\|_{W^{\sigma, 2}(I)}=\left(\int_{I}|u|^2 \, \dd x+\int_{I} \int_{I} \frac{|u(x)-u(y)|^2}{|x-y|^{1+2\sigma}} \,\dd x \, \dd y\right)^{\frac{1}{2}}.
	$$
	We first prove a (fractional) hyperbolic analog of~\cite[Lemme II.2]{Lions}. To avoid ambiguities, we distinguish between a radial function $u$ and its profile $\tilde u$ in the next statements.
	\begin{lemma}\label{lemma: L1}
		Suppose $0< \sigma \leq 1$. For all $u\in \mathrm{H}^{\sigma}_{\text{rad}}$  it holds
		\[
		\|(\sinh \cdot )^{\frac{n-1}{2}} \, \tilde u\|_{W^{\sigma,2}(1, \infty)}\leq C\, \|u\|_{\mathrm{H}^{\sigma}}.
		\]
	\end{lemma}
	
	\begin{proof}
	
	Let us start from the case $\sigma=1$.  Recall that for radial functions on $\mathbb{H}^n$ we have 
			$$
			\left|\nabla_{\mathbb{H}^n} u\right|=\left|\frac{\dd \tilde u}{\dd r}\right|.
			$$
			Thus, the gradient term in the Sobolev norm $\mathrm{H}^1$ simplifies to:
			$$
			\int_{\mathbb{H}^n}\left|\nabla_{\mathbb{H}^n} u\right|^2\, \dd \mu=c_n\int_0^{\infty}\left|\frac{\dd \tilde u}{\dd r}\right|^2 \, (\sinh r)^{n-1} \,\dd r.
			$$
			Therefore, for $v(r)=(\sinh r)^{\frac{n-1}{2}}\,\tilde u(r)$, we have
			\[
			\frac{\dd v}{\dd r}=\frac{n-1}{2}(\sinh r)^{\frac{n-3}{2}}\cosh r \,\tilde u(r)+ (\sinh r)^{\frac{n-1}2}\frac{\dd \tilde u}{\dd r}.
			\]
			Since for $r>1$ it holds $\cosh r\asymp \sinh r$, we get
			\[
			\left|\frac{\dd v}{\dd r}\right|\leq C |v(r)| +(\sinh r)^{\frac{n-1}2}\left|\frac{\dd \tilde u}{\dd r}\right|, \quad r>1.
			\]
			Hence, given that $\|v\|_{L^2(0, \infty)}=C_n\|u\|_{L^2}$, we obtain
			\[
			\left\|\frac{\dd v}{\dd r}\right\|_{L^{2}(1, \infty)} \leq C\|u\|_{L^{2} }+\left\|(\sinh r)^{\frac{n-1}{2}} \frac{\dd \tilde u}{\dd r}\right\|_{L^{2}(0, \infty)} \leq C\|u\|_{\mathrm{H}^{1} },
			\]
			which proves the statement.
			
For $\sigma \in (0,1)$ we argue by interpolation.  Indeed, since  $$\|(\sinh \cdot )^{\frac{n-1}{2}}\, \tilde u\|_{L^2(0, \infty)}=C_n\|u\|_{L^2}$$
as already recalled, we see that the linear operator $T \colon L^{2}_{\rad} \to L^{2}(0,\infty)$	given by
\[
T(u) = (\sinh \cdot )^{\frac{n-1}{2}}\, \tilde u
\]
is bounded $L^{2}_{\rad} \to L^{2}(1,\infty)$ and $\mathrm{H}^{1}_{\rad} \to W^{1,2}(1,\infty)$. Since
		\[
		[L^2(1, \infty), W^{1, 2}(1, \infty)]_{\sigma}=W^{\sigma, 2}(1, \infty),
		\]
		and
		\[
		[L^2 , \mathrm{H}^{1} ]_{\sigma}=\mathrm{H}^{\sigma} ,
		\]
see for instance~\cite[\S 10.1--10.3]{LionsMagenes} and~\cite[Lemma 2.2]{An92} respectively, by the Riesz--Thorin interpolation theorem~\cite[Theorem 4.4.1]{BerghLofstrom} we get
\[
T \colon \mathrm{H}^{\sigma}_{\rad} \to W^{\sigma,2}(1,\infty)
\]
		boundedly, whence the result.
		\end{proof}
		We are now ready to prove Proposition~\ref{lemma: L3}.

	\begin{proof}[Proof of Proposition~\ref{lemma: L3}]
		Suppose $u\in \mathrm{H}^{\sigma}_{\text{rad}}$. Then by \cite[Theorems 6.7 and 6.10]{Hitch} on the real line, for $0<\sigma \leq 1/2$ and for all $q\in \big(2, \frac{2}{1-2\sigma}\big]$ (notice that $\frac{2n}{n-2\sigma}<\frac{2}{1-2\sigma}$ for $n\geq 2$), we get
		\[
		\|(\sinh r)^{\frac{n-1}{2}}\, \tilde u \|_{L^q(1, \infty)}\leq C\, \|(\sinh r)^{\frac{n-1}{2}}\, \tilde u\|_{W^{\sigma,2}(1, \infty)}.
		\]
		Then, by Lemma~\ref{lemma: L1}
		\[
		\|(\sinh r)^{\frac{n-1}{2}}\, \tilde u \|_{L^q(1, \infty)}\leq C\, \|(\sinh r)^{\frac{n-1}{2}}\, \tilde u\|_{W^{\sigma, 2}(1, \infty)} \leq C \|u\|_{\mathrm{H}^{\sigma}(\mathbb{H}^{n})},
		\]
		which finishes the proof.
		\end{proof}

	\subsection{Radialization}
	 In this section we discuss properties of radialized functions, and in particular we prove that the radialization of a function does not increase its Sobolev $\mathcal H_{\lambda,\sigma}$ norm. For $ \phi\in \mathcal{C}_c^{\infty}( \mathbb{H}^n)$, denote its radialization by 
	\[
	\phi^{\#}(x)=\int_{K} \phi(kx)\, \dd k.\]
	Clearly, radial functions satisfy $\phi=\phi^{\#}$ and $\phi\in \mathcal{C}_c^{\infty}$ implies that $\phi^{\#}\in \mathcal{C}_c^{\infty}$. By the fact that for all $\xi \in \mathbb{R}$, the spherical functions $\varphi_{\xi}$ are bi-$K$-invariant and $|\varphi_{\xi}|\leq \varphi_0$, hence belong to all $L^p$, $2<p\leq \infty$,  the Fubini theorem yields
	\begin{align*}
		(\mathcal{H}\phi^{\#})(\xi)&=\int_G \int_{K}\phi(kx)\, \varphi_{-\xi}(x)\, \dd k\,\dd x =\int_G \phi(x)\, \varphi_{-\xi}(x)\, \dd x \\
		&=\int_K \int_G \phi(x)\,\e^{(-i\xi+\rho)\tau(k^{-1}x)}\,\dd x\,  \dd k,
	\end{align*}
	where the last equality follows from the integral representation~\eqref{eq: spherical hyp} of $\varphi_{\xi}$ and a change of variables in $K$. Therefore, by~\eqref{H-Ftr},
\[
		(\mathcal{H}\phi^{\#})(\xi)=\int_{K}	\widehat{\phi}(\xi, k\mathbb{M})\, \dd k, \quad \phi\in \mathcal{C}_c^{\infty},
\]
	so by Cauchy--Schwarz, 
	\begin{align}\label{eq: compare}
		\int_K|(\mathcal{H}\phi^{\#})(\xi)|^2\, \dd k=	|(\mathcal{H}\phi^{\#})(\xi)|^2
		&\leq \left(\int_K |\widehat{\phi}(\xi, k\mathbb{M})|\, \dd k \right)^2 \notag \\
		&\leq \int_K |\widehat{\phi}(\xi, k\mathbb{M})|^2\, \dd k.
	\end{align}

	We are now ready to discuss how radialization affects the norm $\|\cdot\|_{\lambda, \sigma}$.
	\begin{proposition}\label{prop:rad}
		Suppose $0\leq \lambda \leq \lambda_{0}$ and $\sigma \in (0,1)$. If $u \in \mathcal{H}_{\lambda,\sigma}$, then $u^{\#} \in \mathcal{H}_{\lambda,\sigma}$ and
			\[
		\|u^{\#}\|_{\lambda, \sigma}\leq \|u\|_{\lambda, \sigma}.
		\]
		Moreover, $\mathcal{C}_c^{\infty, \text{rad}}$ is dense in $\mathcal{H}_{\lambda, \sigma}^{\rad}$.
	\end{proposition}
	
		\begin{proof}
		We begin by showing that the stated inequality holds for $\phi\in \mathcal{C}_{c}^{\infty}$. By the Plancherel formula,
		\begin{align*}
			\|\phi^{\#}\|^2_{\lambda, \sigma}-\|\phi\|^2_{\lambda, \sigma}	&=\|(\Delta^{\sigma}-\lambda^{\sigma})^{\frac{1}{2}}\phi^{\#}\|_2 ^2 -\|(\Delta^{\sigma}-\lambda^{\sigma})^{\frac{1}{2}}\phi\|_2^2\\
			&=\int_K\int_{\mathbb{R}} ((\xi^2+\lambda_0)^{\sigma}-\lambda^{\sigma})\left(|(\mathcal{H}\phi^{\#})(\xi)|^2- |\widehat{\phi}(\xi, k\mathbb{M})|^2\right)\, \frac{\dd \xi}{|\textbf{c}(\xi)|^2}  \, \dd k\\
			&\leq 0,
		\end{align*}
		owing to~\eqref{eq: compare}, since Fubini arguments are justified due to the Paley--Wiener theorem~\eqref{eq: PW}.
		
		Suppose now that $u\in \mathcal{H}_{\lambda,\sigma}$. We claim that 
		\begin{equation}\label{sharpsharpclaim}
		(\phi_{j}) \subseteq \mathcal{C}_{c}^{\infty} \colon \phi_{j} \to u  \mbox{ in } \mathcal{H}_{\lambda, \sigma} \quad \Longrightarrow \quad \phi_{j}^{\#} \to u^{\#} \, \mbox{ in } \mathcal{H}_{\lambda, \sigma}.
		\end{equation}
This will prove both statements, because if~\eqref{sharpsharpclaim} holds, then
\[
\|u^{\#}\|_{\lambda, \sigma} = \lim_{j} \|\phi_{j}^{\#}\|_{\lambda, \sigma}  \leq  \lim_{j} \|\phi_{j}\|_{\lambda, \sigma} = \|u\|_{\lambda, \sigma},
\]
and if $u$ is radial, namely $u=u^{\#}$, then $\phi_{j}^{\#} \to u$ in $\mathcal{H}_{\lambda, \sigma}$. Therefore, it remains to prove the claim.

Let then $(\phi_{j})$ be a sequence in $\mathcal{C}_{c}^{\infty}$ such that $\phi_{j} \to u$ in $\mathcal{H}_{\lambda, \sigma}$. By the above, the sequence $(\phi_{j}^{\#})$ is Cauchy in $\mathcal{H}_{\lambda, \sigma}$; whence it converges to some $v\in \mathcal{H}_{\lambda, \sigma}$. By Theorem~\ref{thm:fracPoincare2},  $\phi_{j} \to u$ in $L^{q}$ for some $q>2$, as well as $(\phi_{j}^{\#}) \to v$ in $L^{q}$. Since convergence in $L^{q}$, up to subsequences, holds almost everywhere too, we obtain that $v=u^{\#}$, and this completes the proof.
		\end{proof}

	\section{The semilinear fractional elliptic equation}\label{Sec:9}

The goal of the section is to prove Theorem~\ref{propeqlambda}. We shall first prove the existence of nonnegative weak solutions to the equation
	\begin{equation}\label{fracequation}
		\Delta^{\sigma} u - \lambda^{\sigma} u - u^{\gamma}=0, \qquad 1<\gamma< \frac{n+2\sigma}{n-2\sigma},
		\end{equation}
	for $0\leq \lambda\leq \lambda_{0}$ and $\sigma \in (0,1)$, in the following sense.
	
	\begin{definition}\label{weak-MS}
	Suppose $0\leq \lambda\leq \lambda_{0}$, $\sigma \in (0,1)$ and $1<\gamma\leq  \frac{n+2\sigma}{n-2\sigma}$. A nonnegative weak solution to equation~\eqref{fracequation} is a function $u\in \mathcal{H}_{\lambda,\sigma}$ such that
	\begin{equation}\label{sol1}
		\langle u ,v\rangle_{\lambda,\sigma} = \int_{\mathbb{H}^n} u^{\gamma}v\, \dd\mu  \qquad \forall v\in \mathcal{H}_{\lambda,\sigma}.
		\end{equation}
	\end{definition}
Notice that the pairing in the right hand side of~\eqref{sol1} is well defined: since $1<\gamma \leq \frac{n+2\sigma}{n-2\sigma}$, $u,v\in L^{\gamma+1}$ by Theorem~\ref{thm:fracPoincare2} and $u^{\gamma} \in L^{(\gamma+1)/\gamma} = L^{(\gamma+1)'}. $

	The first result we shall prove is then the following; only after this, we shall prove that the solution actually matches the regularity stated in Theorem~\ref{propeqlambda}.
	
		\begin{theorem}\label{existence-MS}
			Suppose $0\leq \lambda \leq \lambda_{0}$, $\sigma \in (0,1)$ and $1<\gamma< \frac{n+2\sigma}{n-2\sigma}$. Then~\eqref{fracequation} has at least one nontrivial nonnegative weak solution in $\mathcal{H}_{\lambda,\sigma}$.
		\end{theorem}
	
	We begin by equivalently characterizing the notion of weak solution, wherefrom the name.
	
	\begin{lemma}\label{lem:equiv}
	Suppose $0\leq \lambda\leq \lambda_{0}$, $\sigma \in (0,1)$ and $1<\gamma\leq  \frac{n+2\sigma}{n-2\sigma}$. A nonnegative function $u\in \mathcal{H}_{\lambda,\sigma}$ is a weak solution to equation~\eqref{fracequation} if and only if
	\begin{equation}\label{sol2}
			\int u \cdot (\Delta^{\sigma}-\lambda^{\sigma})\phi \, \dd\mu =  \int_{\mathbb{H}^n} u^{\gamma}\phi\, \dd\mu  \qquad \forall \phi \in \mathcal{C}_{c}^{\infty}.
		\end{equation}
	\end{lemma}
\begin{proof}
Suppose $\phi \in \mathcal{C}_{c}^{\infty}$ and $(\phi_{j}) \subset \mathcal{C}_{c}^{\infty}$ such that $\phi_{j} \to u$ in $\mathcal{H}_{\lambda,\sigma}$. Then
		\begin{equation}\label{scalarprodsmooth}
			\begin{split}
				\langle u ,\phi\rangle_{\lambda,\sigma} 
				&= \lim_{j} \langle \phi_{j} ,\phi\rangle_{\lambda,\sigma} \\
				&= \lim_{j} \int (\Delta^{\sigma}-\lambda^{\sigma})^{1/2} \phi_{j} \cdot (\Delta^{\sigma}-\lambda^{\sigma})^{1/2} \phi \, \dd\mu\\
				& = \lim_{j} \int \phi_{j} \cdot (\Delta^{\sigma}-\lambda^{\sigma}) \phi \, \dd\mu =\int u \cdot (\Delta^{\sigma}-\lambda^{\sigma}) \phi \, \dd\mu,
			\end{split}
		\end{equation}
		the last step because $\phi_{j} \to u$ in $L^{\gamma+1}$ by Theorem~\ref{thm:fracPoincare2} and $ (\Delta^{\sigma}-\lambda^{\sigma}) \phi  \in L^{(\gamma+1)'}$ by Proposition~\ref{Prop 4.3}. This shows that if~\eqref{sol1} holds, then~\eqref{sol2} holds too.
		
		On the other hand, suppose~\eqref{sol2} holds, $v \in\mathcal{H}_{\lambda,\sigma}$, $(\phi_{j}), (\psi_{j})\subset \mathcal{C}_{c}^{\infty}$ such that $\phi_{j} \to u$ and $\psi_{j} \to v$ in $\mathcal{H}_{\lambda,\sigma}$. Then, using~\eqref{scalarprodsmooth} and the fact that $u\in L^{\gamma+1}$ whence $u^{\gamma}\in L^{\frac{\gamma+1}{\gamma}} = L^{(\gamma+1)'}$, and $\psi_{j} \to v$ in $L^{\gamma+1}$,
		\begin{align*}
			\int u^{\gamma} v \, d\mu
			& = \lim_{j} \int u^{\gamma} \psi_{j} \, d\mu = \lim_{j} \int u  \cdot (\Delta^{\sigma}-\lambda^{\sigma}) \psi_{j} \, d\mu = \lim_{j} \langle u,\psi_{j}\rangle_{\lambda,\sigma}  = \langle u,v\rangle_{\lambda,\sigma},
		\end{align*}
		which is~\eqref{sol1}.
\end{proof}

	Let us now define some functionals and some spaces which will be of use in the following.
	
	\begin{definition}
Suppose $0\leq \lambda\leq \lambda_{0}$, $\sigma \in (0,1)$ and $1<\gamma\leq  \frac{n+2\sigma}{n-2\sigma}$. We denote by $I$ and $J$ the functionals on $\mathcal H_{\lambda,\sigma}$ given by
\[
	I(u) = \frac{\|u\|_{\lambda,\sigma}^{2}}{\|u\|_{\gamma+1}^{2}}, \qquad 	J(u) = \frac{1}{2}\|u\|_{\lambda,\sigma}^{2} -\frac{1}{\gamma+1}\|u\|_{\gamma+1}^{\gamma+1}.
	\]
	We also define the Nehari manifold
	\[
	\mathcal{N} = \{u \in \mathcal H_{\lambda,\sigma}\setminus\{0\} \colon \|u\|_{\lambda,\sigma}^{2} = \|u\|_{\gamma+1}^{\gamma+1} \}, 
	\]
	and its submanifolds
	\[
\mathcal{N}^{\rad}_{\geq 0} =  \{u\in \mathcal{N} \colon u \geq 0, u \mbox{ radial} \},	\qquad		\mathcal{N}^{\rad}_{ \pm} =  \{u\in \mathcal{N} \colon u \geq 0 \mbox{ or } u \leq 0, \, u \mbox{ radial} \}.
			\]
	\end{definition}
Observe that $\mathcal{N}$, $\mathcal{N}^{\rad}_{\geq 0}$ and  $\mathcal{N}^{\rad}_{ \pm} $, endowed with the distance induced by the norm of $\mathcal{H}_{\lambda,\sigma}$, are complete metric spaces: indeed, sign, radiality and normalization are preserved by limits (by Theorem~\ref{thm:fracPoincare2}, convergence with respect to the norm of $\mathcal{H}_{\lambda,\sigma}$ implies, up to subsequence, convergence almost everywhere). Moreover, a sequence in $\mathcal{N}$ cannot converge to zero, since again by Theorem~\ref{thm:fracPoincare2} there is $C>0$ such that $\|u\|_{\gamma+1} \geq C$ for all $u\in \mathcal{N}$.

\smallskip

	Let us notice that when $u$ belongs to $\mathcal{N}$ then
	\begin{equation}\label{IJ}
		I(u) = \alpha_{\gamma}^{-2\alpha_{\gamma}} J(u)^{ {2\alpha_{\gamma}} }, \qquad \alpha_{\gamma}= \frac{1}{2}- \frac{1}{\gamma+1}, \quad u\in \mathcal{N}.
	\end{equation}
	
	For the next lemma, recall that a functional $E$ on $\mathcal{H}_{\lambda,\sigma}$ is said to be differentiable at $u\in \mathcal{H}_{\lambda,\sigma}$ if there exists an element $E'(u) \in \mathcal{H}_{\lambda,\sigma}$ such that
\[
		E(w+u) = E(u) +\langle E'(u),w\rangle_{\lambda,\sigma} + o(\|w\|_{\lambda,\sigma}) \qquad \forall w \in \mathcal{H}_{\lambda,\sigma}.
\]
	We say that $E$ is $\mathcal{C}^{1} = \mathcal{C}^{1}( \mathcal{H}_{\lambda,\sigma})$ if $E$ is differentiable at every point of $\mathcal{H}_{\lambda,\sigma}$ and $E' \colon  \mathcal{H}_{\lambda,\sigma} \to  \mathcal{H}_{\lambda,\sigma}$ is continuous.

	\begin{lemma}\label{J'}
		Suppose $0\leq \lambda \leq \lambda_{0}$, $\sigma \in (0,1)$ and $1<\gamma\leq \frac{n+2\sigma}{n-2\sigma}$. Then $J\in \mathcal{C}^{1}$ and $J'(u) = u - T(|u|^{\gamma-1}u)$ for all $u\in \mathcal H_{\lambda,\sigma}$; equivalently, 
		\[
		\langle J'(u),v\rangle_{\lambda,\sigma} = \langle u,v\rangle_{\lambda,\sigma} - \int_{\mathbb{H}^n} |u|^{\gamma-1}uv\, \dd\mu, \qquad u,v\in \mathcal  H_{\lambda,\sigma}.
		\]
	\end{lemma}
	\begin{proof}
		Write
		\[
		J_{1}(u) =  \frac{1}{2}\|u\|_{\lambda,\sigma} , \qquad J_{2}(u) =  \frac{1}{\gamma+1}\|u\|_{\gamma+1}^{\gamma+1}.
		\]
		Since for $u,w\in \mathcal{H}_{\lambda, \sigma}$
		\[
		J_{1}(w+u) = \frac{1}{2}\|w+u\|_{\lambda, \sigma}^{2} = \frac{1}{2}\|u\|_{\lambda, \sigma}^{2} + \langle u, w\rangle_{\lambda, \sigma} + \frac{1}{2}\|w\|_{\lambda, \sigma}^{2}  
		\]
		one sees that $J_{1}$ is differentiable at $u$ and $J_{1}'(u) =u$, whence $J_{1}\in \mathcal{C}^{1}$.
		
As for $J_{2}$, recall that $u^{\gamma} \in L^{(\gamma+1)/\gamma} = L^{(\gamma+1)'} \subseteq \mathcal{H}_{\lambda,\sigma}^{-1}$ thanks to Theorem~\ref{thm:fracPoincare2}, the duality being given by~\eqref{qduality}. We write $[u^{\gamma}]^{*}$ to distinguish between $u^{\gamma}$ and the functional $[u^{\gamma}]^{*} \in \mathcal{H}_{\lambda,\sigma}^{-1}$.  By the identity, for any differentiable $f$ and $F$ such that $F'=f$, and $a,b\in \R$,
\begin{equation}\label{Fab}
		F(a+b) = F(a) + f(a) b + \int_{0}^{1}b^{2} (1-s) f'(a+sb)\,\dd s
\end{equation}
	we deduce, for $w\in  \mathcal{H}_{\lambda, \sigma}$ (and $f(x)=|x|^{\gamma-1}x$ and $F(x) = \int_{0}^{x}f(t)\, \dd t =\frac{1}{\gamma+1} |x|^{\gamma+1}$)
		\begin{align*}
		J_{2}(u+w)
			&=  \int_{\mathbb{H}^n} F(u+w) \, \dd\mu\\
			& = \int_{\mathbb{H}^n}( F(u) + f(u)w)\, \dd\mu + R(u,w)\\
			& = J_{2}(u) + [f(u)]^{*}(w)+ R(u,w) = J_{2}(u) + \langle T(f(u)), w\rangle_{\lambda, \sigma} + R(u,w)
		\end{align*}
		where we used Lemma~\ref{lemmastar} in the last equality ( {recall $T\colon \mathcal{H}_{\lambda, \sigma}^{-1}\rightarrow \mathcal{H}_{\lambda, \sigma}$, $T(v^{*})=v$}). By~\eqref{Fab}, H\"older's inequality and Theorem~\ref{thm:fracPoincare}	
		\begin{align*}
			R(u,w) 
			&\leq C \int_{\mathbb{H}^n} (|u|^{\gamma-1} + |w|^{\gamma-1})|w|^{2}\, \dd\mu\\
			& \leq C \int_{\mathbb{H}^n}  |w|^{\gamma+1}\dd\mu + C \||u|^{\gamma-1}\|_{\frac{\gamma+1}{\gamma-1}} \||w|^{2}\|_{\frac{\gamma+1}{2}}\\
			& \leq C ( \|w\|_{\gamma+1}^{\gamma+1} + \|u\|_{\gamma+1}^{\gamma-1} \|w\|_{\gamma+1}^{2})\\
			&  = o(\|w\|_{\lambda, \sigma}).
		\end{align*}
		This shows that $J'_{2}(u) = T(f(u))$. It remains to prove that $J'_{2}$ is continuous.
		
Assume that $u,w\in \mathcal{H}_{\lambda, \sigma}  $ are such that $\|u\|_{\lambda, \sigma}, \|w\|_{\lambda,\sigma}\leq C$. Then by Lemma~\ref{lemmastar}
		\begin{align*}
			\| J'_{2}(u) - J'_{2}(w)\|_{{\lambda,\sigma}} 
			&= \| T(f(u))- T(f(w))\|_{\lambda,\sigma}\\
			& \leq \| f(u)- f(w)\|_{\mathcal{H}^{-1}_{\lambda,\sigma}} \\
			&\leq C\| f(u)- f(w)\|_{\frac{\gamma+1}{\gamma}}
		\end{align*}
		the last inequality by the embedding $L^{ \frac{\gamma+1}{\gamma}} \subseteq \mathcal{H}_{\lambda,\sigma}^{-1}$. 	Observe now that since  $f(b)-f(a) = (b-a) \int_{0}^{1} f'(a+s(b-a))\, ds$, one gets
		\begin{equation}\label{fbfa}
		\begin{split}
			|f(b)-f(a)| 
			&\leq |b-a|\int_{0}^{1}|f'(a+s(b-a))|\, \dd s \\
			& \leq C|b-a|(|a|^{\gamma-1} + |b-a|^{\gamma-1}) \leq C|b-a|(|a|^{\gamma-1} + |b|^{\gamma-1}),
		\end{split}
		\end{equation}
		whence for some $\eta>0$
		\begin{equation}\label{aslater}
		\begin{split}
			\| f(u)- f(w)\|_{\frac{\gamma+1}{\gamma}} 
			&\leq C \bigg( \int_{\mathbb{H}^n} \Big(|u|^{\gamma-1} + |w|^{\gamma-1})|u-w|\Big)^{\frac{\gamma+1}{\gamma}}\bigg)^{\frac{\gamma}{\gamma+1}}\\
			& \leq C\|(|u|^{\gamma-1} + |w|^{\gamma-1})^{\frac{\gamma+1}{\gamma}}\|_{\frac{\gamma}{\gamma-1}}^{\frac{\gamma}{\gamma+1}} \||u-w|^{\frac{\gamma+1}{\gamma}}\|_{\gamma}^{\frac{\gamma}{\gamma+1}} \\
			& = C(\|u\|_{\gamma+1}^{\eta} +\|w\|_{\gamma+1}^{\eta}) \|u-w\|_{\gamma+1}\\
			& \leq C \|u-w\|_{\lambda,\sigma}
			\end{split}
		\end{equation}
		where we used once more Theorem~\ref{thm:fracPoincare}. We conclude that $J'_{2}$ is Lipschitz continuous on bounded sets, whence $J_{2}\in \mathcal{C}^{1}$.
	\end{proof}

	\begin{lemma}\label{Effielemma}
		Suppose $0\leq \lambda \leq \lambda_{0}$, $\sigma \in (0,1)$ and $1<\gamma\leq \frac{n+2\sigma}{n-2\sigma}$. If $u \in \mathcal{H}_{\lambda,\sigma}^{\rad}$, then
		\[
		\langle J'(u), v\rangle_{\lambda,\sigma} = \langle J'(u), v^{\#}\rangle_{\lambda,\sigma} \qquad \forall v\in \mathcal{H}_{\lambda,\sigma},
		\]
		whence $\|J'(u)\|_{\lambda,\sigma} = \sup_{v\in  \mathcal{H}_{\lambda,\sigma}^{\rad} \setminus\{0\}} \langle J'(u),\frac{v}{\|v\|_{\lambda,\sigma}}\rangle_{\lambda,\sigma} $.
	\end{lemma}
	
	\begin{proof}
		We first show that the lemma holds true for test functions, whence the full claim for $\mathcal{H}_{\lambda, \sigma}$ will follow by density. 
		In other words, let us first prove that 
		\begin{equation}\label{eq: JderCc}
			\langle J'(\phi),  {\psi} \rangle_{\lambda, \sigma}=\langle J'( {\phi}),  {\psi}^{\#} \rangle_{\lambda, \sigma}, \quad  {\phi},  {\psi}\in \mathcal{C}_c^{\infty}, \,  {\phi} \text{ radial.}
		\end{equation}
		As we have shown in Proposition~\ref{Prop 4.3}, if $ {\phi}\in\mathcal{C}_c^{\infty}$ then the function $\Delta^{\sigma} {\phi}\in L^p$, $1\leq p\leq \infty$; in addition, it is radial if $ {\phi}$ is radial. Hence, we can write
		\begin{equation*}
		\Delta^{\sigma} {\phi}(x)-\lambda^{\sigma} {\phi}(x)=\int_K (\Delta^{\sigma} {\phi}(kx)- \lambda^{\sigma} {\phi}(kx)) \, \dd k .
		\end{equation*}
		It follows that Fubini arguments are justified, so 
		\begin{equation}\label{eq: Obs 0}
		\begin{split}
			\langle  {\phi}, {\psi}\rangle_{\lambda, \sigma}&= \int_G (\Delta^{\sigma} {\phi}(x)-\lambda^{\sigma} {\phi}(x))\,  {\psi}(x)\, \dd x  \\
			&=\int_G \int_K (\Delta^{\sigma} {\phi}(kx)- \lambda^{\sigma} {\phi}(kx)) \,  {\psi}(x)\, \dd k \, \dd x  \\
			&=\int_G (\Delta^{\sigma} {\phi}(x)- \lambda^{\sigma} {\phi}(x)) \, \int_K  {\psi}(kx)\, \dd k \, \dd x \\
			&=\langle  {\phi}, {\psi}^{\#}\rangle_{\lambda, \sigma}. 
		\end{split}
		\end{equation}
		Similarly, one can show that 
		\begin{equation}\label{eq: Obs 0'}
			\int_{G}| {\phi}|^{\gamma-1}\,  {\phi}\,  {\psi}\, d\mu=\int_{G}| {\phi}|^{\gamma-1}\,  {\phi}\,  {\psi}^{\#} \, d\mu,
		\end{equation}
		which concludes the proof of~\eqref{eq: JderCc} by Lemma~\ref{J'}.
		
		We now pass to the whole space $\mathcal{H}_{\lambda, \sigma}$. First, we show that for any $u, v\in \mathcal{H}_{\lambda, \sigma}$ with $u$ radial, we have
		\begin{equation}\label{eq: half1}  
			\langle u,v \rangle_{\lambda, \sigma}= \langle u,v^{\#} \rangle_{\lambda, \sigma}.
		\end{equation}
		Let the sequences $ {\phi}_j,  {\psi}_j\in \mathcal{C}_c^{\infty}$,  with $ {\phi}_j$ in addition radial (owing to  Proposition~\ref{prop:rad}) be  such that
		\[
		 {\phi}_j \to u, \qquad  {\psi}_j \to v \qquad \text{in }\, \mathcal{H}_{\lambda, \sigma}.
		\]
		This implies that 
		\[
		\langle  {\phi}_j, {\psi}_j \rangle_{\lambda, \sigma} \rightarrow \langle u,v \rangle_{\lambda, \sigma},
		\]
		while by~\eqref{eq: Obs 0} and~\eqref{sharpsharpclaim} we obtain
		\[
		\langle  {\phi}_j, {\psi}_j \rangle_{\lambda, \sigma}= \langle  {\phi}_j, {\psi}_j^{\#} \rangle_{\lambda, \sigma}\rightarrow \langle u,v^{\#} \rangle_{\lambda, \sigma},
		\] 
		so~\eqref{eq: half1} is proved. 
		
		Thus it remains to prove that for all $u, v\in \mathcal{H}_{\lambda, \sigma}$ with $u$ radial, it holds
		\begin{equation} \label{eq: int p}
			\int_{\mathbb{H}^n}|u|^{\gamma}\, u\, v\, d\mu= \int_{\mathbb{H}^n}|u|^{\gamma}\, u\, v^{\#}\, d\mu.
		\end{equation}
		Notice first that~\eqref{eq: int p} holds for $ {u=\phi}\in \mathcal{C}_{c}^{\infty}$ radial and $v\in \mathcal{H}_{\lambda, \sigma}$. Indeed, let $( {\psi}_j)\subseteq \mathcal{C}_c^{\infty}$ be such that $ {\psi}_j\rightarrow v$ in $\mathcal{H}_{\lambda, \sigma}$, hence also in $L^{\gamma+1}$ due to the embedding. Then
		\[
		\left| 	\int_{\mathbb{H}^n} | {\phi}|^{\gamma-1}\,  {\phi}\,  {\psi}_j- \int_{\mathbb{H}^n} | {\phi}|^{\gamma-1}\,  {\phi}\, v \right|\leq  \int_{\mathbb{H}^n}| {\phi}|^{\gamma}\, | {\psi}_j-v|\leq \|| {\phi}|^{\gamma}\|_{(\gamma+1)'}\, \| {\psi}_j-v\|_{\gamma+1},
		\]
		which, combined with~\eqref{eq: Obs 0'} for test functions, allows us to conclude. Therefore, to show~\eqref{eq: int p}, it suffices to prove that for $( {\phi}_j)\subseteq \mathcal{C}_c^{\infty}$ such that $ {\phi}_j\rightarrow u$ in $\mathcal{H}_{\lambda, \sigma}$, and for all $w\in \mathcal{H}_{\lambda, \sigma}$,
		\begin{equation*}
			\int_{\mathbb{H}^n} | {\phi}_j|^{\gamma-1}\,  {\phi}_j\, w \rightarrow \int_{\mathbb{H}^n} |u|^{\gamma-1}\, u\, w
		\end{equation*}
		(and then apply it to $w=v$ and $w=v^{\#}$, owing to the previous observation). Now since $w\in \mathcal{H}_{\lambda, \sigma}\subseteq L^{\gamma+1}$, it is enough to prove that 
		\[
		| {\phi}_j|^{\gamma}\,  {\phi}_j \rightarrow |u|^{\gamma}\, u \quad \text{in } (L^{\gamma+1})'= L^{\frac{\gamma+1}{\gamma}}.
		\]
But this can be obtained by arguing as in~\eqref{aslater}, by means of~\eqref{fbfa} and H\"older's inequality.
	\end{proof}


		\begin{proposition}\label{Ekeland-PS}
			There exists a sequence $(u_{j})$ in $\mathcal{N}^{\rad}_{\geq 0}$ which is minimizing for $J$ (i.e., such that $J(u_{j}) \to \inf_{u\in \mathcal{N}^{\rad}_{\geq 0}} J(u)$) and such that $J'(u_{j}) \to 0$ in $\mathcal{H}_{\lambda,\sigma}$.
		\end{proposition}

		\begin{proof}
	We begin by observing that
		\[
	 \inf_{u\in \mathcal{N}^{\rad}_{\geq 0}} J(u) = \inf_{u\in \mathcal{N}^{\rad}_{\pm}} J(u)
		\]
		since $J(u)= J(-u)$ for all $u\in \mathcal{H}_{\lambda,\sigma}$.	 Fix a sequence $(\epsilon_{j}) $ such that $\epsilon_{j} \to 0$. By the Ekeland principle~\cite[Chapter I, Theorem 5.1]{Struwe} and reasoning as in the beginning of the proof of~\cite[Chapter I, Corollary 5.3]{Struwe} we can find a minimizing sequence $(u_{j})$ in $\mathcal{N}^{\rad}_{\pm }$ such that
			\[
			J(u_{j} ) \leq J(u_{j} + w) + \epsilon_{j}\|w\|_{\lambda,\sigma} \qquad \forall w \in \mathcal{N}^{\rad}_{\pm}.
			\]
			By this and Lemma~\ref{J'}, being $J$ differentiable at $u_{j}$, we obtain
			\[
			J(u_{j} ) \leq J(u_{j}) + \langle J'(u_{j}),w\rangle_{\lambda,\sigma} + o(\|w\|_{\lambda,\sigma}) + \epsilon_{j}\|w\|_{\lambda,\sigma} \qquad \forall w \in \mathcal{N}^{\rad}_{\pm},
			\]
			namely
			\[
			0 \leq  \langle J'(u_{j}),w\rangle_{\lambda,\sigma} + o(\|w\|_{\lambda,\sigma}) + \epsilon_{j}\|w\|_{\lambda,\sigma} \qquad \forall w \in \mathcal{N}^{\rad}_{\pm}.
			\]
			By considering $-w$ in place of $w$ (observe that $- w \in \mathcal{N}^{\rad}_{\pm}$ for all $w\in \mathcal{N}^{\rad}_{\pm}$) we get also
			\[
			0 \leq  \langle J'(u_{j}),- w\rangle_{\lambda,\sigma} + o(\|w\|_{\lambda,\sigma}) + \epsilon_{j}\|w\|_{\lambda,\sigma} \qquad \forall w \in \mathcal{N}^{\rad}_{\pm},
			\]
			and by combining these we get
			\[
			| \langle J'(u_{j}), w\rangle_{\lambda,\sigma}| \leq  o(\|w\|_{\lambda,\sigma}) + \epsilon_{j}\|w\|_{\lambda,\sigma}  \qquad \forall w \in \mathcal{N}^{\rad}_{\pm}.
			\]
			By homogeneity, since for all nonzero $w \in \mathcal{H}_{\lambda,\sigma}$ there is $\beta \neq 0$ such that $\beta w\in \mathcal{N}$, we get 
			\[
			| \langle J'(u_{j}), w\rangle_{\lambda,\sigma}| \leq  o(\|w\|_{\lambda,\sigma}) + \epsilon_{j}\|w\|_{\lambda,\sigma}  \qquad \forall w \in \mathcal{H}^{\rad}_{\lambda, \sigma, \pm},
			\]
			where $\mathcal{H}^{\rad}_{\lambda, \sigma, \pm} = \{u  \in \mathcal{H}^{\rad}_{\lambda, \sigma}: \, u\geq 0 \mbox{ or } u\leq 0 \}$.
			Let now $w \in \mathcal H^{\rad}_{\lambda,\sigma}$.  Then $w = w^{+} + w^{-}$ where
			\[
			w^{+} = \max(0,w) = \frac{w + |w|}{2}, \qquad w^{-} = -\max(0,-w) = \frac{w- |w|}{2},
			\]
			and since $w\in \mathcal H_{\lambda,\sigma}$, then  $|w|\in \mathcal H_{\lambda,\sigma}$ by Corollary~\ref{cor:abs}; whence $w^{+},w^{-} \in \mathcal H_{\lambda,\sigma}$, are radial and do not change sign; in other words, $w^{+},w^{-} \in \mathcal{H}^{\rad}_{\lambda, \sigma, \pm} $. Therefore,
			\[
			| \langle J'(u_{j}), w\rangle_{\lambda,\sigma}| \leq 2  \|w\|_{\lambda,\sigma} (\epsilon_{j} + o(1)) \qquad  \forall w\in \mathcal{H}_{\lambda,\sigma}^{\rad}.
			\]
			By Lemma~\ref{Effielemma}, we get
			\[
			| \langle J'(u_{j}), w\rangle_{\lambda,\sigma}| \leq 2  \|w\|_{\lambda,\sigma} (\epsilon_{j} + o(1)) \qquad  \forall w\in \mathcal{H}_{\lambda,\sigma},
			\]
			which means $J'(u_{j}) \to 0$. Then we have a sequence $(u_{j})$ in $\mathcal{N}^{\rad}_{\pm}$ which minimizes $J$ and such that $J'(u_{j}) \to 0$. Finally, consider the sequence $(|u_{j}|)_{j}$ which then belongs to $\mathcal{N}^{\rad}_{\geq 0}$, and either $|u_{j}|= u_{j}$ or $|u_{j}|=-u_{j}$; this does not affect the behavior of $J$ or $J'$ on the sequence, i.e.\ $J(|u_{j}|) =J(u)$ and $J'(|u_{j}|) \to 0$, as $J'(-u_{j}) = -J'(u_{j})$.
		\end{proof}
		
		We are now ready to prove Theorem~\ref{existence-MS}, inspired by~\cite[Theorem 6.2]{DuttaSandeep}.

		\begin{proof}[Proof of Theorem~\ref{existence-MS}]
Consider
			\[
			S = S_{\lambda,\sigma} = \inf_{u\in \mathcal{H}_{\lambda,\sigma} \setminus\{0\}} I(u).
			\]
			First, we observe that $I$ is invariant under scaling, namely $I(u) = I (\beta u)$ for all $u$ and all $\beta \neq 0$. Therefore
			\[
			S= \inf_{u\in \mathcal{N}} I(u).
			\]
By Proposition~\ref{prop:rad} and Corollary~\ref{cor:abs}, moreover,
\begin{equation}\label{SNrad+}
			S =  \inf_{u\in \mathcal{N}^{\rad}_{\geq 0}} I(u).
\end{equation}
			Let now $(u_{j}) \subseteq \mathcal{N}^{\rad}_{\geq 0}$ be the minimizing sequence for $J$ given by Proposition~\ref{Ekeland-PS}. By~\eqref{SNrad+} and~\eqref{IJ}, 
\[
S=  \inf_{u\in \mathcal{N}^{\rad}_{\geq 0}} I(u) =  \alpha_{\gamma}^{-2\alpha_{\gamma}}   \inf_{u\in \mathcal{N}^{\rad}_{\geq 0}}  J(u)^{ {2\alpha_{\gamma}} } =   \alpha_{\gamma}^{-2\alpha_{\gamma}}   \lim_{j} J(u_{j})^{ {2\alpha_{\gamma}} } = \lim_{j}I(u_{j})
\]
namely $(u_{j})$ is a minimizing sequence for $I$ on $ \mathcal{H}_{\lambda,\sigma}$, such that $J'(u_{j})\to 0$ in $\mathcal{H}_{\lambda,\sigma}$.
			
			Since $J'(u_{j}) \to 0$ in $\mathcal{H}_{\lambda,\sigma}$, by Lemma~\ref{J'}, for all $\epsilon>0$ there is $j_{0}$ such that for all $j\geq j_{0}$ and all $v\in \mathcal{H}_{\lambda,\sigma}$
			\begin{align}\label{Ek}
				|\langle J'(u_{j}), v\rangle_{\lambda,\sigma}| 
				&= \bigg|\langle u_{j},v\rangle_{\lambda,\sigma} - \int_{\mathbb{H}^n} u_{j}^{\gamma}v\, \dd \mu\bigg| \leq \epsilon \|v\|_{\lambda,\sigma} \qquad \forall v\in \mathcal{H}_{\lambda,\sigma}.
			\end{align}
			Since  $I(u_{j}) \to S$, and $(u_{j}) \in \mathcal{N}$, we have
			\[
			\|u_{j}\|_{\lambda,\sigma} \to S^{\frac{\gamma+1}{2(\gamma-1)}}, \qquad \|u_{j}\|_{\gamma+1} \to S^{\frac{1}{\gamma-1}}.
			\]	
			In particular, $(u_{j})$ is a bounded sequence in $\mathcal{H}_{\lambda,\sigma}^{\rad}$. By the Banach--Alaoglu theorem, up to subsequences we can suppose $u_{j} \rightharpoonup u $ for some $u\in \mathcal{H}_{\lambda,\sigma}$. By Theorem~\ref{thm: cpt radial emb}, again up to subsequences, $u_{j} \to u$ in $L^{\gamma+1}$ and pointwise almost everywhere. In particular, $u \geq 0$ almost everywhere. Moreover
			\[
			\|u\|_{\gamma+1} = \lim_{j} \|u_{j}\|_{\gamma+1} = S^{1/(\gamma-1)} >0
			\]
			from which $u$ is not the zero function. By~\eqref{Ek}, passing to the limit, we get
			\[
			\langle u ,v\rangle_{\lambda,\sigma} = \int_{\mathbb{H}^n} u^{\gamma}v\, \dd \mu  \qquad \forall v\in \mathcal{H}_{\lambda,\sigma},
			\]
			namely $u$ is a nonnegative weak solution of~\eqref{fracequation}.
		\end{proof}

		\subsection{Boundedness of the solutions}\label{section: bddness}
	We shall now discuss some properties of the nonnegative weak solutions to equation~\eqref{fracequation}, proving that they are actually bounded.
	 	
We begin with a technical lemma; but before doing so, for $0\leq \lambda \leq \lambda_{0}$, let us denote by $k_{\lambda,\sigma}$ the convolution kernel of the operator $(\Delta^{\sigma} - \lambda^{\sigma})^{-1}$ and let us split $k_{\lambda,\sigma}$ as we did for $k_{\sigma}$ in~\eqref{k0infty}:
\begin{equation}\label{k0inftylambda}
	k_{\lambda,\sigma}^{0}=\mathbbm{1}_{B(o,1)} \, k_{\lambda, \sigma}, \qquad k_{\lambda,\sigma}^{\infty}=\mathbbm{1}_{B(o,1)^{c}} \, k_{\lambda, \sigma}
	\end{equation}
(notice that $k_{\lambda_{0},\sigma} = k_{\sigma}$). Since
\begin{equation}\label{dominance}
k_{\lambda, \sigma}(x)=\int_0^{\infty}\e^{\lambda^{\sigma}t}\, P_t^{\sigma}(x)\, \dd t \leq \int_0^{\infty}\e^{\lambda_0^{\sigma}t}\, P_t^{\sigma}(x)\, \dd t=k_{\sigma}(x),
\end{equation}
we get the following consequence of Corollary~\ref{corkaest}.
	\begin{remark}\label{rem:corkaest}
	Suppose $0\leq \lambda \leq \lambda_{0}$ and $\sigma \in (0,1)$. Then $k^{0}_{\lambda, \sigma} \in L^{p}$ for all $p\in \big[1,\frac{n}{n-2\sigma}\big)$ while $k^{\infty}_{\lambda, \sigma} \in L^{p}$ for all $p\in (2,\infty]$.
	\end{remark}

		\begin{lemma}\label{lemma: dixit}
Suppose $0\leq \lambda \leq \lambda_{0}$, $0<\sigma<1$ and $1<\gamma\leq \frac{n+2\sigma}{n-2\sigma}$ and let $u$ be an element of $\mathcal{H}_{\lambda,\sigma}$. Then 
		\[
 u^{\gamma} \cdot ( |\psi|\ast k_{\lambda,\sigma})\in L^{1} \qquad \text{for all} \quad \psi \in (\Delta^{\sigma}-\lambda^{\sigma})(\mathcal{C}_{c}^{\infty}).
		\]
		\end{lemma}
	
	\begin{proof}
Pick $u\in \mathcal{H}_{\lambda,\sigma}$. Since $u \in L^{q}$ for all $2<q\leq \frac{2n}{n-2\sigma}$ by Theorem~\ref{thm:fracPoincare2}, we have
				\[
		u^{\gamma}\in L^Q, \quad  1<\frac{2}{\gamma}<Q:=\frac{q}{\gamma}\leq \frac{2n}{\gamma(n-2\sigma)}.
		\]
		Recall also that if $\psi \in (\Delta^{\sigma}-\lambda^{\sigma})(\mathcal{C}_{c}^{\infty})$, then $\psi\in L^p$ for all $1\leq p\leq \infty$ by Proposition~\ref{Prop 4.3}.
	We write $k_{\lambda,\sigma}$ as in~\eqref{k0inftylambda}, and use Remark~\ref{rem:corkaest}.
	
	Let us first treat the part at infinity. Given that $1<\frac{2}{\gamma}<2$, it follows from Young's inequality and the facts that $\psi\in L^1$ and $k_{\lambda,\sigma}^{\infty} \in L^{p}$ for $p\in (2,\infty]$ that
	$$|\psi|\ast k_{\lambda,\sigma}^{\infty} \in L^p, \quad p\in (2, \infty].$$
		Since $2/{\gamma}<2$, there is $Q<2$ such that $u^{\gamma}\in L^Q$; then $|\psi|\ast k_{\lambda,\sigma}^{\infty} \in L^{Q'}$ and		\[
		\langle u^{\gamma}, |\psi|\ast k_{\lambda,\sigma}^{\infty} \rangle <+\infty.
		\]
		We now pass to the local part. By Young's inequality again, since $k_{\lambda,\sigma}^{0}\in L^1$, we get that $|\psi|\ast k_{\lambda,\sigma}^{0}\in L^p$ for all $p\in [1, \infty]$. Then it becomes clear that 
		\[
		\langle u^{\gamma}, |\psi|\ast k_{\lambda,\sigma}^{0} \rangle <+\infty.
		\]
		
		Altogether, 
		\[
		\langle u^{\gamma}, |\psi|\ast k_{\lambda,\sigma} \rangle <+\infty, \quad \forall \psi \in (\Delta^{\sigma}-\lambda^{\sigma})(\mathcal{C}_{c}^{\infty}),
		\]
		which completes the proof.
	\end{proof}

\begin{proposition}\label{uL1loc}
Suppose $0\leq \lambda \leq \lambda_{0}$, $\sigma\in (0,1)$ and $1<\gamma< \frac{n+2\sigma}{n-2\sigma}$ and let $u$ be a nonnegative weak solution of~\eqref{fracequation}. Then $u= u^{\gamma}*k_{\lambda,\sigma}$ as $L^{1}_{\mathrm{loc}}$ functions, hence almost everywhere.
	\end{proposition}	
\begin{proof}
Since $2/{\gamma}<2$, there is $Q<2$ such that $u^{\gamma}\in L^Q$; moreover $k_{\lambda,\sigma}^{\infty} \in L^{Q'}$ and $k_{\lambda,\sigma}^{0} \in L^{1}$ by  Remark \ref{rem:corkaest}. Then by Young's inequality 
\[
 u^{\gamma}*k_{\lambda,\sigma} = u^{\gamma}*k_{\lambda,\sigma}^{0} + u^{\gamma}*k_{\lambda,\sigma}^{\infty} \in L^{Q} +L^{\infty} \subseteq  L^{1}_{\mathrm{loc}}.
 \]
Since also $u\in L^{1}_{\mathrm{loc}}$, it will be enough to prove that 
\begin{equation}\label{eqtestL1loc}
	\langle u,\phi \rangle = \langle u^{\gamma}*k_{\lambda,\sigma},\phi \rangle \qquad \forall \phi \in \mathcal{C}_c^{\infty}.
\end{equation}
Recall that for $\phi\in \mathcal{C}_c^{\infty}$, we have
\[
		(\Delta^\sigma-\lambda^{\sigma})^{-\frac{1}{2}}\phi\in L^2,\, \quad  \phi *k_{\lambda,\sigma} = (\Delta^\sigma-\lambda^{\sigma})^{-1}\phi\in \mathcal{H}_{\lambda, \sigma},
\]
	the first due to~\eqref{eq: negPlanch} and the second follows by this and the definition of $\mathcal{H}_{\lambda, \sigma}$.
	
We have
	\begin{align}
\langle u^{\gamma}\ast k_{\lambda,\sigma}, \phi \rangle
		&=\langle u^{\gamma}, \phi\ast k_{\lambda,\sigma} \rangle \notag\\
		&=\langle u, \phi*k_{\lambda,\sigma}\rangle_{\lambda, \sigma}. \label{eq: D'1}
	\end{align}
	The first equality follows from Lemma~\ref{lemma: dixit} and Fubini, while the second from the facts that $\phi*k_{\lambda,\sigma}\in \mathcal{H}_{\lambda, \sigma}$ and $u$ is a weak solution. We claim now that 
	\[
	\langle u, \phi*k_{\lambda,\sigma}\rangle_{\lambda, \sigma}= \langle u, \phi\rangle.
	\]
	Indeed, let $(\psi_j)\subseteq \mathcal{C}_c^{\infty}$ be such that $\psi_j\rightarrow u \in \mathcal{H}_{\lambda, \sigma}$. Then
	\begin{align}
		\langle u, \phi*k_{\lambda,\sigma}\rangle_{\lambda, \sigma}&=\lim_{j\rightarrow \infty} \langle \psi_j, (\Delta^\sigma-\lambda^{\sigma})^{-1}\phi\rangle_{\lambda, \sigma} \notag \\
		&=\lim_{j\rightarrow \infty} \langle (\Delta^\sigma-\lambda^{\sigma})^{\frac{1}{2}}\psi_j, (\Delta^\sigma-\lambda^{\sigma})^{-\frac{1}{2}}\phi\rangle \notag \\
		&=\lim_{j\rightarrow \infty} \langle \psi_j, \phi\rangle  \notag\\
		&=\langle u, \phi\rangle , \label{eq: D'2}
	\end{align}
	where the second equality follows by the definition of the inner product on $\mathcal{H}_{\lambda, \sigma}$, the third by self-adjointness, and the fourth because the fractional Sobolev embedding of Theorem \ref{thm:fracPoincare} implies $\psi_j \rightarrow u$ in $L^q$ when $\psi_j \rightarrow u$ in $\mathcal{H}_{\lambda, \sigma}$. Combining~\eqref{eq: D'1} and~\eqref{eq: D'2} yields~\eqref{eqtestL1loc} and the proof is complete.
\end{proof}	
				
	\begin{proposition}\label{prop: u bdd}
		Suppose $\sigma \in (0,1)$ and $1<\gamma<\frac{n+2 \sigma}{n-2 \sigma}$, and let $u$ be a nonnegative weak solution to~\eqref{fracequation}. Then $u \in L^{\infty}$.
	\end{proposition}
\begin{proof}
 By Proposition~\ref{uL1loc},
	\[
	u=u^\gamma\ast (k_\sigma^0 +k_\sigma^{\infty}),
	\]
	 where $u^\gamma \in L^Q, \frac{2}{\gamma}<Q<\frac{2 n}{\gamma(n-2 \sigma)}$. Then $u^\gamma * k_\sigma^{\infty} \in L^{\infty}$ by Young's inequality {(recall $k_{\sigma}^{\infty}\in L^{p}, p\in(2, \infty]$ by Corollary~\ref{corkaest})}, and
	\begin{align*}
		u & =u^\gamma\ast k_\sigma^{\infty}+u^\gamma \ast k_\sigma^0 \\
		& ={u^\gamma\ast k_\sigma^{\infty}} + \left(u^\gamma * k_\sigma^{\infty}+ u^\gamma * k_\sigma^0\right)^\gamma * k_\sigma^0 \\
		& \lesssim {u^\gamma * k_\sigma^{\infty}}+\left(u^\gamma * k_\sigma^{\infty}\right)^\gamma * k_\sigma^0+\left(u^\gamma * k_\sigma ^0\right)^\gamma * k_\sigma^0.
	\end{align*}
	Notice that $u^\gamma * k_\sigma^{\infty} \in L^{\infty}$ also implies that the middle term $\left(u^\gamma * k_\sigma^{\infty}\right)^\gamma * k_\sigma^0 \in L^{\infty}$ since $k_{\lambda,\sigma}^0\in L^1$.
	
Let us denote, to simplify the notation, 
	\[
	L^{(a,b)}:=\bigcap_{s\in (a,b)}L^s
	\]
as sets. Iteratively, it is enough to show that, given
	$$
	Tu=u^\gamma * k_\sigma^0
	$$
	with $u \in L^{\left(2, \frac{2 n}{n-2 \sigma}\right)}$, then there exists $N$ such that $T^N u \in L^{\infty}.$
	
	By Young's inequality and Corollary~\ref{corkaest}, 
	$$T u \in L^{\left(\frac{2}{\gamma}, \frac{2 n}{\gamma(n-2 \sigma)}\right)} * L^{\left[1, \frac{n}{n-2 \sigma}\right)} \subseteq L^r, \quad  
	\frac{\gamma n-2 \gamma\sigma-4 \sigma}{2 n}<\frac{1}{r}<\frac{\gamma}{2}.
	$$
If $\gamma n-2 \gamma\sigma-4 \sigma\leq 0$, then we are done. Otherwise, suppose  $\gamma n-2 \gamma\sigma-4 \sigma> 0$. Then  $T^2 u=\left(u^\gamma * k_\sigma^{0}\right)^\gamma * k_\sigma^{0}=(T u)^\gamma * k_\sigma^0$ where 
\[
\left(u^\gamma \ast k_\sigma^{0}\right)^\gamma=(T u)^\gamma \in L^{\left(\frac{2}{\gamma^{2}}, \frac{2 n}{\gamma^2n-2 \gamma^2\sigma-4 \sigma \gamma}\right)}.
\]
	Hence
	$$T^2 u \in L^{\left(\frac{2}{\gamma^{2}}, \frac{2 n}{\gamma^2n-2 \gamma^2\sigma-4 \sigma \gamma}\right)} * L^{\left[1, \frac{n}{n-2 \sigma}\right)} \subseteq L^r,\quad \frac{\gamma^2n-2 \gamma^2\sigma-4 \sigma \gamma-4\sigma}{2 n}<\frac{1}{r}<\frac{\gamma^2}{2}.$$
	By induction, we claim that 
	\begin{equation}\label{eq: induction}
		T^N u \in L^r,  \quad  \frac{\gamma^N(n-2 \sigma)-4 \sigma\left(\sum_{k=0}^{N-1} \gamma^k\right)}{2 n}<\frac{1}{r}<\frac{\gamma^N}{2}.
		\end{equation}
	Indeed, assume that for some $N\geq 3$
	$$
	T^{N-1} u \in L^t, \quad  \frac{\gamma^{N-1}(n-2 \sigma)-4 \sigma\left(\sum_{k=0}^{N-2} \gamma^k\right)}{2 n}<\frac{1}{t}<\frac{\gamma^{N-1}}{2}.
	$$
	Then $\left(T^{N-1} u\right)^\gamma \in L^{s}$, where $s=\frac{t}{\gamma}$ is such that 
	$$
	\frac{\gamma^{N}(n-2 \sigma)-4 \sigma\left(\sum_{k=1}^{N-1} \gamma^k\right)}{2 n}<\frac{1}{s}<\frac{\gamma^N}{2}
	$$
Then $T^N u=(T^{N-1}u)^{\gamma}\ast k_{\sigma}^0\in L^r$, where 
$$\frac{\gamma^{N}(n-2 \sigma)-4 \sigma\left(\sum_{k=1}^{N-1} \gamma^k\right)}{2 n}+\frac{n-2\sigma}{n}-1<\frac{1}{{r}}<\frac{\gamma^N}{2}+1-1,$$
	which in turn proves the induction claim~\eqref{eq: induction}.

	Therefore, to conclude, it is enough to prove that there exists $N$ such that
	$$
	\Gamma=\gamma^N(n-2 \sigma)-4 \sigma\bigg(\sum_{k=0}^{N-1} \gamma^k\bigg)<0.
	$$
Since $1<\gamma<\frac{n+2 \sigma}{n-2 \sigma} = 1+ \frac{4\sigma}{n-2\sigma}$, there is $\varepsilon \in (0,4\sigma)$ such that $\gamma=1+\frac{\varepsilon}{n-2 \sigma}$. Then $\sum_{k=0}^{N - 1} \gamma^k=\frac{\gamma^N-1}{\gamma-1}=\frac{(n-2 \sigma)}{\varepsilon}\left(\gamma^{N}-1\right)$
	whence
	\begin{align*}
	\Gamma&=\gamma^N(n-2 \sigma)-\frac{4 \sigma}{\varepsilon}(n-2 \sigma)\left(\gamma^N-1\right) =\frac{n-2 \sigma}{\varepsilon} 4\sigma\left(1- \gamma^N(1-\tfrac{\varepsilon}{4\sigma}) \right).
	\end{align*}
It remains to observe that that there exists $N$ such that
	$$
\quad 1-\Big(1+\frac{\varepsilon}{n-2 \sigma}\Big)^N\Big(1-\frac{\varepsilon}{4 \sigma}\Big)<0,
	$$
	namely $$\Big(1+\frac{\varepsilon}{n-2 \sigma}\Big)^N\Big(1-\frac{\varepsilon}{4 \sigma}\Big)>1,$$
since $1-\frac{\varepsilon}{4 \sigma}>0$ as $\varepsilon<4 \sigma$ and $\left(1+\frac{\varepsilon}{n-2 \sigma}\right)^N \rightarrow+\infty$ as $N \rightarrow \infty$.
\end{proof}
There are some interesting implications of the fact that $u\in L^{\infty}$, where $u$ is a nonnegative weak solution to~\eqref{fracequation}, among which the fact that~\eqref{fracequation} ``self-improves'' the regularity of its solutions. This is what we shall discuss now.

\subsection{Regularity of solutions}\label{subsection: regularity MS} Our aim here is to prove the regularity of the nonnegative weak solutions to~\eqref{fracequation} claimed in Theorem~\ref{propeqlambda}. Motivated by this and by~\cite{MS}, and inspired by~\cite[Lemma 4.4]{CS14}, we show first that if $\sigma>1/2$ then their gradient is bounded. 

\begin{proposition}\label{prop:gradbounded}
Suppose $0\leq \lambda\leq  \lambda_{0}$, $\sigma \in \big(\tfrac{1}{2},1\big)$ and $1<\gamma<\frac{n+2 \sigma}{n-2 \sigma}$. If $u$ is a nonnegative weak solution to~\eqref{fracequation}, then $\nabla u\in L^{\infty}$.
\end{proposition}

\begin{proof}
Using that 
\[
\nabla u=\nabla (u^{\gamma}*k_{\lambda, \sigma})=u^{\gamma}*\nabla k_{\lambda,\sigma},
\]
the first equality in~\eqref{dominance}, which implies 
\[
|\nabla k_{\lambda, \sigma}|\leq \int_0^{\infty}\e^{\lambda^{\sigma}t}\, |\nabla P_t^{\sigma}|\, \dd t \leq \int_0^{\infty}\e^{\lambda_0^{\sigma}t}\, |\nabla P_t^{\sigma}|\, \dd t,
\]
and the radiality of $k_{\sigma} = k_{\lambda_{0},\sigma}$, it is sufficient to obtain upper bounds for   
 	\begin{equation}\label{eq: der fract heat}
			\int_{0}^{\infty} \e^{t\lambda_0^{\sigma}}\left|\frac{\dd}{\dd r}P^{\sigma}_{t}(r)\right|\, \dd t.
		\end{equation}
	Let us stress that we do not aim for sharp bounds in what follows: we claim that for all $\sigma \in (0,1)$, it holds
	\begin{equation}\label{finalclaim}
		\left|\frac{\dd}{\dd r}k_{\sigma}(r)\right|\lesssim r^{-n+2\sigma-1}, \quad 0<r<1, \qquad \left|\frac{\dd}{\dd r}k_{\sigma}(r)\right|\lesssim(1+r)^K\, \e^{-\frac{n-1}{2}r}, \quad r\geq 1
		\end{equation}
		for some positive constant $K>0$. Let us assume the claim for a moment. If $\sigma>1/2$ then
	\[
	\nabla k_{\sigma} \in L^1(B(o,1)), \qquad \nabla k_{\sigma}\in L^p(\mathbb{H}^n \setminus B(o,1)), \quad  \forall p\in (2, \infty].
	\]
	Therefore, since $u^{\gamma} \in L^q, \, q\in (2/\gamma, \infty]$ with $2/\gamma<2$, it follows by Young's inequality that 
	\[
	\nabla u=u^{\gamma}\ast \nabla k_{\sigma} \in L^{\infty}.
	\]
Therefore, it remains to prove the claim~\eqref{finalclaim}.

In view of the subordination formula to the standard heat kernel, we have 
		\[
		\left|\frac{\dd}{\dd r}P^{\sigma}_t(r)\right|\leq \int_{0}^{\infty} \left|\frac{\dd}{\dd r}h_{s}(r)\right|\, \eta_t^{\sigma}(s)\, \dd s,
		\]
		where it is known that \cite{ADY96}
		\[
		\left|\frac{\dd}{\dd r}h_{s}(r)\right|\lesssim h_s(r)\times \begin{cases}
			1+	\frac{1}{\sqrt{s}}+\frac{r}{s}, \quad &\text{if } r<1,\\
			r\left(1+\frac{1}{s}\right)\quad &\text{if } r\geq 1.
		\end{cases}
		\]
		
	\emph{Case I.} $r\geq 1$. Then 
		\begin{align}
			\left|\frac{\dd}{\dd r}P^{\sigma}_t(r)\right|&\lesssim r \int_{0}^{1}s^{-1}\, h_s(r)\,\eta_t^{\sigma}(s)\, \dd s+ r \int_{1}^{\infty} h_s(r)\,\eta_t^{\sigma}(s)\, \dd s \notag  \\
			&\lesssim r \, J(t,r)\, + r\, P_t^{\sigma}(r), \label{eq: der fract J}
		\end{align}
		 We can control $J(t,r)$ using the Laplace method, since $t+r\geq 1$: indeed, notice that the exponential terms which occur by the estimates of $h_s(r)$ and $\eta_t(s)$ are the same as those in the subordination formula $P_t^{\sigma}(r)=\int_{0}^{\infty} h_s(r)\,\eta_t^{\sigma}(s)\, \dd s$. Hence, working as in \cite[p.93]{GS04}, we get
		\[
		J(t,r)\lesssim (t+r)^{-1} \, P_t^{\sigma}(r), \quad r \geq 1.
		\]
		Therefore, we get $\left| \frac{\dd}{\dd r} P_t^{\sigma}(r)\right| \lesssim r P_t^{\sigma}(r)$ so 
		\[
		\left| \frac{\dd}{\dd r} k_{\sigma}(r)\right| \lesssim r\int_{0}^{\infty}\e^{\lambda_0^{\sigma}t}\, P_t^{\sigma}(r)\, \dd t\lesssim r\, k_{\sigma}(r) \lesssim (1+r)^K \, \e^{-\frac{n-1}{2}r}, \quad r\geq1,
		\]
		for some $K>0$, by known estimates of $k_{\sigma}$ at infinity, see Proposition~\ref{prop: ka est}.
				
		\emph{Case II.} $r< 1$. We have
		\[
		1+\frac{1}{\sqrt{s}}+\frac{r}{s}=\frac{s+\sqrt{s}+r}{\tau} \asymp \begin{cases}
			r\, s^{-1}, &\quad s< r^2\\
			s^{-\frac{1}{2}}, &\quad r^2<s< 1\\
			1, &\quad s\geq 1. \\
		\end{cases}
		\]
		Therefore, 
		\begin{equation}\label{eq: small space}
		\left| \frac{\dd}{\dd r}P_t^{\sigma}(r)\right| \lesssim \left\{ \int_{0}^{r^2}r\, s^{-1} +\int_{r^2}^{1}s^{-\frac{1}{2}}+\int_{1}^{\infty} 1 \right\}h_s(r)\, \eta_t^{\sigma}(s) \, \dd s 
		\end{equation}
	If $t+r\geq 1$ (so $t$ is large) then we can argue as before, via the Laplace method, which would still give 
	\[
	\left| \frac{\dd}{\dd r}P_t^{\sigma}(r)\right|\lesssim P_t^{\sigma}(r), \quad r<1, \quad t\geq 1.
	\]
	On the other hand, if $t+r<1$ then~\eqref{eq: small space}	yields
	\begin{align*}
	\left| \frac{\dd}{\dd r}P_t^{\sigma}(r)\right|\lesssim \int_{0}^{r^2}r\, s^{-\frac{n}{2}}\, \e^{-\frac{r^2}{4s}}\, \eta_t^{\sigma}(s) \, \frac{\dd s}{s} +\int_{r^2}^{1}s^{-\frac{n}{2}}  \e^{-\frac{r^2}{4s}}\, \eta_t^{\sigma}(s) \, \frac{\dd s}{s^{\frac{1}{2}}} +P_t^{\sigma}(r),
		\end{align*}
		where we used for the first two integrals that $h_s(r)\asymp s^{-\frac{n}{2}}\, \e^{-\frac{r^2}{4s}}$ for $s$ and $r$ small. But for all $\kappa\geq 0$,
		\[
		\int_{0}^{\infty} s^{-\frac{n}{2}}\, \e^{-\frac{r^2}{4s}}\, \eta_t^{\sigma}(s) \, \frac{\dd s}{s^{\kappa}}\asymp t\, (t^{\frac{1}{2\sigma}}+r)^{-n-2\sigma-2\kappa}, \quad t+r<1,
		\]
		see e.g. \cite[Lemma 9]{P24}. Hence, for $t+r<1$, we get
		\begin{align*}
			\left| \frac{\dd}{\dd r}P_t^{\sigma}(r)\right|&\lesssim t\,r\,  (t^{\frac{1}{2\sigma}}+r)^{-n-2\sigma-2}+t\, (t^{\frac{1}{2\sigma}}+r)^{-n-2\sigma-1}+t\, (t^{\frac{1}{2\sigma}}+r)^{-n-2\sigma} \\
			&\lesssim  t\, (t^{\frac{1}{2\sigma}}+r)^{-n-2\sigma-1}.
		\end{align*}
		Altogether, substituting in~\eqref{eq: der fract J}, we have the following upper bounds for the radial derivatives of $P_t^{\sigma}(r)$ around the origin:
		\[
		\left|\frac{\dd}{\dd r}P^{\sigma}_t(r)\right|\lesssim \begin{cases}
			 t \, (t^{\frac{1}{2\sigma}}+r)^{-(n+2\sigma+1)},\quad &\text{if } r+t<1, \\
			 P_t^{\sigma}(r),\quad &\text{if } r<1, \quad t\geq 1.
		\end{cases}  
		\]
		
		We can finally return to radial derivatives of $k_{\sigma}(r)$ around the origin in view of  $\eqref{eq: der fract heat}$.  Then
		\begin{align*}
			\left|\frac{\dd}{\dd r}k_{\sigma}(r)\right|&\lesssim  \int_{0}^{1} t \, (t^{\frac{1}{2\sigma}}+r)^{-(n+2\sigma+1)}\, \dd t \, +\, \int_{1}^{\infty} \e^{t\lambda_0^{\sigma}}P^{\sigma}_{t}(r)\, \dd t \\
			&\lesssim r^{-n-2\sigma-1}\int_{0}^{r^{2\sigma}} t\,\dd t+\int_{r^{2\sigma}}^{1} t^{-\frac{n+1}{2\sigma}} \, \dd t + k_{\sigma}^{0}(r)\\
			&\lesssim r^{-n+2\sigma-1}, \quad r<1,
		\end{align*}
		since $k_{\sigma}^{0}(r)\asymp r^{-n+2\sigma}$.  The proof of the claim~\eqref{finalclaim} is now complete.
\end{proof}
We are now ready to complete the proof of Theorem~\ref{propeqlambda} via the following.

\begin{proposition}
Suppose $0\leq \lambda \leq \lambda_{0}$, $\sigma \in (0,1)$ and $1<\gamma<\frac{n+2 \sigma}{n-2 \sigma}$, and let $u$ be a nonnegative weak solution to~\eqref{fracequation}.  Then $u\in \mathcal{C}^{0,\alpha}$ if $\sigma< 1/2$, while $u\in \mathcal{C}^{1,\alpha}$ if $\sigma\geq 1/2$, for all $\alpha \in (0,1)$. Moreover, $u\in L^{q}$ for all $q\in (2,\infty]$. 
\end{proposition}

\begin{proof}
Observe that by Lemma~\ref{lem:equiv} and since $u\in L^{\infty} \subseteq L_{\sigma}$, the function $u$ is also a solution to~\eqref{fracequation} in the sense of distributions. Since $f(u) := \lambda^{\sigma}u+u^{\gamma} \in L^{\infty}$, the distributional fractional Laplacian $\Delta^{\sigma}u$ is also in $L^{\infty}$. Therefore, by~\cite[Proposition 2.6(c)]{BanEtAl},
\begin{itemize}
	\item[(i)] if $0<\sigma\leq 1/2$, then $u\in \mathcal{C}^{0,\alpha}$ for any $0<\alpha<2\sigma$;
	\item[(ii)] if $1/2<\sigma\leq 1$, then $u\in \mathcal{C}^{1,\alpha}$ for any $0<\alpha<2\sigma-1$.
\end{itemize}
If $0< \sigma < 1/2$, then in turn $f(u)\in \mathcal{C}^{0,\alpha}$ for $0<\alpha<2\sigma$, whence $\Delta^{\sigma}u$ is also in $\mathcal{C}^{0,\alpha}$ for $0<\alpha<2\sigma$. By choosing $\alpha$ sufficiently small so that $\alpha + 2\sigma\leq 1$, by iterating~\cite[Proposition 2.6(b)]{BanEtAl} we have $u\in \mathcal{C}^{0,\alpha}$ for all $\alpha \in (0,1)$.

Similarly, if $\sigma=1/2$  then $\Delta^{\sigma}u$ is in $\mathcal{C}^{0,\alpha}$ for all $\alpha \in (0,1)$ by (i), whence $u\in \mathcal{C}^{1,\alpha }$ for all $\alpha \in (0,1)$ by~\cite[Proposition 2.6(b)]{BanEtAl}.

Suppose now  that $\sigma> 1/2$. Then $u\in \mathcal{C}^{0,1}$ since $\nabla u$ is bounded by Proposition~\ref{prop:gradbounded}, and since $u$ is also bounded, $u\in \mathcal{C}^{0,\alpha}$ for all $0<\alpha<1$. Moreover, $f(u)$ has the same H{\"o}lder continuity as $u$ since by~\eqref{fbfa}
		\[
|u(x)^{\gamma}-u(y)^{\gamma}|\leq C \|u\|_{\infty}^{\gamma-1}|u(x)-u(y)|, \qquad x,y\in \mathbb{H}^{n}.
		\]
		Then $\Delta^{\sigma} u\in \mathcal{C}^{0,\alpha}$ for all $\alpha \in (0,1]$. Since $\alpha + 2\sigma >1$ for all $\alpha>0$, by~\cite[Proposition 2.6(b)]{BanEtAl} we obtain that $u\in  \mathcal{C}^{1,\alpha}$ for all $\alpha \in (0,1]$.
		
		 Finally, combining the fact that $u\in L^{\infty}$ with $u\in L^q$, $q\in (2, \frac{2n}{n-2\sigma})$ by Theorem~\ref{thm:fracPoincare2}, we get that $u\in L^q$ for $q\in (2,\infty)$.
\end{proof}

	\section{Riemannian symmetric spaces of non-compact type}\label{Sec:10}

In this section we discuss how all our results extend to \emph{rank one} symmetric spaces while Theorem~\ref{thm:fracPoincare}, actually holds on Riemannian symmetric spaces of non-compact type.

\subsection{Symmetric spaces} We adopt the standard notation and refer to \cite{Hel} for more details. Let $G$ be a semi-simple Lie group, connected, noncompact, with finite center, and $K$ be a maximal compact subgroup
of $G$. The homogeneous space $\mathbb{X} = G/K$ is a Riemannian symmetric space of noncompact
type. Let $\mathfrak{g} = \mathfrak{k} \oplus \mathfrak{p}$ be the Cartan decomposition of the Lie algebra of $G$. The Killing form
of $\mathfrak{g}$ induces a $K$-invariant inner product on $\mathfrak{g}$, hence a $G$-invariant Riemannian metric
on $G/K$. Fix a maximal abelian subspace $\mathfrak{a}$ in $\mathfrak{p}$. We identify $\mathfrak{a}$ with its dual $\mathfrak{a}^{*}$ by means of
the inner product inherited from $\mathfrak{p}$. Let $\Sigma \subseteq \mathfrak{a}$ be the root system of $(\mathfrak{g}, \mathfrak{a})$ and let $W$ be the associated Weyl group. Choose a
positive Weyl chamber $\mathfrak{a}^{+}\subseteq \mathfrak{a}$ and let $\Sigma^{+}\subseteq \Sigma$ be the corresponding subsystem of positive roots.
Denote by $\rho=\frac{1}{2}\sum_{\alpha \in \Sigma^{+}}m_{\alpha}\alpha$ the half sum of positive roots counted with their multiplicities. The bottom of the spectrum $\lambda_0$ of the nonnegative Laplace-Beltrami operator $\Delta$ is equal to $|\rho|^2$. Let $n$ be the dimension and $\nu$ be the pseudo-dimension 
(or dimension at infinity) of $\mathbb{X}$. Notice that the only space with $n=2$ is the hyperbolic plane, and that $\nu=3$ on real hyperbolic space; in general, we always have $\nu\geq 3$. Finally, $\ell:=\dim\mathfrak{a}$ is called the rank of $G/K$.

We have the decompositions 
\begin{align*}
	\begin{cases}
		\,G\,=\,N\,(\exp\mathfrak{a})\,K 
		\qquad&\textnormal{(Iwasawa)}, \\[5pt]
		\,G\,=\,K\,(\exp\overline{\mathfrak{a}^{+}})\,K
		\qquad&\textnormal{(Cartan)}.
	\end{cases}
\end{align*}
Denote by $A(x)\in\mathfrak{a}$ and $x^{+}\in\overline{\mathfrak{a}^{+}}$
the middle components of $x\in{G}$ in these two decompositions, and by
$|x|=|x^{+}|$ the distance to the origin $o=\{K\}$. Then the Haar measure 
on $G$ writes
\begin{align*}
	\int_{G}f(x)\,\diff{x} 
	=\,
	|K/{\mathbb{M}}|\,\int_{K}\,
	\int_{\mathfrak{a}^{+}}\delta(x^{+})\, 
	\int_{K} f(k_{1}(\exp x^{+})k_{2}) \, \diff{k_2}\,  \diff{x^{+}} \, \diff{k_1}, 
\end{align*}
with density
$$ \delta(x^{+})\,
	=\,\prod_{\alpha\in\Sigma^{+}}\,
	(\sinh\langle{\alpha,x^{+}}\rangle)^{m_{\alpha}}\leq \e^{2\langle\rho, x^{+} \rangle}.$$
Here $K$ is equipped with its normalized Haar measure, $\mathbb{M}$ denotes the centralizer of $\exp\mathfrak{a}$ in $K$ and the volume 
of $K/{\mathbb{M}}$ can be computed explicitly, see \cite[Eq (2.2.4)]{AJ99}.
It follows that locally, we have in terms of the Riemannian measure that $\mu(B(o, s))\asymp s^n, 0<s<1$.

\subsection{Fourier analysis}
The spherical Fourier transform (Harish-Chandra transform)
$\mathcal{H}$ is defined by
\[
	\mathcal{H}f(\xi)
	=\int_{G}\varphi_{-\xi}(x)\,f(x) \, \dd x
	\qquad\forall\,\xi\in\mathfrak{a},\
	\forall\,f\in\mathcal{C}_c^{\infty}(K\backslash{G/K}),
\]
where $\varphi_{\xi}\in\mathcal{C}^{\infty}(K\backslash{G/K})$ is the
spherical function of index $\xi \in \mathfrak{a}$, with the integral representation
\begin{align*}
	\varphi_{\xi}(x)\, 
	=\,\int_{K}\e^{\langle{i\xi+\rho,\,A(kx)}\rangle}\, \diff{k}.
\end{align*} 
Notice that for all $\xi\in \mathfrak{a}$ we have
\begin{equation}\label{eq:Xphi0}
|\varphi_{\xi}(\exp x^{+})|\leq \varphi_{0}(\exp x^{+}) \asymp\,
\Big\lbrace \prod_{\alpha\in\Sigma_{r}^{+}} 
1+\langle\alpha,x^{+}\rangle\Big\rbrace\,
\e^{-\langle\rho, x^{+}\rangle}
\qquad\forall\,x^{+}\in\overline{\mathfrak{a}^{+}}.
\end{equation}

More generally, we have the Helgason-Fourier transform for test functions,
\[
	\widehat{f}(\xi, k\mathbb{M})\,
	=\,
	\int_{G}\,
	f(gK)\,\e^{\langle{-i\xi+\rho,\,A(k^{-1}g)}\rangle} \,\diff{g}
\]
which boils down to the spherical transform 
when $f$ is bi-$K$-invariant. The inversion formula is given by
\[
f(gK)=|W|^{-1}\int_K\int_{\mathfrak{a}}\widehat{f}(\xi, k\mathbb{M})\,\e^{\langle{i\xi+\rho,\,A(k^{-1}g)}\rangle} \,\frac{\diff{\xi}}{|\textbf{c}(\xi)|^2}\, \diff{k}, \quad f\in \mathcal{C}_c^{\infty}(G/K).
\] 
Here $|\textbf{c}(\xi)|^{-2}$ is the Plancherel density, with behavior
\[
|\textbf{c}(\xi)|^{-2}\leq C\, |\xi|^{\nu-\ell} \, (1+|\xi|)^{n-\nu}.
\]

The fractional Poincar{\'e} inequality is then valid on any symmetric space:
\begin{theorem}
Suppose $\sigma\in (0,1)$ and $2<q\leq\frac{2n}{n-2\sigma}$. Then on any Riemannian symmetric space of the non-compact type $\mathbb{X}$, there exists $C>0$ such that
	\begin{equation*}
		\int_{\mathbb{X}} (|\Delta^{\sigma/2} u|^{2}-\lambda_{0}^{\sigma}u^{2})\, \dd \mu \geq C \|u\|_{q}^{2}, \quad u\in \mathcal{C}_c^{\infty}.
	\end{equation*}
\end{theorem}
\begin{proof}
	We only sketch parts of the proof. First of all, notice that  the discussion on positive and negative powers of the shifted fractional Laplacian, carried out for hyperbolic space, is valid for arbitrary rank symmetric spaces. More precisely, if $u\in \mathcal{C}_c^{\infty}$ then  $(\Delta^{\sigma}u-\lambda_0^{\sigma}u)^{\frac{\alpha}{2}}\in L^2$ for all $\alpha>0$, while $(\Delta^{\sigma}u-\lambda_0^{\sigma}u)^{-\frac{\alpha}{2}}\in L^2$ for all $0<\alpha<\nu/2$; for the second case, we justify following the hyperbolic space case arguments and observing that
	\begin{align*}
	\int_{K} \int_{\mathfrak{a}\cap B(0,1)}((|\xi|^2+\lambda_{0})^{\sigma}-\lambda_{0}^{\sigma})^{-\alpha}\,|\widehat{u}(\xi, k\mathbb{M})|^2\frac{\dd \xi}{|\textbf{c}(\xi)|^2}\, \dd k &\lesssim \int_{\mathfrak{a}\cap B(0,1)}|\xi|^{2\alpha} \, |\xi|^{\nu-\ell}\, \dd \xi\\
	&\lesssim \int_{0}^{1} s^{2\alpha+\nu-1}\, \dd s,
	\end{align*}
which is finite for $\alpha <\nu/2$. Therefore, as in the hyperbolic space case, it suffices to control
\[
k_{\sigma}(x)=\int_{0}^{\infty}\e^{\lambda_{0}^{\sigma}t}\, P_t^{\sigma}(x)\, \dd t.
\]
The arguments for the local part $k_{\sigma}^{0}$ run the same way as for $\mathbb{H}^{n}$. For the part at infinity one can show that $$k_{\sigma}^{\infty}(x)\lesssim (1+|x|)^{K}\varphi_{0}(x),\quad  |x|>1,$$ 
working similarly.  Let $\kappa$ be a sufficiently regular bi-$K$-invariant function on $G$. Then the convolution operator with kernel $\kappa$ 
\[
\| \ast \kappa \|_{L^{q'}\rightarrow L^q} \leq \bigg(\int_{G} \varphi_{0}(x)\, |\kappa(x)|^{\frac{q}{2}} \dd x\bigg)^{\frac{2}{q}}, \qquad \forall q\in [2, +\infty).
\] 
Notice that this result has been proved in several contexts. For $q = 2$, it is the
so-called Herz's criterion, see for instance \cite{Cow97}. For $q > 2$, the proof carried
out on Damek--Ricci spaces \cite[Theorem 4.2]{APV11} adapts straightforwardly in
the higher rank case. Hence, considering the bi-$K$-invariant kernel $$\kappa(x):=(1+|x|)^K \, \varphi_{0}(x)$$
 it suffices to check that $\varphi_0\,\kappa^{q/2}\in L^1$ for $q>2$. Taking into account \eqref{eq:Xphi0} for the ground spherical function and the Cartan decomposition, we see that for any $M>0$, the integral
\begin{align*}
	\int_{\mathfrak{a}}  (1+|x^{+}|)^{M} \, \e^{-\frac{q-2}{2}\langle \rho, x^{+}\rangle} \dd x^{+} \lesssim \int_{0}^{\infty} \, (1+r)^{M+\ell-1} \e^{-\frac{q-2}{2} \rho_{\min}r}\, \dd r
\end{align*}
is finite, since 
$\rho_{\min} :=\min_{x^{+}\in\mathfrak{a}^{+},  |x^{+}|=1}  \langle \rho, x^{+} \rangle \in(0, |\rho|]$.
We omit further details. 
\end{proof}
We also have the following Fujita-type results on arbitrary Riemannian noncompact symmetric spaces in the non-critical regime of parameters; one follows essentially the same arguments by relying on pointwise estimates \cite{GS04} and norm estimates \cite{CGMII} of the fractional heat kernel (see Proposition \ref{prop:oldFujita}).
\begin{theorem}
Let $\mathbb{X}$ be a Riemannian noncompact symmetric space of pseudo-dimension $\nu$. Suppose $h(t)=\e^{\beta t}$ for some $\beta>0$ and $\sigma \in (0,1)$. Define $\gamma^{*} = 1 + \frac{\beta}{\lambda_{0}^{\sigma}}$. Then the following holds.
	\begin{itemize}
		\item If $1<\gamma< \gamma^{*}$, then every nontrivial nonnegative mild solution to the problem~\eqref{fracheat} blows up in finite time.
		\item If $\gamma> \gamma^{*}$, or if $\gamma= \gamma^{*}$ and $\beta>\frac{2}{\nu}\lambda_{0}^{\sigma}$, then there exists a nontrivial nonnegative mild global solution to the problem~\eqref{fracheat} for sufficiently small initial data $f$.
	\end{itemize}
\end{theorem}

On the other hand, all our results extend to \emph{rank one} symmetric spaces (including real hyperbolic spaces). Indeed, if $\phi$ is a test function on any rank one symmetric space $\mathbb{X}$, then for all $\sigma\in (0,1)$
\[
\Delta^{\sigma} \phi(x)=\frac{1}{|\Gamma(-\sigma)|}\, \mathrm{p.v.}\int_{\mathbb{X}} (\phi(x)- \phi(y))\, P_0^{\sigma}(d(x,y))\,\dd \mu(y),
\]
owing to the fact that $P_0^{\sigma}=\int_{0}^{\infty}h_t\, \frac{\dd t}{t^{1+\sigma}}$, as subordinate to the heat kernel, is a radial function in the rank one case, and so is the Jacobian of $\exp_x$ at any $x\in \mathbb{X}$. This allows to define fractional powers of the Laplacian for less regular functions with certain control at infinity, namely the class $L_{\sigma}$. All arguments used thereafter are valid for rank one symmetric spaces (e.g. the expansion due to \cite{Ion} of spherical functions $\varphi_{\xi}(r)$ for large $r$), with obvious minor modifications. We omit the details.

\begingroup
  \renewcommand{\addcontentsline}[3]{} 
  \section*{Declarations}
\endgroup
\subsection*{Data Availability} No data, models, or code were generated or used during the study.

\subsection*{Conflicts of Interest} The authors have no conflicts of interest to declare.

	\end{document}